\documentclass[12pt]{amsart}

\usepackage{amsmath}
\usepackage{amssymb}
\usepackage{bbm}
\usepackage[colorlinks,citecolor=red,pagebackref,hypertexnames=false]{hyperref}
\usepackage{esint} 
\usepackage{enumitem} 
\usepackage{verbatim} 
\usepackage[dvipsnames]{xcolor}

\newcommand{\R}{{\mathbb R}}       
\newcommand{\N}{{\mathbb N}}

\newcommand{\DD}{{\mathcal D}}

\newcommand{\HH}{{\mathcal H}}

\newcommand{\EE}{{\mathcal E}}

\newcommand{\diam}{{\rm diam}}
\newcommand{\dist}{{\rm dist}}

\newcommand{\rf}[1]{{(\ref{#1})}}

\newcommand{\supp}{\operatorname{supp}}

\newcommand{\vphi}{{\varphi}}
\newcommand{\ve}{{\varepsilon}}
\newcommand{\vv}{{\vspace{2mm}}}

\newcommand{\wt}[1]{{\widetilde{#1}}}
\newcommand{\wh}[1]{{\widehat{#1}}}

\newcommand{\G}{{\mathsf G}}
\newcommand{\VG}{{\mathsf {VG}}}

\newcommand{\LD}{{\mathsf{LD}}}
\newcommand{\BA}{{\mathsf{BA}}}

\newcommand{\pStop}{\mathsf{W}_{G_A}}
\newcommand{\ppStop}{\mathsf{W}_{G_A,0}}

\newcommand{\Whit}{\mathsf{W}}

\newcommand{\Apsi}{\mathcal{A}_{\psi}}
\newcommand{\apsi}{\mathfrak{a}_{\psi}}
\newcommand{\Agapsi}{\mathcal{A}_{\Gamma,\psi}}

\newcommand{\A}{\mathcal{A}}
\newcommand{\lca}{\mathfrak{a}}

\def\Xint#1{\mathchoice
{\XXint\displaystyle\textstyle{#1}}%
{\XXint\textstyle\scriptstyle{#1}}%
{\XXint\scriptstyle\scriptscriptstyle{#1}}%
{\XXint\scriptscriptstyle\scriptscriptstyle{#1}}%
\!\int}
\def\XXint#1#2#3{{\setbox0=\hbox{$#1{#2#3}{\int}$ }
\vcenter{\hbox{$#2#3$ }}\kern-.58\wd0}}

\def\avint{\;\Xint-}

\newcommand{\chara}{\mathbbm{1}}
\newcommand{\ps}[1]{\left( #1 \right)}   \newcommand{\av}[1]{\left| #1 \right|}
\newcommand{\dr}{\frac{dr}{r}}

\newcommand{\vp}{\varphi}
\newcommand{\hi}{\mathcal{H}^1}

\newcommand{\ep}{\operatorname{ep}}

\newcommand{\Sect}{\operatorname{X}}
\newcommand{\kap}{\varkappa}
\newcommand\Hdist{\operatorname{dist}_H}
\newcommand\excess{\operatorname{excess}}

\usepackage{pgf,tikz} 
\usetikzlibrary{intersections}

\definecolor{ffffff}{rgb}{1.0,1.0,1.0}
\definecolor{qqqqff}{rgb}{0.0,0.0,1.0}
\definecolor{ffqqqq}{rgb}{1.0,0.0,0.0}
\definecolor{zzzzqq}{rgb}{0.6,0.6,0.0}
\definecolor{marronet}{rgb}{0.6,0.2,0}
\definecolor{negre}{rgb}{0,0,0}
\definecolor{vermell}{rgb}{0.8,0.05,0.05}
\definecolor{blau}{rgb}{0.3,0.2,1.}
\definecolor{blauclar}{rgb}{0.,0.,1.}
\definecolor{grisfosc}{rgb}{0.25098039215686274,0.25098039215686274,0.25098039215686274}
\definecolor{verd}{rgb}{0.1,0.6,0.1}
\definecolor{taronja}{rgb}{0.9,0.6,0.05}
\definecolor{vermellclar}{rgb}{1.,0.,0.}
\definecolor{verdet}{rgb}{0,0.8,0.1}
\definecolor{blauverd}{rgb}{0,0.4,0.2}
\definecolor{grisclar}{rgb}{0.6274509803921569,0.6274509803921569,0.6274509803921569}

\newtheorem{theorem}{Theorem}[section]
\newtheorem{lemma}[theorem]{Lemma}
\newtheorem{mlemma}[theorem]{Main Lemma}
\newtheorem{coro}[theorem]{Corollary}

\newtheorem{claim}{Claim}
\newtheorem{conjecture}[theorem]{Conjecture}
\newtheorem*{theorem*}{Theorem}

\newtheorem*{claim*}{Claim}
\newtheorem*{KP}{Key Property}

\theoremstyle{definition}
\newtheorem{definition}[theorem]{Definition}

\theoremstyle{remark}
\newtheorem{rem}[theorem]{\bf Remark}

\numberwithin{equation}{section}

\newcommand{\brem}{\begin{rem}}
\newcommand{\erem}{\end{rem}}

\begin{document}

\title{A proof of Carleson's $\ve^2$-conjecture}

\author{Benjamin Jaye}

\author{Xavier Tolsa}

\author{Michele Villa}

\address{Benjamin Jaye\\
School of Mathematical and Statistical Sciences, Clemson University, Clemson, South Carolina, 29631, USA}

\email{bjaye@clemson.edu}

\address{Xavier Tolsa \\
ICREA, Passeig Llu\'{\i}s Companys 23 08010 Barcelona, Catalonia\\  Departament de Matem\`atiques, and BGSMath
 \\
Universitat Aut\`onoma de Barcelona
\\
08193 Bellaterra (Barcelona), Catalonia.
}

\email{xtolsa@mat.uab.cat}

\address{Michele Villa \\
School of Mathematics, the University of Edinburgh, James Clerk Maxwell Building, Peter Guthrie Tait Rd, Edinburgh EH9 3FD, Scotland}

\email{s1246023@sms.ed.ac.uk}

\thanks{B. J. was partially supported by NSF through DMS-1800015 and the CAREER Award  DMS-1847301}
\thanks{X.T. was partially supported by MTM-2016-77635-P (MICINN, Spain) and 2017-SGR-395 (AGAUR, Catalonia).
}
\thanks{M.V. was supported by The Maxwell Institute Graduate School in Analysis and its
Applications, a Centre for Doctoral Training funded by the UK Engineering and Physical
Sciences Research Council (grant EP/L016508/01), the Scottish Funding Council, Heriot-Watt
University and the University of Edinburgh. Part of this work was carried out while M.V. was visiting X.T. at UAB}


\begin{abstract}
In this paper we provide a proof of the Carleson $\ve^2$-conjecture.  This result yields a characterization (up to exceptional sets of zero length) of the tangent points of a Jordan curve in terms of the finiteness of the associated Carleson $\ve^2$-square function.\end{abstract}

\maketitle

\tableofcontents
\section{Introduction}

Over the last thirty years,  beginning with the work of Jones \cite{Jo} and David-Semmes \cite{DS1}, there has been a great deal of activity concerning the study of geometric square functions which measure the regularity of sets through a multi-scale analysis. Usually, the motivation for the study of such geometric
square functions stems from the wish to solve different problems in complex analysis, harmonic analysis, or PDE's which depend on (variants of) quantitative rectifiability. For example, in the solution of the Vitushkin's conjecture for sets of finite length by David \cite{David-vitus}, Melnikov's curvature of measures and its connection with
 Jones' $\beta$-numbers play an essential role (see L\'eger \cite{Leger}). Analogously, the solution
of the David-Semmes problem concerning the $L^2$-boundedness of the codimension $1$ Riesz transform by Nazarov, Tolsa, and Volberg \cite{NToV} uses the so-called BAUP criterion for uniform rectifiability found
by David and Semmes (see \cite[p.\ 139]{DS2}). More recently, Naber and Valtorta \cite{Naber-Valtorta} have
extended the use of related techniques involving an $L^2$-variant of the Jones' $\beta$-numbers
to solve different questions in the area of free boundary problems,
and more precisely, on the singularities of minimizing harmonic maps.

 In this paper we solve a longstanding conjecture of Lennart Carleson concerning another geometric square function.
To formulate the problem we need to introduce some notation.  Let $\Omega^+$ be a proper open set in $\R^2$, and set $\Gamma =\partial\Omega^+$ and
$\Omega^- = \R^2\setminus \overline{\Omega^+}$.
For $x\in\R^2$ and $r>0$, denote by $I^+(x,r)$ and $I^-(x,r)$ the longest open arcs of the circumference
$\partial B(x,r)$ contained in $\Omega^+$ and $\Omega^-$, respectively (they may be empty). Then we define
$$\ve(x,r) = \frac1r\,\max\big(\big|\pi r- \HH^1(I^+(x,r))\big|,\, \big|\pi r- \HH^1(I^-(x,r))\big|\big).$$
Here $\HH^1$ denotes the one-dimensional Hausdorff measure.  The Carleson $\ve^2$-square function is given by
\begin{equation}\label{eqEE}
\EE(x)^2 :=\int_0^1 \ve(x,r)^2\,\frac{dr}r.
\end{equation}

If $\Gamma$ is a line then $\EE(x)=0$ for all $x\in \Gamma$.  Carleson conjectured that (\ref{eqEE}) encodes some regularity properties of $\Gamma$.

\begin{conjecture}[Carleson's $\ve^2$-conjecture]\label{carleps2}  Suppose $\Gamma$ is a Jordan curve.  Except for a set of zero $\mathcal{H}^1$-measure, $\Gamma$ has a tangent at $x\in \Gamma$ if and only if $\EE(x)<\infty$.
\end{conjecture}

See Section \ref{def:cones-tangents} below for the precise definition of a tangent.  This conjecture may be found in Bishop \cite[Conjecture 3]{Bishop-conjectures}, Bishop-Jones \cite{Bishop-Jones-schwartzian}, Garnett-Marshall \cite[p.220]{GM}, and David-Semmes \cite[Section 21]{DS1}, for example.

The `only if' direction of Conjecture \ref{carleps2} is a well-known result:  When $\Gamma$ is a Jordan curve, it follows from the Ahlfors distortion theorem and an argument of Beurling that
$$\EE(x)<\infty \quad\mbox{ for $\HH^1$-a.e. tangent point of $\Gamma$.}$$
See \cite[p.79]{Bishop-Jones-schwartzian}, for example. Therefore, the content of Conjecture \ref{carleps2} is that the converse statement should also hold true.


In this paper we prove Conjecture \ref{carleps2}.  We also show that an analogue of this result also holds for two-sided corkscrew open sets, which have a scale invariant topological assumption but are not necessarily connected,  see Section \ref{def:2sidedcorkscrew} below for the definition. Our precise result is the following theorem.

\begin{theorem}\label{teo-main}
Let $\Omega^+\subset \R^2$ be either a Jordan domain or a two-sided corkscrew open set, let $\Gamma=
\partial\Omega^+$, and let $\EE$ be the
associated square function defined in \rf{eqEE}. Then the set
$G=\{x\in\Gamma:\EE(x)<\infty\}$ is rectifiable and at $\HH^1$-almost every point of $G$ there exists a tangent
to $\Gamma$.
\end{theorem}

Recall here that a set $E\subset\R^2$ is called rectifiable if there are  Lipschitz maps
$f_i:\R\to\R^2$, $i\in \N$, such that
\begin{equation}\label{eq001}
\HH^1\Big(E\setminus\textstyle\bigcup_{i=1}^{\infty} f_i(\R)\Big) = 0.
\end{equation}

As a corollary, we get:

\begin{coro}\label{carlcoro}
Let  $\Gamma\subset\R^2$ be a Jordan curve or a two-sided corkscrew open set, and let $\EE$ be the
associated square function defined in \rf{eqEE}. Then, the set of tangent points of $\Gamma$
coincides with the points $x\in\Gamma$ such that $\EE(x)<\infty$, up to a set of zero measure $\HH^1$.
\end{coro}

One should compare Theorem \ref{teo-main} to an influential theorem of Bishop-Jones \cite{Bishop-Jones-schwartzian}, who obtained an analogue of the $\ve^2$-conjecture for the $\beta$-numbers introduced by Jones \cite{Jo}.  Define
 \begin{equation}\label{betadefn}\beta_{\infty, \Gamma}(B(x,r)) = \frac{1}{r}\inf_{\substack{L\text{ a line}\\L\cap B(x,r)\neq\varnothing}}\sup_{y\in \Gamma\cap B(x,r)}\dist(y,L).
 \end{equation}
 In other words, $\beta_{\infty,\Gamma}(B(x,r))$ is the infimum over $\beta>0$ so that $\Gamma\cap B(x,r)$ is contained in a strip of width $r\beta$.  Certainly then, we have that
 \begin{equation}\label{betaepscomparison}\ve(x,r')\lesssim \beta_{\infty,\Gamma}(B(x,r))\text{ whenever } r/2<r'<r,
 \end{equation}
whenever $x$ is a tangent point of $\Gamma$ and that $r$ is small enough.
 Bishop-Jones (Theorem 2 in \cite{Bishop-Jones-schwartzian}) proved that, up to a set of $\HH^1$-measure zero,  a Jordan curve $\Gamma$ has a tangent at $x\in \Gamma$ if and only if
\begin{equation}\label{betasquarefinite}
\int_0^{1}\beta_{\infty, \Gamma}(B(x,r))^2\frac{dr}{r}<+\infty.
\end{equation}

 The arguments used in \cite{Bishop-Jones-schwartzian} to show that
 \rf{betasquarefinite} holds for the tangent points  $x\in\Gamma$ are also valid for two-sided corkscrew open sets.
 Consequently, in view of (\ref{betaepscomparison}), this result completes the proof of Corollary \ref{carlcoro}  (in the case of Jordan domains, one can alternatively appeal to the aforementioned argument based on the Ahlfors distortion theorem).
 In the opposite direction, while the Bishop-Jones theorem is weaker than Theorem \ref{teo-main}, it has a beautifully concise proof.

 In the monograph \cite{DS1}, David and Semmes develop a quantitative analogue of rectifiability for Ahlfors regular sets, and prove characterizations of it in terms of square functions involving modifications of the $\beta$-numbers.   When Conjecture \ref{carleps2} is discussed on p. 141 of \cite{DS1}, it is described how the coefficients $\ve(x,r)$ are not sufficiently stable to apply the methods developed in \cite{DS1}, even if one only wishes to show the rectiability of an Ahlfors regular subset of $\{x\in \Gamma: \mathcal{E}(x)<\infty\}$.

 With this in mind, one can point to Main Lemma \ref{l:main1} below as one of the main technical innovations of this paper, which, roughly speaking provides some control of the numbers $\beta_{\infty,\Gamma}(B(x,r))$ at points where $\mathcal{E}(x)<\infty$.  This amount of control on the $\beta$-numbers is nowhere near strong enough to directly obtain (\ref{betasquarefinite}), but it is sufficient to be able to adapt a scheme originating in the work of David and Semmes \cite{DS1} and adapted to the non-homogeneous setting by L\'eger \cite{Leger} which enables us to prove Theorem \ref{teo-main}, see Main Lemma \ref{l:main2}.

  Crucial to the proofs of both of the Main Lemmas is the introduction of several smoother square functions (see Section \ref{smoothsquares}).  While being controlled by the Carleson square function, these smoother square functions behave in a considerably more stable manner in compactness arguments, and this additional stability enables us to obtain some basic geometric information about natural limit situations (namely, that the limiting ``curve'' should contain an analytic variety, see property (1) of Lemma \ref{lem:basicsquare}).   This basic information is then considerably refined by employing the  \emph{admissible pairs} property (Definition \ref{def:admissible}), which is obtained as a consequence of the finiteness of the Carleson square function itself (see, for instance, Lemmas \ref{lemline} and \ref{l:convjordan4}).  Additionally, of central importance to the proof of Main Lemma \ref{l:main2} is the fact that the smoother square functions control the Lipschitz constant of a compactly supported Lipshitz graph (see Section \ref{s:fourier}), and it is this property that provides the main mechanism required to adapt L\'eger's scheme.

 Of course, as a consequence of Theorem \ref{teo-main} and Theorem 2 of \cite{Bishop-Jones-schwartzian}, we have that, up to a set of $\HH^1$-measure zero, $\mathcal{E}(x)<\infty$ if and only if (\ref{betasquarefinite}) holds at $x\in \Gamma$.

\subsection*{Acknowlegement}  This work began at the workshop \emph{Harmonic Analysis in Nonhomogeneous Settings and Applications} at the University of Birmingham in June 2019.  Special thanks go to Mar\'{\i}a Carmen Reguera for organizing this workshop and inviting the authors.  We also thank Jonas Azzam for stimulating conversations about the project in its early stages.

\section{Preliminaries}

\subsection{Constants}  We will denote by $C, c>0$ absolute constants that may change from line to line.  We will often use the symbol $A \lesssim B$ to mean that $A\leq CB$.  The symbol $A\gtrsim B$ is just another way of writing $B\lesssim A$.  The symbol $A\approx B$ means that both $A\lesssim B$ and $B\lesssim A$.  If a constant is allowed to depend on a given parameter, the parameter dependence will be described in parenthesis or a subscript; for example, $C_{\ve,\kap}$ and $C(\ve,\kap)$ both denote a constant that may depend on parameters $\ve$ and $\kap$.  Then $A\lesssim_{\ve,\kap}B$ means that $A\leq C(\ve,\kap) B$.

We shall write $A\ll B$ to mean that ``$A$ is much smaller than $B$'', namely that $A\leq c B $ for a sufficiently small absolute constant $c>0$.

\subsection{Balls, annuli, and neighbourhoods}
Balls $B(x,r)$ are assumed to be open. Also, when we say that a set $B\subset\R^2$ is a ball, we mean
an open ball, unless otherwise stated. We denote by $r(B)$ its radius.

The notation $A(x,r,R)$ stands for an open annulus centered at $x$
with inner radius $r$ and outer radius $R$.

For a set $E$ and $r>0$, the notation $U_r(E)$ denotes the open $r$ neighbourhood of $E$.

\subsection{Jordan domains}\label{def:Jordandomain}  A domain is a connected open set. We call a domain $\Omega^+$ a Jordan domain if its boundary $\Gamma = \partial\Omega^+$ is a Jordan curve.  In this case (by the Jordan curve theorem), $\Omega^- = \R^2\backslash \Omega^+$ is also a Jordan domain.

\subsection{Measures}  Throughout the paper, by a measure we shall mean a non-negative locally finite Borel measure.

For $C_0>0$, a measure $\mu$ has $C_0$-linear growth if $$\mu(B(x,r))\leq C_0r\text{ for all }x\in\R^2\text{ and }r>0.$$

For a ball $B\subset \R^2$, we write
$$\Theta_\mu(B) =\frac{\mu(B)}{r(B)}.$$
This should be understood as a kind of $1$-dimensional density of $\mu$ over $B$.

\subsection{2-sided corkscrew open sets}\label{def:2sidedcorkscrew}

Let $\Omega\subset\R^2$ be an open set. We say that $\Omega$ satisfies the $c$-corkscrew condition (or just, corkscrew condition) if there
exists some $c>0$ such that for
all $x\in\partial\Omega$ and all $0<r< \diam(\Omega)$ there exists some ball
$B\subset \Omega\cap \overline{B(x,r)}$ with $r(B)\geq c r$.

We say that $\Omega$ satisfies the $2$-sided ($c$-)corkscrew condition if both $\Omega$ and $\R^2\setminus \overline\Omega$ satisfy the ($c$-)corkscrew condition.

We say that $\Omega\subset\R^2$ is a $2$-sided corkscrew  open set (or domain) if it is an open set (or domain) that satisfies the $2$-sided corkscrew condition.

For example, quasicircles are $2$-sided corkscrew domains.  Indeed, quasicircles are simply connected $2$-sided corkscrew domains that satisfy a Harnack chain condition, according to Peter Jones (see Theorem 2.7 in \cite{Jerison-Kenig}), or in other words, they are the same as planar simply connected NTA domains. On the other hand, it is easy to check that there are simply connected $2$-sided corkscrew domains which are not quaiscircles.

\subsection{Cones and tangents}\label{def:cones-tangents}

For a point $x\in\R^{2}$, a unit vector $u$,
and an aperture parameter $a\in(0,1)$ we consider the two sided cone with axis in the direction of $u$ defined by
$$X_a(x,u)=\bigl\{y\in\R^{2}:|(y-x)\cdot u|> a|y-x|\bigr\}.$$

Given an open set $\Omega^+\subset\R^{2}$ and $x\in\partial\Omega^+$,
we say that $\partial\Omega^+$ has a {\em tangent} at $x$, and that $x$ is a tangent point for $\partial\Omega^+$ if there exists a unit vector $u$ such that, for all
$a\in(0,1)$, there exists some $r>0$ such that
$$\partial \Omega^+\cap X_a(x,u)  \cap B(x,r) =\varnothing,$$
and moreover, one component of $X_a(x,u)\cap B(x,r)$ is contained in $\Omega^+$ and the other in $\Omega^-=\R^2\setminus \overline{\Omega^+}$.
The line $L$ orthogonal to $u$ through $x$ is called a tangent line at $x$. Notice that this notion of tangent is associated with the domain $\Omega^+$, and it would be more appropriate to say that
$L$ is a tangent for $\Omega^+$.



\section{Smoother square functions}\label{smoothsquares}

Several smoother versions of the Carleson square function play an important role in our analysis, as one can see by the statements of Main Lemmas \ref{l:main1} and \ref{l:main2} in the next section.  In this section we show that these smoother square functions are controlled by the Carleson square function (with the addition of an absolute constant).\\

Suppose that $\Omega^+\subset \R^2$ is an open set, $\Gamma = \partial \Omega^+$ and $\Omega^- = \R^2\backslash \overline{\Omega^+}$.\\

First denote
$$\alpha^+(x,r) = \bigg|\frac\pi2 - \frac1{r^2}\int_{\Omega^+} e^{-|y-x|^2/r^2}\,dy\bigg|,$$
and set
$$\mathcal{A}(x)^2 = \int_0^1 \alpha^+(x,r)^2\,\frac{dr}r.$$

More generally, for a non-negative smooth function $\varphi:\R\to [0,\infty)$ satisfying
$$\int_0^{\infty}\varphi(t) t\sqrt{\log(e+t)}dt<\infty,
$$
set $\psi(x) = \varphi(|x|)$.  Consider
$$\apsi(x,r) = \Bigl|c_{\psi} - \frac{1}{r^2}\int_{\Omega^+}\psi\Bigl(\frac{x-y}{r}\Bigl)dy\Bigl|,\;\text{ where }c_{\psi} = \int_{\R^2_+}\psi\Big(\frac{y}{r}\Big)dy,
$$
and define
$$\Apsi^1(x)^2 = \int_0^1\apsi (x,r)^2\frac{dr}{r}.
$$

\begin{lemma}\label{lem1}
There is a constant $C_{1,\psi}$  such that for every $x\in\R^2$, every $R>0$, and $M\geq 1$, we have
$$\int_0^R \apsi(x,r) ^2\frac{dr}{r}\leq C_{1,\psi} \int_0^{MR} \ve(x,r) ^2\frac{dr}{r}  + 2\int_M^\infty \varphi(t)\, t\,\Big(\log^+\frac tM\Big)^{1/2}
\,dt.$$
\end{lemma}

\begin{proof}
Observe that, by integrating polar coordinates centered at $x$,
\begin{align}\label{eqal82}
\apsi(x,r) & = \left|c_\psi - \frac1{r^2}\int_0^\infty
\varphi\Bigl(\frac{s}{r}\Bigl) \HH^1(\partial B(x,s)\cap\Omega^+)\,ds\right|\\
& = \frac1{r^2}\left|\int_0^\infty
\varphi\Bigl(\frac{s}{r}\Bigl) \big(\pi s - \HH^1(\partial B(x,s)\cap\Omega^+)\big)\,ds\right|.\nonumber
\end{align}
Next recall that $I_{\pm}(x,s)$ are the longest arcs in $\Omega^{\pm}\cap \partial B(x,s)$, so
$$I^+(x,s)\subset \partial B(x,s)\cap\Omega^+ \subset \partial B(x,s)\setminus I^-(x,s),$$
and consequently
$$\HH^1(I^+(x,s))\leq \HH^1( \partial B(x,s)\cap\Omega^+) \leq 2\pi s - \HH^1(I^-(x,s)).$$
Subtracting $\pi s$ from this inequality easily yields that
$$\big|\pi s - \HH^1(\partial B(x,s)\cap\Omega^+)\big| \leq s\,\ve(x,s),$$
which, when plugged into \rf{eqal82} yields
\begin{equation}\label{eq:apsionescale}\apsi(x,r)\leq  \frac1{r^2}\int_0^\infty
\varphi\Bigl(\frac{s}{r}\Bigl)\ve(x,s)\,s\,ds.\end{equation}
Squaring both sides of (\ref{eq:apsionescale}) and integrating over $r\in(0,R)$ yields,
\begin{align*}
\left(\int_0^R \apsi(x,r) ^2\frac{dr}{r}\right)^{1/2}
 & \leq \left(\int_0^R  \left(\frac1{r^2}\int_0^\infty
\varphi\Bigl(\frac{s}{r}\Bigl)\ve(x,s)\,s\,ds\right)^2\,\frac{dr}r\right)^{1/2} \\
&= \left(\int_0^R  \left(\int_0^\infty
\varphi(t)\ve(x,tr)\, t\,dt\right)^2\,\frac{dr}r\right)^{1/2}\\
&\stackrel{\text{Minkowski's inequality}}{\leq} \int_0^\infty
\left(\int_0^R
\ve(x,tr)^2\,\frac{dr}r\right)^{1/2}\varphi(t)\, t\,dt \\
&= \int_0^\infty
\left(\int_0^{tR}
\ve(x,u)^2\,\frac{du}u\right)^{1/2}\,\varphi(t)\, t\,dt.
\end{align*}
For $t\geq1$ and some $M>1$, we split
$$\int_0^{tR}
\ve(x,u)^2\,\frac{du}u \leq \int_0^{MR}
\ve(x,u)^2\,\frac{du}u + \int_{MR}^{tR}\frac{du}u = \int_0^{MR}
\ve(x,u)^2\,\frac{du}u + \log^+ \frac tM.$$
This bound certainly also holds for $t\in(0,1)$, so we get
\begin{align*}
\int_0^\infty
\left(\int_0^{tR}
\ve(x,u)^2\,\frac{du}u\right)^{1/2}\,\varphi(t)\, t\,dt &\leq  \left(\int_0^{MR}
\ve(x,u)^2\,\frac{du}u\right)^{1/2}  \int_0^\infty
\varphi(t)\, t\,dt \\
&\quad + \int_M^\infty \varphi(t)\, t\,\Big(\log^+\frac tM\Big)^{1/2}
\,dt,
\end{align*}
and the lemma follows.
\end{proof}
\vv


\section{The two main lemmas and the  proof of the main theorem}

Having introduced an array of square functions, we may now state the primary two technical results of this paper.  The paper splits into two essentially disjoint parts, which use very different techniques.

\subsection*{Part I}  The first part of the paper concerns the use of compactness arguments to show, roughly speaking, that the curve $\Gamma$ must be quite flat near points where the Carleson square function is finite.

\begin{mlemma}\label{l:main1}
Let $\Omega^+\subset\R^2$ be either a Jordan domain or a two sided corkscrew open set, let $\Gamma=\partial\Omega^+$, and let $\mu$ be a measure with $1$-linear growth
supported on $\Gamma$. Let $B$ be a ball centered at $\Gamma$ such that
$$\mu(B)\geq \theta r(B),$$
for some $\theta \in (0, 1)$.
Given any $\ve>0$, there exists $\delta\in (0,1)$, depending on $\theta$ and $\ve$ (and the two sided corkscrew parameter in that case), such that if
$$\int_{7B}\int_0^{7 r(B)} [\ve(x,r)^2 +\alpha^+(x,r)^2]\,d\mu(x)\frac{dr}r \leq \delta\,\mu(7B),$$
then
$$\beta_{\infty,\Gamma}(B)\leq \ve.$$
\end{mlemma}

Observe that, roughly speaking, this lemma ensures that $\beta_{\infty,\Gamma}(B)$ is as as small as wished
if a suitable square function involving the coefficients $\ve(x,r)$ and $\alpha^+(x,r)$ is small enough on $\supp\mu\cap7B$, assuming also that $\mu$ has linear growth and that $\mu(B)$ is not too small.
It is important to remark that the lemma yields an estimate of the flatness of $\Gamma\cap B$, not only of
 $\supp\mu\cap B$. This will be crucial later for the proof of Theorem \ref{teo-main}.

  The proof of Main Lemma \ref{l:main1} is considerably easier in the case of $2$-sided corkscrew open sets, since these sets are rather stable under natural limit operations (see Lemma \ref{l:convergence}).  Jordan domains do not have similar stability properties and so the analysis is much more delicate.  However, the case of the $2$-sided corkscrew open set is nevertheless very instructive, as a key part of our analysis is that, if $\Omega^+$ is a Jordan domain, then at points and scales where $\mu$ has a lot of mass, and the Carleson square function is small, one can find corkscrew balls (Lemma \ref{lem2cork}).   This property, which is much weaker than the two sided corkscrew condition insofar as it tells us nothing about $\Gamma$ at points where $\mu$ has little mass, is still sufficient for us to prove Main Lemma \ref{l:main1} with a considerable amount of additional work.

\subsection*{Part II}  The second part of this paper is concerned with improving the local flatness which is provided by Main Lemma \ref{l:main1} into a rectifiability property.  For this we work with the general scheme introduced by David and Semmes \cite{DS1} and extended to the non-homogeneous context by L\'eger \cite{Leger}.  In fact we will not require the full strength of the Carleson square function, but rather a smoother square function.

\begin{mlemma}\label{l:main2}
Let $\Omega^+\subset\R^2$ be an open set, and let $\Gamma=\partial\Omega^+$. Fix $c_0\in(0,1)$, $\theta>0$ and $\ve>0$.  Let $B_0$ be a ball centered at $\Gamma$ and let $\mu$ be a measure with $1$-linear growth
supported on $\Gamma\cap \overline{B_0}$ satisfying the following conditions:
\begin{itemize}
\item  $\mu(B_0)\geq c_0 r(B_0)$.
\item  $\beta_{\infty,\Gamma}(B)\leq \ve$ for any ball $B$ centered at $\Gamma$ such that
 $\mu(B)\geq \theta \, r(B)$.
\item For a radial function $\psi\in C^{\infty}(\R^2)$ with $\chara_{B(0,1)}\leq \psi\leq \chara_{B(0,1.1)}$, it holds that
$$\int_0^{r(B_0)/\ve}\apsi(x,r)^2\frac{dr}{r}\,d\mu(x) \leq \ve \text{ for every }x\in \supp(\mu).$$
\end{itemize}
If $\theta$ is small enough in terms of $c_0$, and $\ve$ is small enough in terms of $\theta$ and $c_0$, then   there exists a Lipschitz graph $\Lambda$ with slope at most $1/10$ such that
   $$  \mu(\Lambda) \geq \frac12\mu(B_0).$$
\end{mlemma}

The key property of the square function generated by the coefficients $\apsi$ which enables a L\'eger type construction are the Fourier estimates carried out in Section \ref{s:fourier}, see in particular Lemma \ref{lemlips}.  Subsequently, we carry  out the construction itself, which has several subtleties due to the nature of our particular square function.

\vv

\subsection{The proof of Theorem \ref{teo-main}}

Before beginning the proof we recall some basic facts about densities:  For a set $E\subset \R^2$ we set
$$\Theta^{1,*}(x,E) = \limsup_{r\to 0}\frac{\HH^1(E\cap B(x,r))}{2r}, \;\; \Theta^1_*(x,E) = \liminf_{r\to 0}\frac{\HH^1(E\cap B(x,r))}{2r}.
$$
For the proof of the following simple lemma, see for example \cite[Theorem 6.1]{Mattila-book}.
\begin{lemma} If $E\subset \R^2$ satisfies $\HH^1(E)\in (0,\infty)$, then
$$\frac{1}{2}\leq \Theta^{1,*}(x,E)\leq 1\text{ for }\HH^1\text{-a.e.} \, x\in E.
$$
\end{lemma}

We will also require two simple properties of Lipschitz graphs.  Suppose that $\Lambda$ is a Lipschitz graph in $\R^2$, and $F\subset \Lambda$, then for $\HH^1$-a.e. $x\in F$
\begin{itemize}
\item $\Lambda$ has a tangent at $x$, and
\item $\Theta^1_*(x,F) = \Theta^{1,*}(x,F)=1$.
\end{itemize}
Both properties follow easily from Lebesgue's theorem on the almost everywhere differentiability of absolutely continuous functions on the real line.

\begin{proof}[\bf Proof of Theorem \ref{teo-main} using the Main Lemmas \ref{l:main1} and \ref{l:main2}]
We have to show that the set
$G=\{x\in\Gamma:\EE(x)<\infty\}$ is rectifiable and that for $\HH^1$-a.e.\ $x\in G$ there exists a tangent
to $\Gamma$.

Standard arguments, see, for instance \cite[Lemma 15.13]{Mattila-book}, yield the rectifiability of the set $G$ from the  existence of tangents to $\Gamma\supset G$ at
$\HH^1$-almost every point of $G$.  Therefore our goal is to prove the statement about the existence of tangents.

 For the sake of contradiction, suppose that the subset $F_0\subset  G$ of those  points $x\in G$ which are not tangent points for $\Gamma$ has positive $\HH^1$ measure. Consider a subset $F\subset F_0$ such that $0<\HH^1(F)<\infty$.

Since the Carleson square function $\mathcal{E}(x)^2<\infty$ for $\HH^1$-a.e. $x\in F$, we have from Lemma \ref{lem1} that
$$\int_0^1[\ve(x,r)^2+\alpha^+(x,r)^2+\apsi (x,r)^2]\frac{dr}{r}<\infty\text{ for }\HH^1\text{-a.e. } x\in F,
$$
where $\psi$ is the function from Main Lemma \ref{l:main2}.

By replacing $F$ by a subset with positive $\HH^1$ measure if necessary,
we may assume that
\begin{equation}\label{e:unifint}
\lim_{s\to0} \int_0^s\big(\ve(x,r)^2 + \alpha^+(x,r)^2+\apsi(x,r)^2\big)\,\frac{dr}r=0\quad \mbox{ uniformly in $F$,}
\end{equation}
and
$$\HH^1(B(x,r)\cap F)\leq 3r\quad \mbox{ for all $x\in F$, $r>0$.}$$
(This second inequality is a consequence of the fact that $\Theta^{1,*}(x,F)\leq 1$ for $\HH^1$-a.e.\ $x\in F$.)

For the choice $c_0=1/9$, pick $\theta>0$ and then $\ve\in (0,\theta)$ small enough positive numbers so that Main Lemma \ref{l:main2} is applicable.  Then choose $\delta>0$ small enough so that
Main Lemma \ref{l:main1} is applicable with the choice $c_0$ replaced by $\theta$.

  Let $R$ be small enough so that
\begin{equation}\label{Repseq}\int_0^{7R/\ve}\big(\ve(x,r)^2 + \alpha^+(x,r)^2+\apsi(x,r)^2\big)\,\frac{dr}r\leq \min(\delta,\ve).\end{equation}
for all $x\in F$.

Denote $\mu=\tfrac{1}{3}\HH^1|_F$.  Then $\mu$ has $1$-linear growth. Recalling that $\Theta^{1,*}(x,F)\geq 1/2$ for $\HH^1$-a.e.\ $x\in F$,
we can find a ball $B_0$ centered at $F$ with radius smaller than $R$ such that
$\mu(B_0)\geq r(B_0)/9 = c_0r(B_0)$. 

We look to apply Main Lemma \ref{l:main2} with the measure $\nu =\mu|_{B_0}$ (which satisfies $\nu(B_0)\geq  c_0r(B_0)$).  Notice that if $B$ is a ball with $\nu(B)\geq \theta r(B)$, then certainly $B\cap B_0\neq \varnothing$, and $r(B)\leq r(B_0)/\theta\leq r(B_0)/\ve$.  Consequently, from (\ref{Repseq}) we infer that $$\int_{7B}\int_0^{7 r(B)} [\ve(x,r)^2 +\alpha^+(x,r)^2]\,d\mu(x)\frac{dr}r \leq \delta\,\mu(7B).$$
But trivially have $\mu(B)\geq \theta r(B)$, and so Main Lemma \ref{l:main1} yields that $\beta_{\infty,\Gamma}(B)\leq \ve$.\\

On the other hand, it is also immediate from (\ref{Repseq}) that
$$\int_0^{r(B_0)/\ve}\apsi(x,r)^2 \frac{dr}{r}d\nu(x)\leq \ve  \text{ on }\supp(\nu).
$$
Consequently, we may apply Main Lemma \ref{l:main2} to find a Lipschitz graph $\Lambda$ such that the set $F_1 = F\cap \Lambda$ satisfies $\HH^1(F_1)>0$.

As a consequence, for $\HH^1$-a.e.\ $x\in F_1$, we have \begin{equation}\label{xgoodprops}\Theta_*^{1}(x,F_1)= 1\text{ and } \Lambda\text{ has a tangent at }x.\end{equation} We claim that, every $x\in F_1$ satisfying (\ref{xgoodprops}) the tangent line for $\Lambda$ at $x$ is also a tangent to $\Gamma$.

To verify the claim, we will appeal to Main Lemma \ref{l:main1}.  Fix $x\in F_1$ satisfying (\ref{xgoodprops}).  Observe that \rf{e:unifint}, along with the condition $\Theta_*^{1}(x,F_1)= 1$, ensure that for any $\ve>0$ we can find $r_0>0$ such that for every $r<r_0$ we can apply Main Lemma \ref{l:main1} with the measure $\mu$ and the ball $B_0=B(x,r)$ (with the constant $c_0$ equal to, say, $1/4$).  Therefore,
\begin{equation}\label{e:betalim1}
 \lim_{r\to0}\beta_{\infty,\Gamma}(B(x,r))=0.
\end{equation}

Now, let $u$ be a unit vector orthogonal to the tangent line at $x$ to $\Lambda$. Observe
that, for every $a\in(0,1)$, $\Gamma\cap X_a(x,u)\cap A(x,r/2,r)=\varnothing$ for  all sufficiently small $r>0$, since otherwise
$\beta_{\infty,\Gamma}(B(x,r))\geq c(a)>0$ for  all $r$ small enough, contradicting \rf{e:betalim1}.

But the condition $\Gamma\cap X_a(x,u)\cap A(x,r/2,r)=\varnothing$ for  all $r$ small enough clearly implies that $\Gamma\cap X_a(x,u)\cap B(x,r)=\varnothing$ for  all $r$ small enough.  Further, it is
immediate that the condition \rf{e:unifint} implies that one component of $X_a(x,u)\cap B(x,r)$ is
contained in $\Omega^+$ and the other in $\Omega^-$, and so $\Gamma$ has a tangent at $x$.  Therefore our claim follows, and this in turn clearly contradicts the fact that the points in $F_0$, and thus the ones in $F_1$, are not tangent points for $\Gamma$.
\end{proof}

\vv


\part*{Part I:  Flatness via compactness arguments}

\section{Basic compactness properties}

\subsection{Weak convergence of measures}

We say that a sequence of (Borel) measures $\mu_j$ converges weakly to a measure $\mu$ if $$\int_{\R^2}f\, d\mu_j = \int_{\R^2}f\, d\mu\text{ for every }f\in C_0(\R^2), $$
where $C_0(\R^2)$ denotes the continuous functions with compact support.

\begin{lemma}\label{mujsubsequence} If $\mu_j$ is a sequence of measures in $\R^2$ satisfying $\sup_j \mu_j(B(0,R))<\infty$ for any $R\in (0,\infty)$, then $\mu_j$ has a weakly convergent subsequence.\end{lemma}

It is easy to see that weak limits are lower-semicontinuous on open sets and upper-semicontinuous on compact sets.\\

Bringing our observations together, we arrive at the following result.

\begin{lemma}\label{lineargrowthcompact}
Fix $C_0, c_0\in (0,\infty)$.  Fix a ball $B_0\subset \R^2$.  Suppose that $\mu_j$ is a sequence of measures with $C_0$-linear growth such that $\mu_j(\overline{B_0})\geq c_0 r(B_0)$ for every $j$.  Then there is a subsequence $\mu_{j_k}$ of the measures which converges weakly to a measure $\mu$ with $C_0$-linear growth satisfying $\mu(\overline{B_0})\geq c_0 r(B_0)$.
\end{lemma}

We next establishing some basic facts about convergence of sets.

\subsection{Convergence of sets} For $B\subset \R^2$ and $x\in \R^2$, set
 $$\dist(x,B)=d(x,B) = \inf_{b\in B}|x-b|.
 $$

   For non-empty sets $A$ and $B$, we define the excess of $A$ over $B$ to be the quantity
 $$\excess(A,B) = \sup_{x\in A}d(x,B),
 $$
 and put $\excess(\varnothing,B)=0$ while $\excess(A,\varnothing)$ is left undefined.

Observe that $\excess(A,B)<\ve$ means that the open $\ve$-neighbourhood of $B$ contains $A$.

The Hausdorff distance between $A$ and $B$ is given by
$$\Hdist(A,B) = \max\{\excess(A,B), \excess(B,A)\}.$$

 For compactness arguments we will require a notion of local convergence.  To this end, we follow \cite{BL} and introduce the relative Walkup-Wets distance. For non-empty sets $A,B$, we define, for $R>1$,
 $$d_R(A,B) = \max\bigl\{\excess(A\cap \overline{B(0,R)},B), \excess(B\cap \overline{B(0,R)},A)\}.
 $$
 (The reader should not be concerned that the quantity $d_R$ need not satisfy the triangle inequality.)

 Observe that
 \begin{equation}\label{dRmonotone} d_{R_1}(A,B)\leq d_{R_2}(A,B)\leq \Hdist(A,B)\text{ if }R_2\geq R_1.
 \end{equation}


\begin{definition}[Local Convergence]\label{def:locconv}  A sequence of non-empty sets $E_j$ converge locally to a non-empty set $E$ (written $E_j\to E$ locally)  if, for any $R>0$,
$$\lim_{j\to \infty}d_R(E_j,E)=0.
$$
\end{definition}

We refer the reader to Section 2 of \cite{BL} for a more thorough introduction to this notion of convergence.  In variational analysis, this notion of convergence is called convergence in the Attouch-Wets topology.

\begin{lemma}\label{lem:baspropslocconv} If $E_j$ are non-empty closed sets that converge locally to a non-empty closed set $E$, then
\begin{enumerate}
\item a compact set $K$ satisfies $K\cap E=\varnothing$ if and only if there is a neighbourhood of $K$ that has empty intersection with $E_j$ for all sufficiently large $j$, and
\item if the sets $E_j$ are contained in a fixed compact set, then $E_j$ converge locally to $E$ if and only if $E_j$ converges to $E$ in the Hausdorff distance.
\end{enumerate}
\end{lemma}

\begin{proof} Both properties are straightforward consequences of the local convergence, so we shall only verify the `if' direction of (1). If $K\cap E=\varnothing$, there exists $r>0$ such that $K_{\delta}$, the $\delta$-neighbourhood of $K$, satisfies $$\inf_{x\in K_{\delta}, \,  y\in E} |x- y|>\delta.$$ But then there exists $R>0$ such that $K_{2\delta}\subset B(0,R)$.  But then
$$d_{2R}(E_j, E)<\delta \text{ for sufficiently large }j,$$
so the open $\delta$-neighbourhood of $E$ contains $E_j\cap B(0, 2R)$ for sufficiently large $j$.  Consequently, $K_{\delta}\cap E_j = K_{\delta}\cap E_j\cap B(0, 2R)=\varnothing$ for sufficiently large $j$.
\end{proof}

We next state a basic compactness result.

\begin{lemma}\label{hauscomp}  Suppose that $E_j$ is a sequence of closed sets in $\R^2$ that intersect $\overline{B(0,1)}$.  Then there is a subsequence $E_{j_k}$ that converges locally to a closed set $E\subset \R^2$ (satisfying $E\cap \overline{B(0,1)}\neq \varnothing$).
\end{lemma}

 This statement can be proved by modifying the usual proof of the relative compactness of a sequence of closed subsets of a compact metric in the Hausdorff topology, see also Theorem 2.5 of \cite{BL} and references therein.\\

Let us now fix open sets $\Omega_j^+\subset \R^2$ with boundary $\Gamma_j = \partial \Omega_j^+$.  We write $\Omega_j^- = \R^2\backslash \overline{\Omega_j^+} $.  Throughout this paper, we will always be working in situations where also
\begin{equation}\label{gammaminusbd}
\Gamma_j = \partial\Omega_j^-.
\end{equation}

\begin{lemma} \label{l:convbasic}
Let $\{\Omega_j^+\}_j$ be a sequence of open sets in the plane. 
Set $\Omega_j^-=\R^2\setminus \overline{\Omega_j^+}$ and suppose that $\Gamma_j=
\partial\Omega_j^+$ satisfies (\ref{gammaminusbd}).
Suppose there are closed sets $G^+,G^-,G_0$ satisfying
    \begin{align*}
    \overline{\Omega_{j}^\pm}\to G^\pm \quad \mbox{and}\quad
        \Gamma_{j} \to G_0 \quad\mbox{locally.}
        \end{align*}
Then
\begin{enumerate}
    \item The limit sets $G^+,G^-,G_0$  satisfy
    $$G^+\cup G^- =\R^2,\quad \;G^+\cap G^-=G_0.$$
 In particular, $G^+\setminus G_0$ and $G^-\setminus G_0$ are open.

    \item There are functions $g^+,g^-\in L^\infty(\R^2)$ such that, for a subsequence $\Omega_{j_k}^{\pm}$
    $$\chara_{\Omega_{j_k}^\pm}\to g^\pm\quad \mbox{weakly $*$ in $L^\infty(\R^2)$},$$
    where
    $$\mbox{$g^+=1$ in $G^+\setminus G_0$\; \quad and \; \quad $g^+=0$ in $G^-\setminus G_0$.}$$
   \end{enumerate}
\end{lemma}

\begin{proof} For property (1), the fact that $G^+\cup G^- =\R^2$ is
obvious.  Since $\Gamma_{j} = \partial \Omega^+_{j} =\partial\Omega^-_{j}\subset \overline{\Omega_{j}^+}\cap \overline{\Omega_{j}^-}$, it is clear that $G_0\subset G^+\cap G^-$.  On the other hand, if $x\in G^+\cap G^-$, then for any $\ve>0$ there exists $j_0\in \mathbb{N}$ such that for
$j\geq j_0$,
$$\dist\big(x,\overline{\Omega_{j}^+}\big)\leq \ve\quad \text{ and }\quad \dist\big(x,\overline{\Omega_{j}^-}\big)\leq \ve.$$
That is, there exist $y_{j}^\pm\in\Omega_{j}^\pm$ such that $|x-y_{j}^\pm|\leq \ve$. There exists
some $z\in \Gamma_{j}$ in the segment $[y_{j}^+,y_{j}^-]$, and thus $|x-z|\leq\ve$ and
$\dist(x,\Gamma_{j})\leq \ve$. Since this holds for all ${j}$ big enough, we deduce that $x$ belongs to
the limit in the Attouch-Wets topology of $\{\Gamma_{j}\}_k$, that is, to $G_0$.

To see the openness of $G^+\setminus G_0$, note that $\R^2 = (G^+\setminus G_0) \cup G^-$ is a disjoint
union. Thus $G^+\setminus G_0 = \R^2\setminus G^-$ is open. Analogously, $G^-\setminus G_0
=\R^2\setminus G^+$ is open.

We now turn our attention to verifying (2).  The existence of $g^{\pm}\in L^{\infty}(\R^2)$ such that, for a subsequence $\Omega_{j_k}$, $\chara_{\Omega_{j_k}^\pm}\to g^\pm$ weakly $*$ in $L^\infty(\R^2)$ is a standard consequence of the Banach-Alouglu theorem.  Now consider a continuous  function $\vphi$ compactly supported on $G^+\setminus G_0$. Recall that $G^+\setminus G_0$ is open and $G^+\setminus G_0 = \R^2\setminus G^-$. Consequently, property (1) of Lemma \ref{lem:baspropslocconv} ensures that there exists some $\ve >0$ such that, for all $k$ big enough,
$$\dist\Big(\supp\vphi,\, \overline{\Omega_{j_k}^-}\Big)\geq \ve.$$
In particular, $\supp\vphi\subset \Omega_{j_k}^+$ for all $k$ big enough, which implies that
$$\int_{\R^2} \chara_{\Omega_{j_k}^+}\,\vphi \,dx= \int_{\R^2}\vphi\,dx\quad \mbox{ for all $k$ big enough,}$$
and proves that $g^+$, the weak $*$ limit of $\chara_{\Omega_{j_k}^+}$, equals $1$ in $G^+\setminus G_0$.
 The proof that $g^+=0$ in $G^-\setminus G_0$ is completely analogous.
\end{proof}

\section{General compactness results involving the Carleson square function}

Throughout this section, fix a sequence of sets $\Omega_j^+$ with $\Omega_j^-=\mathbb{R}^2\backslash \overline{\Omega_j^+}$, such that $\Gamma_j = \partial\Omega_j^+=\partial\Omega_j^-$ (i.e. (\ref{gammaminusbd}) holds).   Assume that $\overline{\Omega_j^\pm}\to G^{\pm}$ and $\Gamma_j \to G_0$ locally as $j\to \infty$.  Consequently, the sets $G^+$, $G^-$ and $G_0$ will satisfy the properties of Lemma \ref{l:convbasic}.

\subsection{Complementary semicircumferences and admissible pairs} It will be convenient to introduce the following definitions, which will be central to our analysis.

\begin{definition}[Complementary semicircumferences]  We say that two closed semicircumferences are \emph{complementary} if they are contained in the same circumference and their intersection consists just of their end-points.
\end{definition}

\begin{definition}[Admissible pairs]\label{def:admissible} We say that a pair of two complementary closed semicircumferences $(S_1,S_2)$
is admissible (for the sequence of sets $\{\Omega_j^+\}_j$) if there exists a subsequence of circular arcs
$I_{j_k}^\pm\subset \partial B(x_{j_k},r_{j_k})\cap \Omega_{j_k}^\pm$, with $x_{j_k}\in\Gamma_{j_k}$, such that
$I_{j_k}^+$, $I_{j_k}^-$ converge to $S_1$, $S_2$ in Hausdorff distance, respectively.
\end{definition}

It is immediate to check that, if $(S_1,S_2)$ is an admissible pair, then $S_1\subset G^+$, $S_2\subset G^-$ and that the common center of $S_1$ and $S_2$ belongs to $G_0$.  Consequently, we will also say that $S_1$ and $S_2$ are admissible for $G^+$ and $G^-$, respectively. We call the common center and radius of $S_1$, $S_2$ the center and radius of the pair, respectively.

Observe that, for a given $x\in G_0$, $r>0$, there may exist more than one admissible pair of
semicircumferences centered at $x$ with radius $r$. \\

Especially when dealing with Jordan domains, we will use the fact that \emph{the set of admissible pairs is closed in the topology of Hausdorff
distance}, that is, if $\{(S_{1,i},S_{2,i})\}$ is a sequence of admissible pairs (possibly with different centers and radii) such that $S_{1,i}$, $S_{2,i}$ converge respectively to $S_1$, $S_2$, then $(S_1,S_2)$ is an admissible pair.

To check this fact, just take for each $i$ a pair of arcs $I_{i}^+$, $I_i^-$ contained
in $\partial B(x_{j_i},r_{j_i})\cap \Omega_{j_i}^\pm$, with $x_i\in\partial\Omega_{j_i}^\pm$, for a suitable $j_i$ such that
$$\Hdist(I_i^\pm,S_{1,i}) \leq \frac 1i.$$
It is clear that the arcs  $I_{i}^+$, $I_i^-$ converge respectively to $S_1$, $S_2$ in Hasdorff distance,
and thus $(S_1,S_2)$ is an admissible pair.

\subsection{A general convergence result}  We set $\ve_j(x,r)$ and $\alpha_j^+(x,r)$ to be the coefficients $\ve(x,r)$ and $\alpha^+(x,r)$ associated with $\Omega_j^+$.

\begin{lemma}\label{lem:basicsquare}
Fix $C_0, c_0>0$.  Let $\{\mu_j\}_j$ be a sequence of measures with $C_0$-linear growth supported on $\Gamma_j$ converging weakly to a measure $\mu_0$ (so $\mu_0$ is supported in $G_0$, and has $C_0$-linear growth).  Suppose $B_0$ is a ball with $\mu_0(\overline{B_0})\geq c_0 r$.  Further assume that, for each $j$, both
\begin{align} \label{eq:G1}
    \int_{7B_0}\int_0^{7r(B_0)}  \alpha^+_{j}(x,r)^2 \, \dr \, d\mu_j(x) \leq \frac{1}{j}\,\mu_j(7B_0),
\end{align}
and
\begin{align} \label{eq:G2}
    \int_{7B_0}\int_0^{7r(B_0)}  \ve_{j}(x,r)^2 \, \dr \, d\mu_j(x) \leq \frac{1}{j}\,\mu_j(7B_0).
\end{align}

Then
\begin{enumerate}
    \item \label{existencevariety} There is an analytic variety $Z$ such that $\supp(\mu_0)\cap 7B_0\subset Z\subset G_0$.
\item  \label{existencepairs} For all $x\in 7B_0\cap \supp\mu_0$ and all $r\in (0,7r(B_0))$ there is a pair of admissible semicircumferences which are contained in $\partial B(x,r)$.
 \end{enumerate}
\end{lemma}

\begin{proof}  We may assume by scaling that $B_0 = B(0,1)$.    The property (\ref{eq:G1}) is responsible for the first conclusion, while (\ref{eq:G2}) is responsible for the second conclusion.\\

\textbf{Proof of (\ref{existencevariety}).} Recall from Lemma \ref{l:convbasic} that there is a subsequence of the open sets $\Omega_{j_k}^{\pm}$ whose characteristic functions converge weak-$*$ in $L^{\infty}$ to functions $g^{\pm}$ with $g^{+}\equiv 1$ on $G^+\backslash G_0$ and $g^+\equiv 0$ on $G^-\backslash G_0$.

 \begin{claim}\label{alpha0suppoort} One has $\alpha_0^+(x,r)=0$ for all $x\in 7B_0\cap\supp\mu_0$ and all $r\in(0,7r(B_0))$, where
 \begin{equation}\label{eqalpha0}
\alpha_0^+(x,r) = \bigg|\frac\pi2 - \frac1{r^2}\int g^+(y)\, e^{-|y-x|^2/r^2}\,dy\bigg|.
\end{equation}
\end{claim}

\begin{proof}[Proof of Claim \ref{alpha0suppoort}]  For any $r>0$, the mapping $x\mapsto \alpha_0^+(x,r)$ is continuous on $\R^2$, so  (\ref{eqalpha0}) will follow once we show that
\begin{equation}\label{eq*12}
\int_{7 B_0}\int_0^{7} \alpha^+_{0}(x,r)^2 r^3 \, dr \, d\mu_0(x)=0,
\end{equation}

Note that \eqref{eq:G1} implies that
\begin{align}\label{e:Gamma-k-power-3}
     \int_{7B_0}\int_0^{7} \alpha^+_{j_k}(x,r)^2 r^3 \, dr \, d\mu_k(x) \leq \frac{\mu_{j_k}(7B_0)}{j_k}\leq\frac{C}{k}.
\end{align}
Consider arbitrary non-negative smooth functions
$\tilde\chara_{7B_0}(x)$, $\tilde \chara_{(0,7)}(r)$ compactly supported in $7B_0$ and $(0,7)$, respectively.
Define
\begin{align*}
    f_k(x,r) := \tilde \chara_{7B_0}(x) \tilde \chara_{[0,7]}(r) \,r^3 \ps{ \frac{1}{r^2} e^{-\av{\frac{\cdot}{r}}^2} * \chara_{\Omega_{j_k}^+} - \frac\pi2 }^2.
\end{align*}
 Since $\chara_{\Omega_{j_k}^+}$ converges weakly $*$ in $L^\infty(\R^2)$ to  $g^+$, then we have that
\begin{align*}
    f_k(x,r) \to f(x,r) \mbox{ pointwise, }
\end{align*}
where
\begin{align*}
    f(x,r) = \tilde \chara_{7B_0}(x) \tilde \chara_{(0,7)}(r) \,r^3 \ps{ \frac{1}{r^2} e^{-\av{\frac{\cdot}{r}}^2} *g^+ - \frac\pi2 }^2.
\end{align*}
Clearly, $f_k$ is a uniformly bounded sequence on $\overline{7B_0 \times [0,7]}$ with uniformly bounded derivative (it is to ensure this condition that we introduce the factor $r^3$ in \eqref{e:Gamma-k-power-3}).
Thus by the Arzel\`{a}-Ascoli Theorem, we deduce that $f_k$ converges uniformly on compacts subsets to  $f$, up to a subsequence which we relabel.

To prove \rf{eq*12} we write
\begin{align*}
    & \iint   f\, dr \, d\mu_0 (x)   = \iint  f\,  dr \, d(\mu_0 - \mu_{j_k}) + \iint (f-f_k) \, dr\, d\mu_{j_k} + \iint  f_k \,dr\,  d\mu_{j_k}
\end{align*}
The first integral tends to $0$ as $k \to \infty$, since clearly $dr\, d\mu_{j_k}$ converges weakly to $dr \, d\mu_0$. Similarly, the second integral converges to $0$ as $k\to \infty$, by the uniform convergence on compact subsets of $f_k$ to $f$. As for the third integral, we see that
\begin{align*}
    \av{ \iint  f_k \, dr \, d\mu_{j_k} } \leq \int_{7B_0} \int_0^7 \alpha_{j_k}^+(x,r)^2 r^3 \, dr \, d\mu_{j_k}(x)  \leq \frac{C}{k},
\end{align*}
by \eqref{e:Gamma-k-power-3}. This immediately gives that
$$\iint   f\, dr \, d\mu_0=0,$$
and since
$\tilde\chara_{7B_0}(x)$, $\tilde \chara_{(0,7)}(r)$ are arbitrary non-negative smooth functions compactly supported in $7B_0$ and $(0,7)$ respectively, \rf{eq*12} follows.\end{proof}

\begin{claim}\label{lem:zeroalphaG0}  For any $x\in\R^2$, if there exists a sequence $r_k\to0$ such that $\alpha_0^+(x,r_k) = 0$ for all $k$, then $x\in G_0$.
\end{claim}

\begin{proof}[Proof of Claim \ref{lem:zeroalphaG0}]
Recall that $g^+=1$ in $G^+\setminus G_0$ and $g^+=0$ in $G^-\setminus G_0$.
So, if $x\in G^+\setminus G_0$, then it is immediate to check that
$$\lim_{r\to0} \frac1{r^2}\int g^+(y)\,e^{-|y-x|^2/r^2}\,dy =
\int e^{-y^2}\,dy = \pi,$$
taking also into account that $G^+\setminus G_0$ is open.
Thus $\alpha^+_0(x,r)$ is bounded away from $0$ for all $r>0$ small enough.

Similarly, if
$x\in G^-\setminus G_0$, then $\lim_{r\to 0}\frac{1}{r^2}\int_{\mathbb{R}^2}g^-(y) e^{-|y-x|^2/r^2} dy=0$, so $\alpha^+_0(x,r)$ is bounded away from $0$ if $r$ is sufficiently small.
\end{proof}

We now complete the proof of property (1).   Set
$$Z= \bigcap_{k\geq 0} \{x\in\R^2: \alpha_0^+(x,2^{-k})=0\},$$
where $\alpha_0^+$ is defined in \rf{eqalpha0}.
By Claims \ref{alpha0suppoort} and \ref{lem:zeroalphaG0},
$$7B_0\cap \supp\mu_0\subset Z\subset G_0.$$
To see that $Z$ is a real analytic variety consider
\begin{align*}
    F := \sum_{k \geq 0} 2^{-k}  \alpha_0^+(\cdot, 2^{-k})^2 = \sum_{k \geq 0} 2^{-k} \ps{ \frac{1}{r_k^2}\,g^+*e^{-|\cdot|^2 2^{2k}} - \frac\pi2 }^2.
\end{align*}
Then $F$ is a real analytic function and $Z=F^{-1}(0)$.

\vv

\textbf{Proof of property (2).}  Denote
$$E_k=\Big\{x\in\Gamma_k\cap 7B_0:  \textstyle{ \int_0^{7} \ve_k(x,r)^2 \,\frac{dr}r\leq \frac1{\sqrt k}}\Big\}.$$
By Chebyshev's inequality, we have
$$\mu_k(7B_0\setminus E_k) \leq \sqrt k \int_{7B_0}\int_0^{7} \ve_k(x,r)^2 \,\frac{dr}r\,d\mu_k(x)
\leq \frac{ \sqrt k}k \,\mu(7B_0)=  \frac1{\sqrt k}\, \mu_k(7B_0).
$$
Set $\tau_k =1-k^{-1/4}.$
For each $x\in E_k$ and $0<r<7$, we have
$$\frac1{\sqrt k}\geq \int_0^{7} \ve_k(x,s)^2 \,\frac{ds}s \geq
\int_{\tau_k r}^{r} \ve_k(x,s)^2 \,\frac{ds}s \geq \inf_{s\in [\tau_k r,r]}  \ve_k(x,s)^2 \,\log\frac1{\tau_k}
\approx \inf_{s\in [\tau_k r,r]}  \ve_k(x,s)^2 \,\frac1{k^{1/4}}.$$
Hence, for all $r\in (0,7)$ there exists some $s_r\in [\tau_k r,r]$ such that
$$\ve_k(x,s_r)\lesssim \frac1{k^{1/8}}.$$
In particular, this implies that for each $x\in E_k$ and $0<r<7$, there exist disjoint arcs $I_k^+(x,s_r),I_k^-(x,s_r)\subset
\partial B(x,s_r)$ satisfying
\begin{equation}\label{eqclau89}
\HH^1(I_k^\pm(x,s_r))\geq \big(\pi - k^{-1/8}\big)\,r\quad \text{
and }\quad I_k^\pm(x,s_r)\subset \overline{\Omega_k^\pm}.
\end{equation}

Let $\wt E_k\subset E_k$ be a compact set such that $\mu_k(\wt E_k)\geq \frac{k-1}k\,\mu_k(E_k)\geq \frac{k-1}{k}\bigl(1-\frac{1}{\sqrt{k}}\bigl)\mu_k(7B_0)$.
Taking a subsequence if necessary, we can assume that $\mu_k|_{\wt E_k}$ converges weakly $*$ to some measure $\sigma$ and that $\wt E_k$ converges in the Hausdorff metric to some compact set $F\subset
\R^2$. In fact, since $\wt E_k\subset \Gamma_k$, we have $F\subset G_0$. Further, it is easy to check that
$\supp\sigma\subset F$, and by \rf{eqclau89} it follows that for all $x\in F$ and all $r\in  (0,7)$, there
exists an admissible pair with radius $r$ and center $x$. 
It just remains to notice that $\sigma|_{7 B_0} = \mu_0|_{7 B_0}$, since for $f\in C_0(7B_0)$, $\Bigl|\int_{\wt E_k}fd\mu_k- \int_{7B_0}fd\mu_k\Bigl|\leq \frac{\|f\|_{\infty}\mu(7B_0)}{\sqrt{k}}$.\end{proof}


\section{The case when $\Omega^+$ is a $2$-sided corkscrew open set}

The objective of this section is to prove Main Lemma \ref{l:main1} in the case of a $2$-sided corkscrew open set.

\begin{lemma}\label{lembeta}
Let $\Omega^+\subset\R^2$ be a $2$-sided $c$-corkscrew open set, let $\Gamma=\partial\Omega^+$, and let $\mu$ be a measure with $C_0$-linear growth
supported on $\Gamma$. Let $B$ be a ball centered at $\Gamma$ such that
$$\mu(B)\geq c_0 r(B),$$
for some $0<c_0\leq C_0$.
Given any $\ve>0$, there exists $\delta>0$ (depending on $C_0,c_0,c,\ve$) such that if
$$\int_{7B}\int_0^{7r(B)} \big(\ve(x,r)^2 + \alpha^+(x,r)^2\big)\,d\mu(x)\frac{dr}r \leq \delta\,\mu(7B),$$
then
$$\beta_{\infty,\Gamma}(B)\leq \ve.$$
\end{lemma}

The next lemma shows that $2$-sided corkscrew open sets enjoy nice limiting properties under Hausdorff limits.

\begin{lemma} \label{l:convergence}
Let $\{\Omega_j^+\}_j$ be a sequence of $c$-corkscrew planar open sets such that $0\in\partial\Omega_j^+$ and $\inf_j\diam(\Omega^+_j)>0$. Let $\Omega_j^-=\R^2\setminus \overline{\Omega_j^+}$ and $\Gamma_j=
\partial\Omega_j^+$.
 Then the following holds.
\begin{enumerate}
    \item There is a subsequence $j_k$ so that
    \begin{align*}
    \Omega_{j_k}^\pm\to \Omega_\infty^\pm \quad \mbox{and}\quad
        \Gamma_{j_k} \to \Gamma_\infty \quad\mbox{locally.} 
    \end{align*}
    \item The limit sets $\Omega_\infty^\pm$ are $2$-sided corkscrew open sets such that
    $\Omega_\infty^- =\R^2\setminus \overline{\Omega_\infty^+}$ and
    $\Gamma_\infty = \partial\Omega_\infty^+$.

    \item $\Omega_\infty^\pm$  satisfy the following: for any ball $B$ such that $\overline{B} \subset \Omega_\infty^\pm$, then a neighbourhood of $\overline B$ is contained in $\Omega_{j_k}^\pm$ for $k$  sufficiently large.
   \end{enumerate}
\end{lemma}

\begin{proof}
This result is essentially known. See, for example Theorem 4.1 in \cite{Kenig-Toro-AENS}.  However, we are not working under precisely the assumptions in \cite{Kenig-Toro-AENS}, so we provide a proof for the reader following Lemma \ref{l:convbasic}.  First, Lemma \ref{hauscomp} provides us with closed sets $G^{\pm}$ and $G_0$ and a subsequence such that $\Omega_{j_k}^{\pm}\to G^{\pm}$ and $\Gamma_{j_k}\to G_0$ locally as $k\to \infty$.  Taking the subsequence $j_k$ and the sets $G^{\pm}$ and $G_0$ provided in that lemma, we set $\Gamma_{\infty}=G_0$, $\Omega_{\infty}^{+} = G^+\backslash G^-$ and $\Omega_{\infty}^{-} = G^-\backslash G^+$.

Fix $r\in (0, \text{diam}(\Omega_{\infty})$).   Observe that $r<\liminf_{k\to \infty}\text{diam}(\Omega_{j_k})$.  If $x\in \Gamma_{\infty}$ then there is a sequence $x_{j_k}\in \Gamma_{j_k}$ with $\lim_{k\to \infty}x_{j_k} = x$.  Since $\Omega_{j_k}$ is a two-sided $c$-corkscrew domain, and $r<\text{diam}(\Omega_{j_k})$ for sufficiently large $k$, then there are $x_{j_k}^{\pm}\in \Omega_j^{\pm}$ with $|x_{j_k}-x_{j_k}^{\pm}|\leq r$ and $B(x_{j_k}^{\pm}, c_0 r)\subset \Omega_{j_k}^{\pm}$ for $k$ large enough.  Passing to a further subsequence if necessary, we may assume $\lim_{k\to \infty}x_{j_k}^{\pm} = x^{\pm}$.  But then $B(x^{\pm}, c_0 r)\subset G^{\pm}$ (for instance, any element of either of these balls can be obtain as a the limit of a sequence belonging to the respective sequences balls $B(x_{j_k}^{\pm}, c_0 r)$), and therefore $B(x^{\pm}, c_0 r)\subset \Omega_{\infty}^{\pm}$. Also notice that $|x^{\pm}-x|\leq r$.

On the other hand, Property (1) from Lemma \ref{l:convbasic}, ensures that $\R^2 = \Omega_{\infty}^+\cup \Gamma_{\infty}\cup \Omega_{\infty}^-$ and the union is disjoint, and so $$\Gamma_{\infty} = \partial \Omega_{\infty}^+ = \partial\Omega_{\infty}^-.$$

Combining our observations yields that $\Omega^+_{\infty}$ is a two-sided corkscrew open set, and additionally, $\Omega_{j_k}^\pm\to \Omega_\infty^\pm $ locally as $k\to \infty$.  Therefore property (1) of the lemma is proved.  Now property (3) follows from property (1) from Lemma \ref{lem:baspropslocconv}, since $\Gamma_{\infty} = \partial \Omega_{\infty}^+ = \partial\Omega_{\infty}^-$.

\end{proof}

We next analyze what we can say about the natural limit situation given by the conclusions of Lemma \ref{lem:basicsquare}, taking into account that the limit set $G_0 = \Gamma_{\infty}$ is the boundary of a $2$-sided corkscrew open set.

\begin{lemma}\label{lemline}
Let $\Omega^+\subset\R^2$ be a non-empty open set, and let $\Omega^- = \R^2\setminus \overline{\Omega^+}$ and $\Gamma=\partial\Omega^+$. Suppose that $\partial\Omega^- = \Gamma$ too.
Let $\mu$ be a measure with $C_0$-linear growth
supported on $\Gamma$ and let $B$ be a ball centered in $\supp\mu$. Suppose that
\begin{itemize}
\item there is an analytic variety $Z$ with $\supp(\mu_0)\cap B\subset Z\subset \Gamma$, and
\item for each $x\in B\cap \supp\mu$ and all $r\in (0,3r(B))$, there exists
two complementary half-circumferences $C^+(x,r)$, $C^-(x,r)$ with radius $r$ and center $x$ such that
$$C^+(x,r)\subset \overline{\Omega^+}\quad \text{ and }\quad C^-(x,r)\subset \overline{\Omega^-}.$$
\end{itemize}
Then $\Gamma\cap B$ is a segment.
\end{lemma}

We remark that the last property regarding the existence complementary half-circum\-ferences $C^+(x,r)$, $C^-(x,r)$ is a consequence of the existence of admissible pairs.

\begin{proof}  Since $\mu$ is non-zero and has linear growth, we have that
$\HH^1(Z) \geq \HH^1(\supp\mu)>0$. Together with the fact that $Z\neq \R^2$ ($\Omega^+$ is non-empty), this implies that there exists an
analytic curve $S$ such that $\mu(S\cap \frac14 B)>0$ (which implies that $\HH^1(S\cap\supp\mu\cap\frac14B)>0$, because of the linear growth of $\mu$).

We claim that $S$ is a segment.
To prove this, it suffices to show that $S$ has vanishing curvature at any point of $\supp\mu\cap S$. Indeed, since this set has positive length and the curvature of
a real analytic arc is locally a real analytic function (with respect the arc-length parametrization from an interval), this implies that the curvature vanishes on the whole arc $S$, and thereby prove that $S$ is a segment.

To show that the curvature of $S$
vanishes at $\supp\mu\cap S$, we will use the following property, which we will call the

\begin{KP}  Given $x\in B\cap\supp\mu$ and $r\in (0,3r(B))$, let $I\subset \partial B(x,r)$ be an arc such that
$\HH^1(I)<\pi r$ whose end-points belong both to $\Omega^+$. Then $I\subset C^+(x,r)$, and thus $I\subset \overline{\Omega^+}$.
The analogous statement holds replacing $\Omega^+$ by $\Omega^-$ and $C^+(x,r)$ by $C^-(x,r)$.
\end{KP}

To verify that the key property holds, note that if $I$ is an arc as above, then its end-points $x_1,x_2$ do not belong to $\overline{\Omega^-}$ (because they belong to
$\Omega^+$). This implies that $x_1,x_2\in C^+(x,r)$, and thus either $I$ or $\partial B(x,r)\setminus I$  is contained in $C^+(x,r)$. The latter cannot hold since $\HH^1(\partial B(x,r)\setminus I)>\pi r = \HH^1(C^+(x,r))$, and so we have
$I\subset C^+(x,r)$.

We are ready now to show that the curvature of $S$
vanishes at every $x\in\supp\mu\cap S$.
Without loss of generality we assume that $x=0$, and that the tangent to $S$ at $0$ is the horizontal axis.

Seeking for a contradiction, suppose that $S$ is strictly convex at $0$ (i.e., if $S$ equals
the graph of the real analytic function $g:(-\delta,\delta)\to\R$ in a neighborhood of $0$, then $g''(0)>0$).

Let $z_1,z_2$ be the two end-points of $S$, and let
$$d_0=\frac12\,\min_{i=1,2}\dist(x,z_i).$$
Let $r\in (0,d_0/2)$ be small enough so that
$B(x,r)\cap S \setminus\{x\}\subset \R^2_+$
where $\R^2_+$ is the {\em open} upper half plane.
We also assume that, moreover, the distance of any point from $ S\cap B(x,r)$
to the horizontal axis is at most $r/1000$.
Let $y_1\in S\cap \partial B(x,r/2)$. Since $\Gamma= \partial\Omega^+$, there exists some ball $B_1\subset\Omega^+$ satisfying
$$\dist(y_1,B_1)+r(B_1)\leq \frac1{10}\dist(y_1,\R^2_-).$$
Let $C_1$ be the circumference centered at $x$ and passing through the center of $B_1$, and let $y_2$ the point belonging to $S\cap C_1$ which is closest to $y_1$ (if $r$ is small enough, the set $S\cap C_1$
consist of two points by strict convexity).
Using now that $\Gamma=\partial\Omega^-$, there exists some ball $B_2\subset\Omega^-$ satisfying
$$\dist(y_2,B_2)+r(B_2)\leq \frac1{10}\,\min\big(\dist(y_2,\R^2_-),r(B_1)\big).$$
Now let  $C_2$ be the circumference centered at $x$ and passing through the center of $B_2$, and let $y_3$ the point belonging to $S\cap C_2$ which is farther from $y_2$ (if $r$ is small enough, the set $S\cap C_2$
consists of two points).

We distinguish now two cases. In the first one we suppose that $B_2$ is above $B_1$ (this happens if
$B_1$ is below $S$), see Figure \ref{fig:2sidedcase1}. Then, using again that $\Gamma=\partial\Omega^+$, there exists some ball $B_3\subset\Omega^+$ satisfying
\begin{equation}\label{eqb3}
\dist(y_3,B_3)+ r(B_3)\leq \frac1{10}\,\min\big(\dist(y_3,\R^2_-),r(B_2)\big).
\end{equation}
In the case that $B_2$ is below $B_1$ (which happens if
$B_1$ is above $S$), using that $\Gamma =\partial\Omega^-$, we can choose the ball $B_3$ so that $B_3\subset\Omega^-$ satisfies also \rf{eqb3}.

 \begin{figure}
\includegraphics[height=5cm]{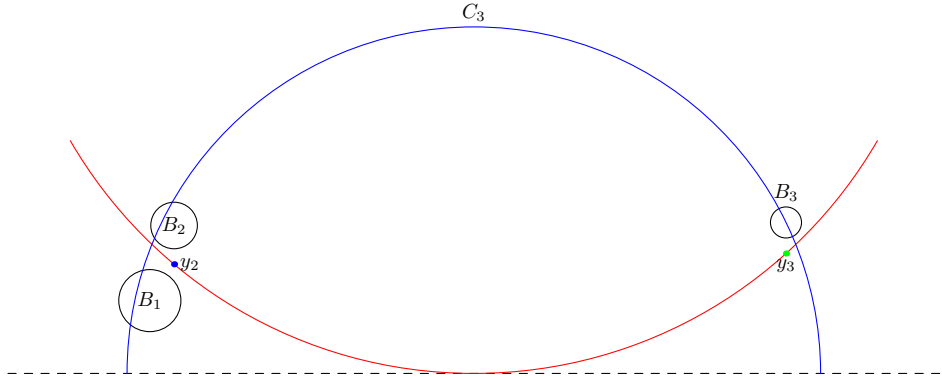}
\caption{The figure depicts the first case, where $B_1$ is below the red curve $S$.  (The ratio of the radii of the balls $B_1,B_2,B_3$ is not to scale, the reader should think of $r(B_1)\gg r(B_2)\gg r(B_3)$.)}
\label{fig:2sidedcase1}
\end{figure}

 In any case, let $C_3$ be the circumference  centered at $x$ passing through the center of $B_3$. Then it follows that
$C_3$ intersects $B_1,B_2,B_3$. Observe that, in either case $B_1,B_2,B_3\subset \R^2_+$.

In the first case, there is an arc in $C_3$ whose end-points  belong respectively to
$B_1,B_3$ (which are contained in $\Omega^+$), passes through $B_2$, and its length is smaller than $\HH^1(C_3)/2$,  due to the fact that its end-points belong to $\R^2_+$. By the Key Property, this arc is contained
in $\overline{\Omega^+}$, which is a contradiction because $B_2\subset\Omega^-$.
 In the second case we deduce that there is an arc in $C_3$ that joints $B_2$ and $B_3$ (which are contained in $\Omega^-$)
and passes through $B_1$, with length smaller than $\HH^1(C_3)/2$. By the Key Property,
the arc is contained in $\overline{\Omega^-}$.
 This is again contradiction, because $B_1\subset\Omega^+$. Hence, the curvature of $S$ at $x$ is zero.

\vv
We now appeal to the following simple fact.
\begin{lemma}\label{varietysegment} If a real analytic variety $Z\subset\R^2$ contains a segment $S$, then it also contains the line $L$ that supports the segment.
\end{lemma}

\begin{proof}[Proof of \ref{varietysegment}]  By a suitable translation and rotation we can assume that the line $L$ supporting $S$
coincides with the horizontal axis of $\R^2$. Let $\Phi:\R^2\to \R$ be a real analytic function such that $Z=\Phi^{-1}(0)$. Then,
the function defined by $\phi(x_1,x_2)=\Phi(x_1,0)$ is real analytic, and it vanishes in the interior of the
set $S\times \R$ and thus it vanishes identically in $\R^2$. That is, $\Phi(x_1,0)=0$ for all $x_1\in \R$, or, in other words, $L\subset \Phi^{-1}(0)=Z$.\end{proof}

\vv
Returning to the proof of Lemma \ref{lemline}, Lemma \ref{varietysegment} shows that $\Gamma$ contains a line $L$ such that $\mu(L\cap \frac14B)>0$.\\

Our next objective consists of showing that $\Gamma\cap B\subset L\cap B$, which will complete the proof of the lemma.  Again, without loss of generality, suppose that $L$ is the horizontal axis.\\

 Suppose that $B\cap \R^2_+\cap\Omega^+\neq\varnothing$. We intend to show that then
$B\cap \R^2_+\subset\Omega^+$.\\

For $x\in L$, consider the following semicircular extension of
 $\Omega^+\cap B(x,3r(B))\cap\R^2_+\cap$ with respect to the center $x$:
\begin{equation}\label{equ+}
U_x^+= \bigcup_{r\in(0,3r(B)): \partial B(x,r)\cap\Omega^+ \cap \R^2_+\neq\varnothing}
\big(\partial B(x,r)\cap \R^2_+\big).
\end{equation}
Observe that $U_x^+$ is also an open set.

\begin{claim}\label{extensionsemicircle} If $x\in \supp(\mu)\cap B(x, 3r(B))\cap \R^2_+$, then $$\Omega^+\cap B(x,3r(B))= U_x^+.$$\end{claim}

\begin{proof}[Proof of Claim \ref{extensionsemicircle}]  The arguments we use are similar to those required to show that the curve $S$ had vanishing curvature.  We need to show that $U_x^+\subset \Omega^+$ (recall $L\subset \partial\Omega^+$).  Assuming otherwise, there exists some point $y\in U_x^+\cap \overline{\Omega^-}$. By connectivity, then we deduce that there
exists some $r\in(0,3r(B))$ such that
$$\partial B(x,r)\cap\Omega^+\cap \R^2_+\neq \varnothing \quad \text{ and }\quad
\partial B(x,r)\cap\Gamma\cap \R^2_+\neq \varnothing.$$
Because of the existence of some point $y'\in \partial B(x,r)\cap\Gamma\cap \R^2_+$, the fact that
$\Gamma=\partial\Omega^-$,
and the openness of $U_x^+$, we deduce that there exists some ball $B_1'\subset U_x^+\cap\Omega^-$.
Let $C_1'$ be the circumference centered at $x$ passing through the center of $B_1'$.
Choose one of the two points $z\in C_1'\cap L$ so that the {\em shortest} arc in $C_1'$
that joints $z$ to $B_1'$ intersects $\Omega^+$. See Figure \ref{fig:extensionbad} below.

 \begin{figure}
\includegraphics[height=8cm]{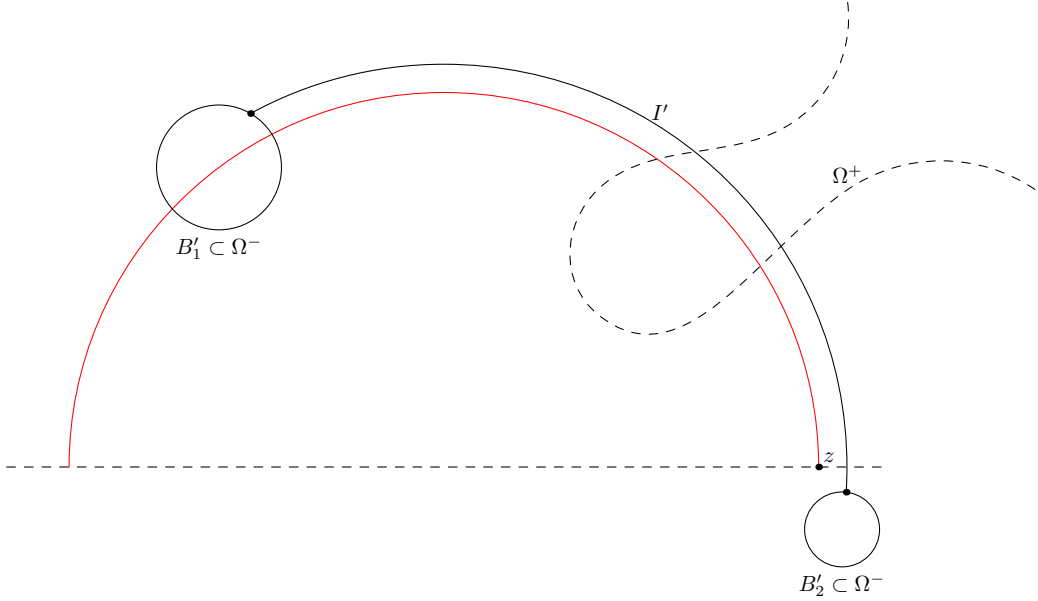}
\caption{The figure depicts geometric set up in the proof of Claim \ref{extensionsemicircle}.  In particular, observe the arc $I'$ with end-points belonging to $\Omega^-$ that intersects $\Omega^+$.}
\label{fig:extensionbad}
\end{figure}

Again by the fact that $\Gamma=\partial\Omega^-$, there exists some ball $B_2'\subset \Omega^-$ such that
$$r(B_2') + \dist(z,B_2') \leq \frac1{100}\,\min\big(r(B_1'),\dist(B_1',L)\big).$$
Let $C_2'$ be the circumference centered at $x$ passing through the center of $B_2'$. It is easy to check
that there is an arc $I'\subset C_2'$ whose end-points belong respectively to $B_1'$ and $B_2'$, such that
it intersects $\Omega^+$, and moreover has length smaller that $\frac12\HH^1(C_2')$. Since $B_1'$ and $B_2'$
are contained in $\Omega^-$, the whole $I'$ is contained in $\overline{\Omega^-}$ by the Key Property, which contradicts
the fact that $I'\cap \Omega^+\neq\varnothing$.\end{proof}

Recall that we are assuming that $B\cap \R^2_+\cap\Omega^+\neq\varnothing$ and we want to show that then
$B\cap \R^2_+\subset\Omega^+$. Suppose that this not the case. Of course, this implies that if $x_B\in \supp(\mu)\cap L$ is the centre of $B$, then $B(x_B,2r(B))
\cap \R^2_+\not\subset \Omega^+$. Let $V$ be a connected component of
$\Omega^+\cap B(x_B,2r(B))
\cap \R^2_+$. Since, by Claim \ref{extensionsemicircle}, $V$ coincides with its semicircular extension centered at $x_B$, it is of the form
$$V= A(x_B,s_1,s_2)\cap \R^2_+\quad \text{ or } \quad V=B(x_B,s_1)\cap \R^2_+,$$
with $s_1< 2r(B)$ in any case (because $V\neq B(x_B,2r(B))\cap\R^2_+$ by assumption). Let $x'\in \frac{1}{10}B\cap L\cap \supp\mu$, $x'\neq x_B$ (the existence of $x'$ is an immediate consequence of the linear growth of $\mu$).
By Claim \ref{extensionsemicircle}, the semicircular extension $U_{x'}$ centred at $x'$ is also
contained in $\Omega^+$, but then $$\Omega_+\supset U_{x'}\supset  \bigcup_{r\in(0,3r(B)): \partial B(x',r) \cap \R^2_+\cap V\neq\varnothing}
\big(\partial B(x',r)\cap \R^2_+\big)\supset \partial B(x_B,s_1) \cap \R^2_+,$$ which contradicts the definition of $V$ as a connected component of $\Omega^+$.

We have now verified that, if  $B\cap \R^2_+\cap\Omega^+\neq\varnothing$ then
$B\cap \R^2_+\subset\Omega^+$.  But by completely analogous arguments, we see that if $B\cap \R^2_+\cap\Omega^-\neq\varnothing$ then
$B\cap \R^2_+\subset\Omega^-$, and one can interchange the upper half plane with the
lower half plane. We therefore conclude that $B\cap \partial\Omega^+\subset L$, and the proof of the Lemma \ref{lemline} is complete.
\end{proof}

\vv

\begin{proof}[\bf Proof of Lemma \ref{lembeta}]
By renormalizing it suffices to prove the lemma for the ball $B_0:=B(0,1)$. We argue by contradiction: then there exists an $\ve>0$ such that for all $k \in \N$, there exists a $2$-sided $c$-corkscrew open set $\Omega_k^+$ with
$\Gamma_k:=\partial \Omega_k^+$ containing $0$ supporting a measure $\mu_k$ with $C_0$-linear growth with $\mu_k(B_0)> c_0$, so that we have
\begin{align} \label{e:Gamma-k-small}
    \int_{7B_0}\int_0^{7} \big(\ve_k(x,r)^2 + \alpha^+_{k}(x,r)^2\big) \, \dr \, d\mu_k(x) \leq \frac{1}{k}\,\mu_k(7_0),
\end{align}
and
\begin{align*}
    \beta_{\infty, \Gamma_k}(B_0) > \ve.
\end{align*}
Here we denote by $\ve_k(x,r)$ and $\alpha^+_{k}(x,r)$ the coefficients $\ve(x,r)$ and $\alpha^+(x,r)$
associated with $\Omega_k^+$.

Observe that the condition $\mu_k(B_0)\approx1$ and the linear growth of
$\mu_k$ imply that $\diam(\Omega_k^+)\geq \diam(\Gamma_k)\geq \diam(B_0\cap\supp \mu_k)\gtrsim1$.  Therefore, passing to a subsequence (which we relabel) if necessary, we may apply Lemma \ref{l:convergence} to find a $2$-sided $c$-corkscrew open sets $\Omega^{\pm}$ such that $\Gamma = \partial \Omega^{\pm}$ and
$$\lim_{k]\to \infty}\Omega_k = \Omega \text{ and }\lim_{k\to \infty}\Gamma_k=\Gamma \text{ locally as }k\to\infty.
$$
This implies that $\beta_{\infty, \Gamma}(\overline{B_0}) \geq \ve$.

Next, Lemma \ref{lineargrowthcompact} ensures that, by passing to a further subsequence if necessary, we may assume that the  measures $\mu_k$ converge weakly to a measure $\mu$, supported on $\Gamma$, with $C_0$-linear growth and $\mu(\overline{B_0})\geq c_0$.

We now apply Lemma \ref{lem:basicsquare}.  Therefore, there is an analytic variety $Z$ such that $7B_0\cap \supp(\mu)\subset Z\subset \Gamma$, and for every $x\in 7B_0\cap \supp(\mu)$ and $r\in (0,7)$
\begin{equation}\begin{split}\label{2sidedproofemicirc}
&\text{there are complementary semicircumferences }(C^+, C^-)\text{ centred at }x\\
&\text{with radius }r\text{ satisfying }C^{\pm}\subset \Omega^{\pm}
\end{split}\end{equation}
Since $\mu(\overline{B_0})\geq c_0$, we can now find a ball $B'$ centred on $\supp(\mu)\cap\overline{B_0}$ such that $7B_0\supset B'\supset \overline{B_0}$ such that $\supp(\mu)\cap B'\subset Z$ and (\ref{2sidedproofemicirc}) holds for every $x\in \supp(\mu)\cap B'$ and $r\in (0, 3r(B'))$. We now apply Lemma \ref{lemline} with the ball $B'$ to conclude that
  $\Gamma_\infty\cap B'$ (and so $\Gamma\cap \overline{B_0}$) is a segment.  This, however, contradicts the fact that $\beta_{\infty, \Gamma}(\overline{B_0}) \geq \ve$.
\end{proof}

\vv


\section{The case of Jordan domains}

In this section we shall prove Main Lemma \ref{l:main1} in the case of a Jordan domain, which we restate for the benefit of the reader.


\begin{lemma}\label{lembetajordan}
Let $\Omega^+\subset\R^2$ be a Jordan domain, let $\Gamma=\partial\Omega^+$, and let $\mu$ be a measure with $C_0$-linear growth
supported on $\Gamma$. Let $B$ be a ball centered at $\Gamma$ such that
$$\mu(B)\geq c_0 r(B),$$
for some $c_0 \in (0, C_0)$.
Given any $\ve>0$, there exists $\delta>0$ (depending on $C_0,c_0,\ve$) such that if
$$\int_{7B}\int_0^{7r(B)} \big(\ve(x,r)^2 + \alpha^+(x,r)^2\big)\,d\mu(x)\frac{dr}r \leq \delta\,\mu(7B),$$
then
$$\beta_{\infty,\Gamma}(B)\leq \ve.$$
\end{lemma}

The first auxiliary result we need is the following, which states that, at points where the Carleson square function is sufficiently small, we may find corkscrew balls.

\begin{lemma}\label{lem2cork}
Let $\Omega^+\subset\R^2$ be a Jordan domain. Let $x\in\Gamma=\partial\Omega^+$, $r>0$, and $x'\in \Gamma\cap \partial B(x,r)$. Suppose that
$$\int_0^{2r}\big(\ve(x,t)^2 + \ve(x',t)^2\big)\,\frac{dt}t \leq \delta,$$
for some $\delta>0$. If $\delta$ is small enough, then there are two balls $B^\pm\subset B(x,r)\cap \Omega^\pm$ such that $r(B^+)\approx r(B^-)\approx r$, where the implicit constants are absolute.
\end{lemma}

\begin{proof}

Without loss of generality, we may assume that $x=0$, $r=1$, and $x'$ lies on the horizontal axis.  It will be convenient to work with rectangles in polar coordinates.  For intervals $I\subset (0,\infty)$ and $P\subset [-\pi,\pi]$, define
$$\Sect(I,P) = \{se^{i\theta}: s\in I, \theta\in P\}.
$$
We call such a set a \emph{polar rectangle}.  Note that, if $I\subset [0,1]$, then we can inscribe a ball inside $\Sect(I,P)$ with radius a constant multiple of $\ell(I)\ell(P)\approx \sqrt{m_2(\Sect(I,P))}$.

We begin with a claim:

\begin{claim}\label{emptyrectangles}  There is an absolute constant $c>0$ such that the following holds:  For intervals $I\subset [1/2,1]$ and $P\subset [\pi/4,3\pi/4]\cup [-3\pi/4, -\pi/4]$, if $\delta$ is sufficiently small (depending on $\ell(I)$ and $\ell(P)$), then there exists a polar rectangle $\Sect'\subset \Sect(I,P)$ such that $m_2(\Sect')\geq cm_2(\Sect)$ and either $\Sect'\subset \Omega^+$ or $\Sect'\subset \Omega^-$.\end{claim}

Let us first show how to prove the lemma using Claim \ref{emptyrectangles}.  First, take $I=[1/2,1]$ and $P=[\pi/4, 3\pi/4]$.  Then we get a polar rectangle $\Sect' = \Sect(I',P')\subset \Sect(I,P)$ with $m_2(X')\gtrsim 1$, and such that $\Sect'\subset \Omega^{\pm}$, provided $\delta$ is small enough. For definiteness let us assume that $\Sect'\subset \Omega^+$.  We then apply the claim again with $I$ replaced by $I'$ and $P$ replaced by $-(\tfrac{1}{3}P') = \{-\theta:\theta\in \frac{1}{3}P'\}$ (here for an interval $P$, $aP$ is the concentric interval of sidelength $a\ell(P)$).  As long as $\delta$ is small enough, there is a polar rectangle $\Sect''=\Sect(I'', P'')\subset \Sect(I',-(\tfrac{1}{3}P'))$ with $\Sect''\subset \Omega^{\pm}$, and $m_2(\Sect'')\gtrsim 1$.  We need to verify that $\Sect''\subset \Omega^-$.

However, if $\Sect''\subset \Omega^+$, then we would have that every circumference $C(0,s)$, with $s\in I''$, has its intersection with $\Sect''$ or $\Sect'$ contained in $\Omega^+$.
But $C(0,s)
\backslash(\Sect'\cup\Sect'')$ is comprised of two arcs with length  at most equal to $(\pi-\ell(P''))s$. Thus $\ve(0, s)\gtrsim 1$ for all $s\in I''$, whence $\int_0^1 \ve(0, s)\frac{ds}{s}\gtrsim 1$.  We have therefore arrived at a contradiction if $\delta$ is sufficiently small.\end{proof}

We now return to verify the claim.

\begin{proof}[Proof of Claim \ref{emptyrectangles}]We may assume that $X=X(I,P)\subset \R_+^2$, the upper half-plane (i.e., $P\subset[\pi/4,3\pi/4]$).  First split $X(I,P)$ into $1000$ polar rectangles $\Sect_j=\Sect(I, P_j)$ with $\ell(P_j) = \tfrac{1}{1000}\ell(P)$. Write $I=[r_1,r_2]$.  Fix $\kap>0$, and consider the circumferences
$$C_s = \partial B(0, s)\text{ for } s\in ((1-\kap) r_2 +\kap r_1, r_2).
$$
If $\delta$ is sufficiently small, 99\% of these circumferences intersect $\Gamma$ in at most $4$ of the polar rectangles $\Sect_j$.  In this case, we call $C_s$ \emph{good}.

Next, for each polar rectangle $\Sect_j$, consider
$$m_j = \mathcal{H}^1\bigl(\{s\in [(1-\kap)r_2+\kap r_1, r_2]: C_s\text{ is good and }C_s\cap \Sect_j\cap \Gamma\neq \varnothing\}\bigl).
$$
Fubini's theorem yields
$$\sum_j m_j \leq 4\kap(r_2-r_1),
$$
and so there exists $j_0$ with $m_{j_0}\leq \frac{4\kap}{1000}(r_2-r_1).$

Consequently,
\begin{equation}\begin{split}\nonumber\mathcal{H}^1&\bigl(\{s\in [(1-\kap)r_2+\kap r_1, r_2]: C_s\cap \Sect_{j_0}\cap \Gamma\neq \varnothing\}\bigl)\\
&\leq m_{j_0}+\mathcal{H}^1(\{ s\in ((1-\kap) r_2 +\kap r_1, r_2): C_s\text{ is not good}\})\\
&\leq \frac{4\kap}{1000}(r_2-r_1)+ \frac{1}{100}\kap(r_2-r_1)\leq \frac{1}{50}\kap(r_2-r_1),
\end{split}\end{equation}
and we conclude that at most only 2\% of the circumferences $C_s$, $s\in ((1-\kap) r_2 +\kap r_1, r_2)$, intersect $\Gamma$ in $\Sect_{j_0}=\Sect(I, P_{j_0})$.

Using the pigeonhole principle, we infer that we can find three pairwise disjoint intervals $I_1, I_2$ and $I_3$ in $[(1-\kap) r_2 +\kap r_1, r_2]$, such that
\begin{itemize}
\item $\ell(I_k)\gtrsim \kap (r_2-r_1)$,
\item $\dist(I_j, I_k)\gtrsim \kap(r_2-r_1)$ if $j\neq k$, and
\item $C(0,s)\cap \Sect_{j_0,k}\cap\Gamma=\varnothing$ whenever $s\in \partial I_k$ for $k=1,2,3$.
\end{itemize}

Consider the three polar rectangles $\Sect_{j_0, k} = \Sect(I_k, P_{j_0})$, which certainly contain $\wt\Sect_{j_0, k} = \Sect(I_k, \wt P_{j_0})$ with $\wt P_{j_0} = \tfrac{1}{10}P_{j_0}$.  We will show that one of the rectangles $\wt \Sect_{j_0,k}$, for some $k=1,2,3$, does not intersect $\Gamma$.

Let us write $\Gamma = \gamma([0,1])$ with $\gamma(0)=0=\gamma(1)$.  First suppose $\Gamma\cap \wt \Sect_{j_0,k}\neq \varnothing$ for some $k\in \{1,2,3\}$.   If we consider $u_0$ such that
$\gamma(u_0)\in \wt \Sect_{j_0,k}$ and
$u_1
 = \max\{u:\gamma([u_0,u])\subset \Sect_{j_0,k}\}$, then since $C(0,s)\cap \Sect_{j_0,k}\cap\Gamma=\varnothing$ for $s\in \partial I_k$, and $0=\gamma(0)=\gamma(1)\notin \Sect_{j_0,k}$, we must have that $\gamma(u_1)\in \{se^{i\theta}: s\in I_k, \theta\in \partial P_{j_0}\}$.  We say that $\Gamma$ goes to the right (left resp.) if $\gamma(u_1)$ lies on the right (left) side boundary of $\Sect_{j_0,k}$.  Assuming that $\wt \Sect_{j_0,k}\cap \Gamma\neq \varnothing$  for every $k=1,2,3$, we therefore see that $\Gamma$ must go to one direction (either left or right) in two of the rectangles, say $\Sect_{j_0,k_1}$ and  $\Sect_{j_0, k_2}$ -- for definiteness let us say the direction is right (analogues arguments handle the other case).

\begin{figure}
\includegraphics[height=5cm, trim={1cm 4cm 1.5cm 3cm},clip]{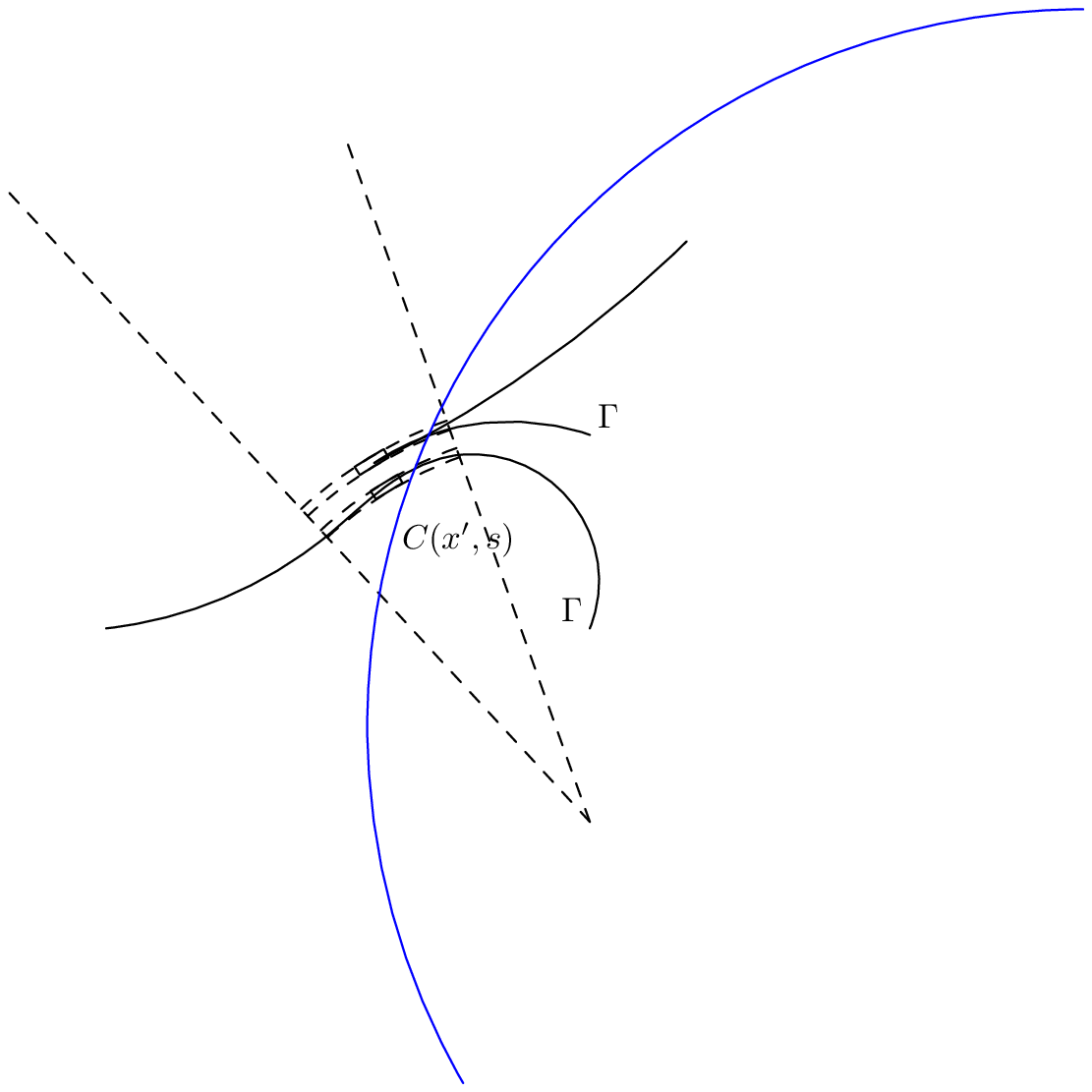}
\caption{The figure depicts a circumference $C(x',s)$ crossing $\Sect_{j_0, k_1}$ to the right of $\wt\Sect_{j_0,k_1}$ and crossing $\Sect_{j_0, k_2}$ to the right of $\wt\Sect_{j_0,k_2}$.}
\label{fig:polarrectangles1}
\end{figure}

If we fix $\kap = 10^{-6}\ell(P_j)$, say, then there is an interval $J$ with $\ell(J)\gtrsim \kap\ell(I)$ so that for every $s\in J$, the circumference $C(x',s)$ crosses $\Sect_{j_0,k_1}$ to the right of $\wt\Sect_{j_0,k_1}$ and also crosses $\Sect_{j_0, k_2}$ to the right of $\wt\Sect_{j_0,k_2}$.  Therefore, insofar as $\Gamma$ goes to the right in both $\Sect_{j_0, k_1}$ and $\Sect_{j_0, k_2}$, a circumference $C(x',s)$ with $s\in J$ intersects $\Gamma$ in the well-separated polar rectangles $\Sect_{j_0, k_1}$ and $\Sect_{j_0, k_2}$ (see Figure \ref{fig:polarrectangles1}), and so $\ve(x',s)\gtrsim_{\kap, \ell(I)} 1$.  But then $\int_0^2 \ve(x',s)^2\frac{ds}{s} \gtrsim_{\kap, \ell(I)} 1$.  If $\delta$ is small enough then we have reached a contradiction.
\end{proof}

We begin by reviewing Lemma \ref{l:convbasic} in the context of a sequences of Jordan domains.

\begin{lemma} \label{l:convjordan}
Let $\{\Omega_j^+\}_j$ be a sequence of Jordan domains in the plane  such that $0\in\partial\Omega_j^+$. Let $\Omega_j^-=\R^2\setminus \overline{\Omega_j^+}$ and $\Gamma_j=
\partial\Omega_j^+$.
 Then the following holds:
\begin{enumerate}
    \item There is a subsequence of domains $\Omega_{j_k}^\pm$ and there are closed sets $G^+,G^-,G_0$ such that
    \begin{align*}
    \overline{\Omega_{j_k}^\pm}\to G^\pm \quad \mbox{and}\quad
        \Gamma_{j_k} \to G_0 \quad\mbox{locally.}
    \end{align*}
    \item The limit sets $G^+,G^-,G_0$  satisfy
    $$G^+\cup G^- =\R^2,\quad \;G^+\cap G^-=G_0.$$
 In particular, $G^+\setminus G_0$ and $G^-\setminus G_0$ are open.

   \end{enumerate}
\end{lemma}

\begin{proof}
The existence of the locally convergent subsequences follows from Lemma \ref{hauscomp}, the property (2) is then a consequence of Lemma \ref{l:convbasic}.
\end{proof}

\vv

We remark that, in the above situation, $G_0$ need not coincide with $\partial G^+$ or $\partial G^-$.
Further, $G_0$ may have non-empty interior, and $G^\pm\setminus G_0$ may be empty.

Our next lemma reviews the basic convergence result Lemma \ref{lem:basicsquare}, also taking into account Lemma \ref{lem2cork}.

\begin{lemma}\label{l:convjordan2}
Let $\{\Omega_j^+\}_j$ be a sequence of Jordan domains in the plane which intersect some ball $B_0$. Let $\Omega_j^-=\R^2\setminus \overline{\Omega_j^+}$ and $\Gamma_j=
\partial\Omega_j^+$. Suppose $G^{\pm}$ and $G_0$ are closed sets with
$$\overline{\Omega_j^{\pm}}\to G^{\pm}\text{ and }\Gamma_j \to G_0\text{ locally as }j\to\infty.
$$
Suppose $\mu_j$ are measures supported on $\Gamma_j$ with $C_0$-linear growth that converge weakly to a measure $\mu_0$ satisfying $\mu_0(\overline{B_0})\geq c_0r(B_0)$.
Suppose that (\ref{eq:G1}) holds, i.e.
\begin{align} \nonumber 
    \int_{7B_0}\int_0^{7r(B_0)}  \alpha^+_{j}(x,r)^2 \, \dr \, d\mu_j(x) \leq \frac{1}{j}\,\mu_j(7B_0),
\end{align}
where $\alpha^+_{j}$ are the coefficients $\alpha^+$ associated with $\Omega_j^+$, and (\ref{eq:G2}) holds, i.e.,
\begin{align} \nonumber
    \int_{7B_0}\int_0^{7r(B_0)}  \ve_{j}(x,r)^2 \, \dr \, d\mu_j(x) \leq \frac{1}{j}\,\mu_j(7B_0),
\end{align}
where $\ve_{j}(\cdot,\cdot)$ are the coefficients $\ve(\cdot,\cdot)$
associated with $\Omega_j^+$.
Then
\begin{enumerate}
\item there is an analytic variety $Z$ with $7B_0\cap \supp(\mu_0)\subset Z\subset G_0$.
\item\label{jordanadmissible} for all $x\in 7 B_0\cap \supp\mu_0$ and all $r\in (0,7r(B_0))$ there is a pair of admissible semicircumferences which are contained in $\partial B(x,r)$.
\item \label{holesbymeasure} for every $M>0$, there exists a constant $c(M)>0$ such that whenever $x\in B_0\cap \supp(\mu_0)$ and $r\in (0,r(B_0))$ are such that $\mu_0(B(x,r))\geq r/M$, then there are two balls
$B^\pm\subset B(x,2r)\cap G^\pm\setminus G_0$ with $r(B^\pm)\geq c(M)r$.
 \end{enumerate}
\end{lemma}

\begin{proof}
The proof of the first two statements are precisely those of Lemma \ref{lem:basicsquare}.  The third assertion is proved by passing to the limit in the result in Lemma \ref{lem2cork}.  Indeed, fix $r' = \frac{r}{3MC_0}$.  Then for sufficiently large $j$, $\mu_j(B(x,r)\backslash B(x,r'))\geq \frac{r}{2M}$.  Since $x\in \supp(\mu)$, for any $s\in (0, \frac{r'}{2})$ we have $\liminf_{j\to\infty}\mu_j(B(x,s))>0$, whence
$$\frac{1}{\mu_j(B(0,s))}\int_{B(x,s)}\int_0^{7r} \ve^+_{j}(y,r)^2 \, \dr \, d\mu_j(y) \leq \frac{1}{j\mu_j(B(0,s))}\,\mu_j(7B_0) \to 0 \text{ as }j\to \infty.
$$
Consequently, for $\delta>0$ as in Lemma \ref{lem2cork}, and for sufficiently large $j$ we can find $z_j\in B(x,s)\cap \supp(\mu_j)$ with $\int_0^{7r} \ve^+_{j}(z_j,r)^2 \, \dr< \delta.$   But now, as $j\to\infty$,
$$\frac{1}{\mu_j(B(x,r)\backslash B(x,r'))}\int_{B(x,r)\backslash B(x,r')}\int_0^{7r} \ve^+_{j}(y,r)^2 \, \dr \, d\mu_j(y) \leq \frac{2M}{jr}\mu_j(7B_0)\to 0,$$
so for large $j$ we can find $z_j'\in \supp(\mu_j)\cap B(x,r)\backslash B(x,r')$ with $\int_0^{7r} \ve^+_{j}(z_j',r)^2 \, \dr< \delta.$  Notice that $z_j'\in \partial B(z_j,t_j)$ with $t_j\in (r'/2, \tfrac{3}{2}r)$.  We apply Lemma \ref{lem2cork} with the points $z_j$ and $z_j'$ and radius $t_j$ (note that $2t_j\leq 7r$) to find balls $B_j^{\pm}\in \Omega_{j}^{\pm}\cap B(z_j, \tfrac{3}{2}r)$ such that $r(B^{\pm}_j) \gtrsim \frac{r}{M}.$  If $s$ is small enough, $B_j^{\pm}\subset \overline{B(x, 3r/2)}$ and we may pass to a subsequence  $B_{j_k}^{\pm}$ which converge in Hausdorff distance to balls $B^{\pm}\subset G^{\pm}$ with $B^{\pm}\subset B(x,2r)$.  But then if $y\in B^{\pm}$ (say $y\in B^+$ for definiteness), then a neighbourhood of $y$ is contained in $B_{j_k}^{+}$ for sufficiently large $k$, and so $\liminf_k\dist(x, \Gamma_{j_k})>0$, which ensures that $y\in G^+\backslash G_0$.  Thus $B^{\pm}\subset G^{\pm}\backslash G_0$.
\end{proof}


\vv

Our next result is an analogue of Lemma \ref{lemline}.  The reader should notice that the conclusion is weaker.  This is due to the fact that we only can infer anything about the structure of the boundary set $G_0$ at points where $\mu$ has lots of mass (via property (\ref{holesbymeasure}) of Lemma \ref{l:convjordan2}).

\begin{lemma}\label{l:convjordan4}Suppose $G^+$, $G^-$ and $G_0$ are three closed sets satisfying $G^+\cup G^-=\R^2$ and $G^+\cap G^-= G_0$.  Suppose that $\mu_0$ is a measure with $C_0$-linear growth and that there is a real analytic variety $Z$ with $\supp(\mu_0)\subset Z\subset G_0$.  Let $B_0\subset\R^2$ be some ball such that $\mu(B_0)>0$ and
\begin{enumerate}
\item \label{holesbymeasure2} for every $M>0$, there exists a constant $c(M)>0$ such that whenever $x\in B_0\cap \supp(\mu_0)$ and $r\in (0, r(B_0))$ are such that $\mu_0(B(x,r))\geq r/M$, then there are two balls
$B^\pm\subset B(x,2r)\cap G^\pm\setminus G_0$ with $r(B^\pm)\geq c(M)r$.
\item \label{jordansemicirc}for every $x\in 2B_0$ and $r\in (0, 2r(B_0))$ there exists two complementary semicircumferences $C^+(x,r)$, $C^-(x,r)$ with radius $r$ and center $x$ such that
$$C^+(x,r)\subset G^+\quad \text{ and }\quad C^-(x,r)\subset G^-.$$
\end{enumerate}
Then there is some line $L$ such that
$$L\subset G_0\quad \text{ and }\quad \mu_0(B_0\cap L)>0.$$
\end{lemma}

There is some natural repetition in the proof of Lemma \ref{l:convjordan4} and Lemma \ref{lemline}, but since the proofs are also quite substantially different, and some readers may want to only consider the case of Jordan domains, we repeat all the relevant details here.

\begin{proof}
Since $\mu_0$ is non-zero and has linear growth, it is clear that
$\HH^1(Z) \geq \HH^1(\supp\mu_0)>0$. Together with the fact that $Z\neq \R^2$ (which follows from property (\ref{holesbymeasure2}), since $\mu(B_0)>0$), this implies that there exists some
analytic arc $S$ such that $\mu_0(S\cap B_0)>0$ (which implies that $\HH^1(S\cap\supp\mu\cap B_0)>0$, because of the linear growth of $\mu_0$). \\

We claim that $S$ is a segment.
To prove this it suffices to show that $S$ has vanishing curvature in a set positive measure $\mu_0$. Indeed, since this set has positive length and the curvature of
a real analytic arc is locally a real analytic function (with respect the arc-length parametrization from an interval), this implies that the curvature vanishes on the whole arc $S$ . Thus $S$ is a segment.\\

To show that the curvature of $S$
vanishes in some set of positive measure $\mu_0$, we will again use the

\begin{KP} Given $x\in B\cap\supp\mu_0$ and $r\in (0,3r(B_0))$, let $I\subset \partial B(x,r)$ be an arc such that
$\HH^1(I)<\pi r$ whose end-points belong both to $G^+\setminus G_0$. Then  $I\subset G^+$.
The analogous statement holds replacing $G^+$ by $G^-$.
\end{KP}

We verify the Key Property as follows: if $I$ is an arc as above, then its end-points $x_1,x_2$ do not belong to $G^-$ (because they belong to
$G^+\setminus G_0$). Hence, if $(C^+,C^-)$ is a pair of complementary semicircumferences at $x$ with radius $r$ satisfying $C^{\pm}\subset G^{\pm}$, we have
 that $x_1,x_2\in C^+$, and thus either $I$ or $\partial B(x,r)\setminus I$  is contained in $C^+$. The latter cannot happen since $\HH^1(\partial B(x,r)\setminus I)>\pi r = \HH^1(C^+)$, and so we have
$I\subset C^+\subset G^+$.\\

We are ready now to show that the curvature of $S$
vanishes at some set of positive measure $\mu_0$.
We consider some set $F\subset S$ such that
$\HH^1(F)>0$ and
 $\mu_0|_F = h\,\HH^1|_F$ for some function $h\approx1$ (with the implicit constant possibly depending on $S$, $F$, and other parameters).
Without loss of generality we assume that $x_0=0$ is a density point of $F$ and that the tangent to $S$  at $x_0$ is the horizontal axis. Aiming for a contradiction, suppose that $S$ is strictly convex at $x_0$ (i.e., if $S$ equals
the graph of the real analytic function $g:(-\delta,\delta)\to\R$ in a neighborhood of $x_0$, then $g''(0)>0$).

Let $z_1,z_2$ be the two end-points of $S$, and let
$$d_0=\frac12\,\min_{i=1,2}\dist(x_0,z_i).$$
Let $r\in (0,d_0/2)$ be small enough so that $g''(\Pi_H(x))$ is comparable to $g''(\Pi_H(x_0))$ in
$B(x_0,r)\cap S$, where $\Pi_H$ is the orthogonal projection on the horizontal axis.
We will prove the following:

\begin{claim}\label{claim1}
There exist some $z\in B(x_0,r/10)\cap S\cap \supp\mu_0$ and some $r'\in (0,r/10)$ and an arc $I\subset \partial B(z,r')$ with $\HH^1(I)< \pi  r'$ such that either
its end-points belong to $G^+\setminus G_0$ and intersects $G^-\setminus G_0$, or its end-points belong to $G^-\setminus G_0$ and intersects $G^+\setminus G_0$.
\end{claim}

The preceding claim asserts that the strict convexity of $g$ at $x_0$ implies that the Key Property is violated. Hence $S$ is a segment.  Lemma \ref{varietysegment} then ensures that $Z$ also contains the line $L$ that supports the segment, thereby completing the proof of the lemma (up to verification of the claim).
\end{proof}

\vv
\begin{proof}[Proof of Claim \ref{claim1}]
Since $x_0$ is a density point of $F$ in $S$, we can take some $t\in (0,r/10)$
such that $\HH^1(F\cap B(x_0,t))\geq (1-\tau)\,\HH^1(S\cap B(x_0,t))$, where $\tau\in(0,10^{-3})$ is some small parameter to be fixed below.

Denote by $L_H$ the horizontal axis and let $J=(-t/2,t/2)\subset L_H$, so that $S\cap B(x_0,t)\supset
g(J)$.

We will appeal to the following simple lemma.

\begin{lemma}\label{lem:denstau}
Fix $\kap > \tau$, and suppose that $\mathcal{F}$ is a finite family of pairwise disjoint intervals contained in $(-t/2,t/2)$ satisfying
$$\sum_{T\in \mathcal{F}}\mathcal{H}^1(g(T))\geq 3\kap \mathcal{H}^1(S\cap B(x_0,t)).
$$
If $\mathcal{F}'$ denotes the subfamily of intervals $T\in \mathcal{F}$ satisfying
$\mathcal{H}^1(g(T)\cap F)\geq \kap \mathcal{H}^1(g(T))$, then
$$\sum_{T\in \mathcal{F}'}\mathcal{H}^1(g(T))\cap F)\geq \kap \mathcal{H}^1(S\cap B(x_0,t)).
$$
\end{lemma}

\begin{proof}  Suppose the conclusion fails.  Then insofar as $\mathcal{F}$ are pairwise disjoint intervals, and $\mathcal{H}^1(g(T)\cap F)< \kap \mathcal{H}^1(g(T))$ for $T\in \mathcal{F}\backslash \mathcal{F}'$, we have
$$\sum_{T\in \mathcal{F}}\mathcal{H}^1(g(T)\cap F)\leq 2\kap \mathcal{H}^1(S\cap B(x_0,t)).
$$
But then
$$\mathcal{H}^1(B(x_0,t)\cap F)\leq 2\kap \mathcal{H}^1(S\cap B(x_0,t))+\mathcal{H}^1\Bigl(S\backslash \bigcup_{T\in \mathcal{F}}g(T)\Bigl)\leq (1-\kap)\mathcal{H}^1(S\cap B(x_0,t)).
$$
The right hand side is strictly smaller than $(1-\tau)\mathcal{H}^1(S\cap B(x_0,t))$, which is our desired contradiction.
\end{proof}

We split $J$ into three intervals $J_l, J_c, J_r$ (where $l, c, r$, stand for left, center, right) with
disjoint interiors such that $\HH^1(J_l)=\HH^1(J_c)=\HH^1(J_r)=\HH^1(J)/3$.
Next we split $J_l$ into $N$ intervals with disjoint interiors of the same length, and we take
$N = c_1\HH^1(J)^{-1}$, with $c_1\in (0,1)$ to be chosen below (depending on $g''(0)$). We denote by $J_l^1,\ldots J_l^N$
this family of intervals. By standard arguments, we find a subfamily $\{J_l^k\}_{k\in K_l}$ of $\{J_l^1,\ldots, J_l^N\}$ such that the intervals $\{10J_l^k\}_{k\in K_l}$ are pairwise disjoint and moreover
$$\sum_{k\in K_l} \HH^1(g(J_{l}^k))\gtrsim \HH^1(g(J_l)) \gtrsim \HH^1(S\cap B(x_0, t)).$$
We may therefore apply Lemma \ref{lem:denstau}, with $\kap$ some absolute constant (provided $\tau$ is small enough), to find a subfamily $\{J_l^k\}_{k\in H_l} \subset \{J_l^k\}_{k\in K_l}$ of the intervals
$J_l^k$
such that \begin{equation}\label{j1kbigf}\HH^1(F\cap g(J_l^k))\approx \HH^1(g(J_l^k)),\end{equation}
and
\begin{equation}\label{eqsjh3}
\sum_{k\in H_l} \HH^1(F\cap g(J_{l}^k))\gtrsim \HH^1(S \cap B(x_0, t)).
\end{equation}
Next, note that the condition (\ref{j1kbigf}) ensures that we can apply property (\ref{holesbymeasure2}) for each $k\in H_l$ to find a ball $B^+_k$ satisfying
$$B^+_k\subset U_{\ell(J_{l}^k)}(g(J_l^k))\cap\Omega^+,\qquad
\text{ with }r(B_k^+)\approx \HH^1(J_{l}^k),$$
where $U_\ell(A)$ stands for the $\ell$-neighborhood of $A$.
Observe that, by the strict convexity of $S$, we have $\dist(g(J_l^k),L_H)\approx\ell(J)^2$. On the other hand,
$$\dist(B_k^+,g(J_l^k))\leq \ell(J_l^k)\leq \frac{\ell(J)}N =c_1\ell(J)^2.$$
So if we choose $c_1$ small enough, then the balls $B_k^+$ are contained in $\R^2_+$ and far from $L_H$.

Next, let $I_k$, $k\in H_l$, be the projection of the balls $B_k^+$, $k\in H_l$, on the axis $L_H$. The intervals $I_k$, $k\in H_{l}$, are disjoint, and moreover, $$\sum_{k\in H_{l}}\mathcal{H}^1(g(\tfrac{1}{10}I_j))\gtrsim \sum_{k\in H_{l}}\mathcal{H}^1(g(I_j))\stackrel{(\ref{eqsjh3})}{\gtrsim} \mathcal{H}^1(S\cap B(x_0,t)).$$
Therefore, appealing to Lemma \ref{lem:denstau} once again, with $\kap$ some absolute constant, we find a family of indices $M_l\subset H_l$ such that
\begin{equation}\label{eq:onetenthbigF}\HH^1(F\cap g(\tfrac1{10} I_k))\approx \HH^1(g(\tfrac{1}{10}I_k))\approx \ell(I_k)\text{ for every }k\in M_l\end{equation} and
$$\sum_{k\in M_l} \HH^1(F\cap g(\tfrac1{10} I_k))\approx \HH^1(S \cap B(x_0,t)).$$

Since (\ref{eq:onetenthbigF}) holds, we may apply property (\ref{holesbymeasure2})
for each $k\in M_l$ to find
 some ball $B^-_k$ satisfying
$$B^-_k\subset U_{\tfrac1{10} \ell(I_k)}(g(\tfrac1{10} I_k))\cap\Omega^+,\qquad
r(B_k^-)\approx \HH^1(I_k).$$
Again, by the strict convexity of $S$, the balls $B^-_k$ are contained in $\R^2_+$ and far away from
$L_H$. Further, by construction the projection $\Pi_H(B^-_k)$ is contained deep inside $\Pi_H(B^+_k)$ for
each $k\in M_l$. In fact, by shrinking the balls $B^-_k$ if necessary, we can assume that
$$\Pi_H(B^-_k)\subset \Pi_H(\tfrac12 B^+_k)\quad \mbox{ for each $k\in M_l$.}$$

Now we denote
$$W_l = \bigcup_{k\in M_l} \Pi_H(\tfrac12B^-_k).$$
By the disjointness of the intervals $10J^k_l$, $k\in M_l$, the intervals $\Pi_H(\tfrac12B^-_k)$
are disjoint and we deduce that $\HH^1(W_l)\approx \HH^1(J)\approx t.$

Next we define an analogous family of balls $\{B^\pm_k\}_{k\in M_r}$ and a set $W_r$, replacing the left
interval $J_l$ by the right one $J_r$.

We claim that there is some $x\in J_c\cap \Pi_H(F)$ such that
$$W_l \cap (2x-W_r)\neq \varnothing.$$
In fact, for an arbitrary point $y_r\in W_r$,
the set $\{2x - y_r:x\in J_c\cap \Pi_H(F)\}$ is of the form $I\setminus X$, where $I$ is an interval
 of length $2\ell(J_c)$ which contains $J_l$ and $X$ is an exceptional set with length at most
 $2\HH^1(J\setminus \Pi_H(F))\leq c\HH^1(g(J)\setminus F)\leq
 c\tau\ell(J)$. So for $\tau$ small enough, $\{2x - y_r:x\in J_c\cap \Pi_H(F)\}$
 intersects $W_l$, since $\HH^1(W_l)\approx \ell(J)\gg c\tau\ell(J)$.

The preceding argument shows that
 there exist $y_l\in W_l$, $y_r\in W_r$, and $x\in J_c\cap \Pi_H(F)$ such that $y_l = 2x-y_r$, or equivalently,
$$x= \frac{y_l+y_r}2.$$
Observe that, in particular, this implies that $|x-y_l|= |x-y_r|\approx\ell(J)$.

Let $k\in M_l$ be such that $y_l\in\Pi_H(\tfrac12B^-_k)$ and $h\in M_r$  such that
$y_r\in\Pi_H(\tfrac12B^-_{h})$. By construction, there are points $y_l^\pm\in \frac12B_k^\pm$ and
$y_r^\pm\in \frac12B_h^\pm$ such that
$$\Pi_H(y_l^-)= \Pi_H(y_l^+) = y_l\quad \text{and}\quad \Pi_H(y_r^-)= \Pi_H(y_r^+) = y_r.$$
We claim that the circumference centered at $g(x)$ (observe that $g(x)\in F\subset S\cap\supp\mu_0$) with radius $|x-y_l|$ intersects the four balls $B_k^\pm$ and $B_h^\pm$. To see this, notice that
$$\big||g(x)- y_l^\pm| - |x-y_l|\big| \leq (1-\cos\alpha^\pm)\,\ell(J)
\lesssim (\alpha^\pm)^2\,\ell(J),$$
where $\alpha^\pm$ is the slope of the line passing through $g(x)$ and $y_l^\pm$, which satisfies $\alpha^\pm\lesssim \ell(J)$ (taking into account that $|x-y_l|= |x-y_r|\approx\ell(J)$ and the quadratic behavior of $g$ close to $x_0=0$). Thus,
$$\big||g(x)- y_l^\pm| - |x-y_l|\big| \lesssim \ell(J)^3\ll \ell(J)^2\approx r(B_k^\pm)
.$$
Analogously,
$$\big||g(x)- y_r^\pm| - |x-y_l|\big| =
\big||g(x)- y_r^\pm| - |x-y_r|\big| \lesssim \ell(J)^3\ll \ell(J)^2\approx r(B_h^\pm).$$
Hence the aforementioned circumference passes through the balls $B_k^\pm$, $B_h^\pm$.

Let $z=g(x)$ and $r'=|x-y_l|$.
It is easy to check that there is an arc contained in the circumference $\partial B(z,r')$ satisfying the
required properties in the claim. To see this, let $H_z$ the open half-plane whose boundary equals the tangent to $S$ at $z$ and containing $S\setminus \{z\}$. It is easy to check that the four balls
$B_k^\pm$, $B_h^\pm$ are contained in $H_z$, taking into account that $g''(\xi)\approx g''(0)$ in the
whole interval $J$ and choosing the constant $c_1$ above small enough if necessary.
\end{proof}

\vv

It would appear that Lemma \ref{l:convjordan4} is the most we can extract out of the assumptions stated there, and in particular using only the existence of complementary pairs.  To say more, we need to use the full strength of the admissible pairs property, which has a memory of the limiting sequence $\Omega_j^+$ of Jordan domains.

Our goal will be to prove the following result.

\begin{lemma}\label{lemclau81}
Suppose that $\Omega_j^+$ is a sequence of Jordan domains such that there are closed sets $G^+, G^-$ and $G_0$ such that, with $\Omega_j^-=\R^2\backslash \overline{\Omega_j^+}$ and $\Gamma_j = \partial\Omega_j$,
$$\lim_{j\to\infty}\overline{\Omega_j^{\pm}} =G^{\pm}, \,\text{ and } \lim_{j\to\infty}\overline{\Gamma_j}=G_0\text{ locally as }j\to \infty.
$$
Suppose $\mu_0$ is a measure supported in $G_0$ with $C_0$-linear growth, and $B_0$ is a ball satisfying
\begin{enumerate}
\item there is a line $L \subset G_0$ with $ \mu_0(B_0\cap L)>0$, and
\item \label{admissiblesubseqence}given any subsequence $\{\Omega^+_{j_k}\}_k$, and for every $x\in \supp(\mu_0)$ and $r\in (0, 3r_0)$, there exists a pair $(S_1,S_2)$ that is admissible for the sequence of domains\footnote{To be clear, $(S_1,S_2)$ is an admissible pair for the given sequence $\{\Omega^+_{j_k}\}_k$ of domains means that we can find a further subsequence $\{j_{\ell}\}_{\ell}$ of $\{j_k\}_k$ such that there exists circular arcs
$I_{j_{\ell}}^\pm\subset \partial B(x_{j_{\ell}},r_{j_{\ell}})\cap \Omega^{\pm}_{j_{\ell}}$, with $x_{j_{\ell}}\in\Gamma_{j_{\ell}}$, such that
$I_{j_{\ell}}^+$, $I_{j_{\ell}}^-$ converge to $S_1$, $S_2$ in Hausdorff distance, respectively.} $\{\Omega_{j_k}^+\}_k$ that is centered at $x$ with radius $r>0$.
\end{enumerate}
Then $G_0\cap B_0\subset L$ and if $H_1$, $H_2$ are the two open half planes whose boundary is $L$, we have
that either
$$H_1\cap B_0\subset G^+\setminus G_0 \;\text{ and }\; H_2\cap B_0\subset G^-\setminus G_0,$$
or
$$H_1\cap B_0\subset G^-\setminus G_0 \;\text{ and }\; H_2\cap B_0\subset G^+\setminus G_0.$$
\end{lemma}
Observe that if a sequence of Jordan domains $\Omega_j^+$ satisfies the assumptions of Lemma \ref{lemclau81}, then so does any subsequence of the domains.

Now we need to introduce some additional notation. Given a pair complementary semicircumferences $(S_1,S_2)$,
we say that the two common
end-points of $S_1,S_2$ are the end-points of the pair $(S_1,S_2)$ and we denote the set of these
end-points by $(S_1,S_2)_{\ep}$.

\begin{lemma}\label{lemlongarcstech}  Under the notation and assumptions of Lemma \ref{lemclau81},  fix $x\in \supp\mu_0\cap B_0$, and $y\in \partial B(x,r)\cap G_0$ for some $r\in (0,2r(B_0))$.  Fix $\ell_0>0$. Suppose that, for sufficiently large $k$, and given any subsequence of the domains, we can find sequences
of pairs $(S_{1, r+1/k}, S_{2, r+1/k})$ and $(S_{1,r-1/k},S_{2, r-1/k})$ which are
admissible for the subsequence of domains, which are centred at $x$ and have radii $r+1/k$ and $r-1/k$ respectively, and such that
$$\liminf_{k\to \infty}\dist(y,(S_{1, r+1/k}, S_{2, r+1/k})_{\ep})\geq \ell_0\text{ and }\liminf_{k\to \infty}\dist(y,(S_{1, r-1/k}, S_{2, r-1/k})_{\ep})\geq \ell_0.
$$
Then, there exists a subsequence of arcs $\gamma_{j_k}\subset \Gamma_{j_k}$ which converge in Hausdorff distance
to an arc $I\subset \partial B(x,r)$ such that $y$ is one of its end-points and $\HH^1(I)\geq \ell_0/5$.
\end{lemma}

\begin{proof}
Consider the sequence of radii $s_k = r(1-\frac1k)$ and $t_k = r(1+\frac1k)$.
By assumption, we can find admissible pairs $(S_{1,s_k}, S_{2,s_k})$ centered at $x$ with radii
$s_k$ satisfying,
\begin{equation}\label{eqcond1*}
\liminf_{k\to \infty} \dist(y,(S_{1,s_k}, S_{2,s_k})_{\ep}) \geq \ell_0.
\end{equation}
Consequently, with $\ve_k$ a decreasing sequence chosen much smaller than $1/k$, there is a subsequence $j_k$, and arcs $I_{s_k}^\pm\subset \Omega_{j_k}^\pm \cap\partial B(y_{j_k},\wt s_k)$  such that
\begin{equation}\label{skcloseprops1}|x-y_{j_k}|\leq \ve_k, \quad |s_k - \wt s_k|\leq \ve_k,\text{ and }|\HH^1(I_{s_k}^\pm )-\pi s_k|\leq \ve_k.\end{equation}
Also, insofar as $y\in G_0$, we may choose the subsequence $j_k$ to ensure that there exists \begin{equation}\label{skcloseprops2}\omega_{j_k}\in \Gamma_{j_k}\text{ with }|y-\omega_{j_k}|<\ve_k/k.\end{equation}

Now, by assumption, we can find admissible pairs $(S_{1,t_k}, S_{2,t_k})$  (for the sequence $\{\Omega_{j_k}\}_k$) centered at $x$ with radius $t_k$.  satisfying \begin{equation}\label{eqcond2*}
\liminf_{k\to \infty}\dist(y,(S_{1,t_k}, S_{2,t_k})_{\ep}) \geq \ell_0.
\end{equation}
Thus, by taking a further subsequence, relabelled again by $\{j_k\}_k$ (which preserves all the properties in (\ref{skcloseprops1})) and (\ref{skcloseprops2}), 
we find $I_{t_k}^\pm\subset \Omega_{j_k}^\pm \cap \partial B(z_{j_k},\wt t_k)$ satisfying
$$\quad |x-z_{j_k}|\leq \ve_k,\quad |t_k - \wt t_k|\leq \ve_k, \text{ and }|\HH^1(I_{t_k}^\pm) -\pi t_k|\leq \ve_k.$$

For $k$ big enough and $\ve_k$ small enough, the end-points of $I_{s_k}^\pm$ and $I_{t_k}^\pm$ are far from $y$.
Say, any end-point $z_k$ of these intervals will satisfy $|y-z_k|\geq 0.9\ell_0.$

Assuming $\ve_k\ll1/k$, the arcs $I_{s_k}^\pm$  are essentially some perturbation of some arcs contained in $\partial B(x,s_k)$, while the arcs $I_{t_k}^\pm$ are also another small perturbation of other arcs from $\partial B(x,t_k)$. In fact, there is a thin tubular neighborhood $U_k$ containing $y$ that satisfies the following:
\begin{itemize}
\item
$U_k =A(x,s_k+2\ve_k,t_k-2\ve_k)\cap V$, where $V$ is the sector of $B(x,2r)$ with axis equal to line passing through $x$ and $y$ and such that its angle of aperture is $\ell_0/4r$, say.

\item Associated with the arc $J_{s_k}:=\partial B(x,s_k+2\ve_k)\cap V\subset \partial U_k$ there is a close arc $I'_{s_k}$ contained
either in $I^+_{s_k}$ or $I^-_{s_k}$ such that $\dist_H(J_{s_k}, I'_{s_k})\leq c\ve_k r$.

\item Associated with the arc $J_{t_k}:=\partial B(x,t_k-2\ve_k)\cap V\subset \partial U_k$ there is a close arc $I'_{t_k}$ contained
either in $I^+_{t_k}$ or $I^-_{t_k}$ such that $\dist_H(J_{t_k}, I'_{t_k}) \leq c\ve_k r$.
\end{itemize}

Now we consider the tubular neighborhood $\wt U_k$ whose boundary is formed by the arcs $I'_{s_k}$, $I'_{t_k}$
and two small segments $\ell_k^1$, $\ell_k^2$ that join the closest respective end-points of $I'_{s_k}$ and $I'_{t_k}$,
so that $\dist_H(U_k,\wt U_k)\lesssim \ve_k r$. See Fig. \ref{fig:twoneighbourhoods}.

\begin{figure}
\includegraphics[height=5cm, trim={0cm 0cm 0cm  0cm},clip]{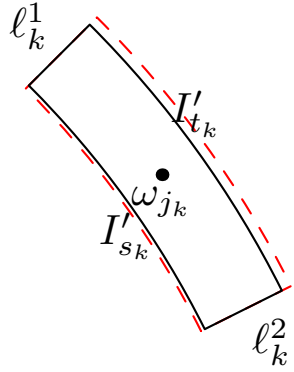}
\caption{The figure depicts the solid black lined tubular neighbourhood $\wt U_k$, and the red dashed neighbourhood $U_k$.}
\label{fig:twoneighbourhoods}
\end{figure}

We distinguish two cases:
\begin{enumerate}
\item
In the first case,  $I'_{s_k}\subset  I^+_{s_k}$ and  $I'_{t_k}\subset  I^-_{t_k}$, or alternatively  $I'_{s_k}\subset  I^-_{s_k}$ and  $I'_{t_k}\subset  I^+_{t_k}$. In both situations, by connectivity
there is a curve $\gamma_{j_k}\subset \partial\Omega_{j_k}^+\cap\wt U_k$ that joints $\ell_k^1$ with $\ell_k^2$.

\item In the second case,  $I'_{s_k}\subset  I^+_{s_k}$ and  $I'_{t_k}\subset  I^+_{t_k}$, or alternatively  $I'_{s_k}\subset  I^-_{s_k}$ and  $I'_{t_k}\subset  I^-_{t_k}$. By construction, the point $\omega_{j_k}\in \Gamma_{j_k}$ satisfies $|\omega_{j_k}-y|<\ve_k/k$, so we have $\omega_{j_k}\in \wt U_k$ and
the  distance of $\omega_{j_k}$ to any of the small segments $\ell_k^1$, $\ell_k^2$ from $\partial\wt U_k$ is at least
$\ell_0/2$, say.
Then, by the connectivity of $\partial \Omega_{j_k}^+$, there is a curve $\gamma_{j_k}\subset \partial\Omega_{j_k}^+\cap\wt U_k$ that joints $\omega_{j_k}$ either to  $\ell_k^1$ or $\ell_k^2$. So its diameter
is at least
$\ell_0/5$.
\end{enumerate}
It is easy to check that the sequence of curves $\gamma_{j_k}$ satisfy the properties asserted in
the lemma.
\end{proof}

\vv

\begin{lemma}\label{lemclau77}
Under the assumptions and notation of Lemma \ref{lemclau81},
let $x\in L\cap\supp\mu_0\cap B_0$ and let $y\in \partial B(x,r)\cap G_0$ for some $r\in (0,2r(B_0))$.
Set
$$\ell_0 = \inf_{(S_1,S_2)} \dist(y,(S_1,S_2)_{\text{ep}}),
$$
where the infimum is taken over all admissible pairs centered at $x$ with radius $r$.

The infimum is attained by an admissible pair centered at $x$ with radius $r$, and if $\ell_0>0$ then
there exists a subsequence of arcs $\gamma_{j_k}\subset \Gamma_{j_k}$ which converge in Hausdorff distance to an arc $I\subset \partial B(x,r)$ such that $y$ is one of its end-points and $\HH^1(I)\geq \ell_0/5$.

\end{lemma}

\begin{proof}  The fact that the infimum is attained by an admissible pair $(S_1,S_2)$ centered at $x$ with radius $r>0$ is an immediate consequence of the closedness of admissible pairs (and Lemma \ref{hauscomp}).  Now suppose $\ell_0 = \dist(y,(S_1,S_2)_{\text{ep}})>0$.

Consider the sequence of radii $s_k = r(1-\frac1k)$ and $t_k = r(1+\frac1k)$.
Let $(S_{1,s_k}, S_{2,s_k})$ and $(S_{1,t_k}, S_{2,t_k})$ be sequences of admissible pairs centered at $x$ with radii
$s_k$ and $t_k$ respectively.  By taking a subsequence, we may assume that these admissible pairs converge in Hausdorff metric to admissible pairs with centre $x$ and radius $r$ say $(S_1^-,S_2^-)$ and $(S_1^+, S_2^+)$.  By the minimal property of $(S_1,S_2)$ we must have that $\dist(y,(S_{1}^{\pm}, S_{2}^{\pm})_{\ep})\geq \dist(y,(S_{1}, S_{2})_{\ep})$. By the closeness property of the set of admissible pairs, we infer that
$$
\liminf_{k\to\infty} \dist(y,(S_{1,s_k}, S_{2,s_k})_{\ep}) \geq \dist(y,(S_1,S_2)_{\ep})
$$
and
$$
\liminf_{k\to\infty} \dist(y,(S_{1,t_k}, S_{2,t_k})_{\ep}) \geq \dist(y,(S_1,S_2)_{\ep}).
$$
Consequently, we may apply Lemma \ref{lemlongarcstech} with $\ell_0=\dist(y,(S_1,S_2)_{\ep})$.
\end{proof}


\begin{lemma}\label{lemclau78}
Under the assumptions and notation of Lemma \ref{lemclau81},
let $x\in L\cap\supp\mu_0\cap B_0$ and let $r\in (0,2r(B_0))$.
Then there exists an admissible pair of semicircumferences centered at $x$ with radius $r$ whose end-points
belong to $L$.
\end{lemma}

\begin{proof}
Let $x\in L\cap\supp\mu_0\cap B_0$ and let $r\in (0,2r(B_0))$.
Suppose that there does not exist an admissible pair of semicircumferences centered at $x$ with radius $r$ whose end-points
belong to $L$.  Let $y\in L\cap \partial B(x,r)$, then, by Lemma \ref{lemclau77},
$$\ell_0 = \inf_{(S_1,S_2)} \dist(y,(S_1,S_2)_{\text{ep}})>0
$$
where the infimum is taken over all admissible pairs centered at $x$ with radius $r$.   Consequently, Lemma \ref{lemclau77} ensures that there exists a subsequence of arcs $\gamma_{j_k}\subset \Gamma_{j_k}$ which converge in Hausdorff distance
to an arc $I\subset \partial B(x,r)$ such that $y$ is one of its end-points and $I$ has length at least
$\ell_0/5$.

By the closeness property of the admissibility property of pairs, for any small $\delta\in(0,r/2)$ there exists another radius $r_\delta
\in (r-\delta,r)$ close enough to $r$ such that, denoting by $y_\delta$ the point in $ L\cap \partial B(x,r_\delta)$ which is closest to $y$,
 any admissible pair $(S_1^\delta,S_2^\delta)$ of semicircumferences contained in $\partial B(x,r_\delta)$
satisfies
$\dist(y_\delta,(S_1^\delta,S_2^\delta)_{\ep})\geq \ell_0/2$.  Observe that $y_\delta\in G_0$ and then, by
applying Lemma
\ref{lemclau77} to the subsequence of domains $\Omega_{j_k}^+$ (observe that the sets $G^+$, $G^-$, $G_0$ associated with the subsequence are the same as the ones associated with the original sequence $\{\Omega_j\}_j$), we infer that there is a subsequence of arcs $\gamma_{\delta,j'_k}\subset \Gamma_{j'_k}$ which converge in Hausdorff distance
to an arc $I_\delta\subset \partial B(x,r_\delta)$ such that $y_\delta$ is one of its end-points and $I_\delta$ has length at least
$\ell_0/10$. By renaming the subsequence, we can assume that $\{j'_k\}_k$ coincides
with $\{j_k\}_k$.

By iterating the preceding argument, we still find another
 $r^\delta
\in (r,r+\delta)$ close enough to $r$ for which, after renaming the subsequence and
denoting by $y^\delta$ the point in $ L\cap \partial B(x,r^\delta)$ which is closest to $y$,
there is a family arcs $\gamma^\delta_{j_k}\subset \Gamma_{j_k}$ which converge in Hausdorff distance
to an arc $I^\delta\subset \partial B(x,r^\delta)$ such that $y^\delta$ is one of its end-points and has length at least
$\ell_0/10$.

Let $x'\in L\cap \supp\mu_0$ with $x'\neq x$. Suppose that $x'$ and $y$ are in the same half-line contained in $L$ with end-point equal to
$x$ (i.e., $x'$ and $y$ are at the same side of $x$ in $L$). Otherwise, in the arguments above we interchange $y$ with the other point from $\partial B(x,r)\cap L$.
It is easy to check that any circumference $\partial B(x'',r')$, with $\delta$ small enough and $x''$ close enough to $x'$, intersects at least two of the arcs $I,I_\delta,I^\delta$ for all $r'$ in
 some interval $H$ of width bounded from below depending on the relative position of $x,x',y,y_\delta,y^\delta$. In fact, the same phenomenon happens
replacing the arcs $I,I_\delta,I^\delta$ by the curves $\gamma_{j_k}, \gamma_{\delta,j_k},\gamma^\delta_{j_k}$, assuming $k$ big enough. From this fact, one deduces easily that there exists some $r'\in H$ such that there is no admissible pair of semicircumferences with center $x'$ and radius $r'$ (associated to the sequence of domains $\{\Omega_{j_k}\}_k$), which is in contradiction with the hypothesis (\ref{admissiblesubseqence}) in Lemma \ref{lemclau81}.
\end{proof}
\vv

\begin{lemma}\label{lemclau79}
Under the assumptions of Lemma \ref{lemclau81},
let $x\in L\cap\supp\mu_0\cap B_0$ and let $r\in (0,2r(B_0))$. Assume that $L$ coincides with the horizontal axis and
suppose that $\partial B(x,r)\cap G_0\cap\R^2_+\neq \varnothing$ (recall  that we assume $\R^2_+$ to be open) and let $y\in\partial B(x,r)\cap G_0\cap\R^2_+$. Then there exists a sequence of arcs $\gamma_{j_k}\subset \Gamma_{j_k}$ which converge in Hausdorff distance
to an arc $I\subset \partial B(x,r)$ such that $y$ is one of its end-points and has length at least
$\dist(y,L)/5$.
\end{lemma}

\begin{proof}
Fix a subsequence $\Omega_{j_k}$ of the domains $\Omega_j$.  This subsequence $\Omega_{j_k}$ again satisfies the assumptions of Lemma \ref{lemclau81} (with the same choices of sets $G^+$, $G^-$ and $G_0$).   By Lemma \ref{lemclau78}, for every $\delta\in (0,r/2)$, there are admissible pairs
$(S_{1,\delta},S_{2,\delta})$ for $\Omega_{j_k}$ centered at $x$ with respective radii equal to any number $s\in (r-\delta, r+\delta)$ such that their end-points belong all to $L$.  Consequently, for sufficiently large $k$ we may certainly find sequences of admissible pairs $(S_{1, r+1/k}, S_{2, r+1/k})$ and $(S_{1,r-1/k},S_{2, r-1/k})$ centred at $x$ with radii $r+1/k$ and $r-1/k$ respectively, and with end points on $H$ (and so at a distance $\dist(y,L)$ from $y$).  Thus, we may apply Lemma \ref{lemlongarcstech} with $\ell_0 = \dist(y, L)$, which completes the proof.
\end{proof}

We are now in a position to complete the proof of Lemma \ref{lemclau81}, which is an immediate consequence of the following statement.

\begin{lemma}\label{lemclau80}
Under the assumptions of Lemma \ref{lemclau81},
let $x\in L\cap\supp\mu_0\cap B_0$ and let $r\in (0,2r(B_0)$.
Then  $\partial B(x,r)\cap G_0\cap\R^2_+= \varnothing$ (assuming $L$ to be the horizontal axis).
\end{lemma}

\begin{proof}
Suppose that  $y\in \partial B(x,r)\cap G_0\cap\R^2_+$. By  Lemma \ref{lemclau79}, there exists a sequence of arcs $\gamma_{j_k}\subset \Gamma_{j_k}$ which converge in Hausdorff distance
to an arc $I\subset \partial B(x,r)$ such that $y$ is one of its end-points and has length at least
$\dist(y,L)/5$.

Let $x'\in\supp\mu_0\cap L$,  with $x'\neq x$, and let $y'$ be the middle point of the arc $I$ (we may assume that $y'\not\in L$), and let $r'=|x'-y'|$, so that $\partial B(x',r')$ intersects $I$ in the middle point.
By connectivity arguments, the existence of the curves $\gamma_{j_k}$ given by Lemma \ref{lemclau79} implies that, for the
subsequence of domains $\Omega_{j_k}$,
there does not exist an admissible pair of semicircumferences centered at $x'$ with radius $r'$ whose
end-points belong to $L$. This fact contradicts Lemma \ref{lemclau78}.
\end{proof}



With Lemma \ref{lemclau81} proved, we are now in a position to complete the proof of Lemma \ref{lembetajordan}.

\vv

\begin{proof}[\bf Proof of Lemma \ref{lembetajordan}]

By renormalizing it suffices to prove the lemma for the ball $B_0:=B(0,1)$. We argue by contradiction: we suppose that there exists an $\ve>0$ such that for all $j \in \N$, there exists a Jordan domain $\Omega_j^+$ with
$\Gamma_j:=\partial \Omega_j^+$ containing $0$ supporting a measure $\mu_j$ with $C_0$-linear growth with $\mu_j(B_0)> c_0$, so that we have
\begin{align*}
    \int_{7B_0}\int_0^{7} \big(\ve_j(x,r)^2 + \alpha^+_{j}(x,r)^2\big) \, \dr \, d\mu_j(x) \leq \frac{1}{j}\,\mu_j(7_0),
\end{align*}
and
\begin{align*}
    \beta_{\infty, \Gamma_j}(B_0) > \ve.
\end{align*}
Here we denote by $\ve_j(x,r)$ and $\alpha^+_{j}$ the coefficients $\ve(x,r)$ and $\alpha^+$
associated with $\Omega_j^+$.

We first apply Lemma \ref{l:convjordan} to pass to a subsequence of the domains such that, with $\Omega_j^- = \R^2\backslash \overline{\Omega_j^+}$,
$$\Omega_j^{\pm}\to G^{\pm} \text{ and }\Gamma_j \to G_0
$$
locally as $j\to \infty$.  By passing to a further subsequence if necessary, we may assume that $\mu_j$ converge weakly to a measure $\mu_0$ with $C_0$-linear growth satisfying $\mu(\overline{B_0})\geq c_0$. Applying Lemma \ref{l:convjordan2}, we infer that the assumptions of Lemma \ref{l:convjordan4} are satisfied with $B_0$ replaced by the ball $2B_0$, so there is a line $L\subset G_0$ with $\mu_0(L\cap 2B_0)>0$.

Observe now that we can also apply Lemma \ref{l:convjordan2} to any subsequence of the domains.  In particular, from the conclusion (\ref{jordanadmissible}) of Lemma \ref{l:convjordan2} applied to a given subsequence, we infer that the assumption (\ref{admissiblesubseqence}) of Lemma \ref{lemclau81} also holds, again with the ball $2B_0$ playing the role of $B_0$ in Lemma \ref{lemclau81}.  Therefore, applying Lemma \ref{lemclau81} to the ball $2B_0$ that satisfies $\mu_0(L\cap 2B_0)>0$, we have that  $G_0\cap 2B_0\subset L$.  Consequently, $\Gamma_j\cap \overline{B_0}$ converges in Hausdorff distance to a subset of $L$, which contradicts $\beta_{\infty, \Gamma_j}(B_0) > \ve$ for sufficiently large $j$.
\end{proof}
\vv


\part*{Part II: From local flatness to rectifiability}

\section{The smooth square function on Lipschitz graphs} \label{s:fourier}

Recall that, given an integrable $C^\infty$ function $\psi:\R^2\to\R$,
an open set $\Omega^+\subset \R^2$ and $x\in\R^2$, $r>0$, we denote
$$c_\psi = \int_{y\in\R^2_+} \psi(y)\,dy,
\qquad \lca_\psi(x,r) = \bigg|c_\psi - \frac1{r^2}\int_{\Omega^+} \psi\bigg(\frac{y-x}{r}\bigg)\,dy\bigg|.
$$
We also set
$$\A_\psi(x)^2 = \int_0^\infty \lca_\psi(x,r)^2\,\frac{dr}r.$$
Remark that we allow $\psi$ to be non-radial.

We fix an even $C^\infty$ function $\vphi:\R\to\R$ such that $\chara_{[-1,1]}\leq \vphi\leq  \chara_{[-1.1,1.1]}$, and we denote
$$\vphi_r(x) = \frac1{r} \vphi\Big(\frac {x}r\Big),\quad\mbox{ for $x\in\R$, $r>0$.}$$

Our objective in this section is to prove the following.

\begin{lemma}\label{lemlips}
Consider a Lipschitz function $f:\R\to\R$ with compact support and let $\Gamma\subset\R^2$ be its Lipschitz graph. Let
$\Omega^+ =\{(x,y)\in\R^2:y>f(x)\}$ and $\Omega^- =\{(x,y)\in\R^2:y<f(x)\}$.
Let $\vphi$ be a function as above and let
$$\psi(x) = \vphi(|x|),\quad \mbox{ for $x\in\R^2$.}$$
Let $\A_\psi$ and $\lca_\psi$ be the associated coefficients defined above.
There exists some $\alpha_0>0$ such that if $\|f'\|_\infty\leq \alpha_0$, then
$$\int_\Gamma \A_\psi(x)^2\,d\HH^1(x)\approx \|f'\|_{L^2(\R)}^2.$$
\end{lemma}

We will prove this result by using the Fourier transform. This will play an
essential role in the proof of Main Lemma \ref{l:main2}.
We need first some auxiliary results.

\begin{lemma}\label{lemfourier}
Let $f:\R\to\R$ be a Lipschitz function with compact support. Then we have
$$\int_\R\int_0^\infty \left|\frac{f*\vphi_r(x) - c(\vphi)f(x)}r \right|^2\,\frac{dr}r\,dx = c\,\|f'\|_{L^2(\R)}^2,$$
where $c(\vphi) = \int_\R\vphi\,dx$ and $c>0$.
\end{lemma}

\begin{proof}
By Plancherel, we have
\begin{align*}
\int_\R\int_0^\infty \left|\frac{f*\vphi_r(x) - c(\vphi)f(x)}r \right|^2\,\frac{dr}r\,dx
& =
\int_\R\int_0^\infty \left|\frac{\wh f(\xi) \wh \vphi(r\xi) - \wh f(\xi) \wh \vphi(0)}r \right|^2\,\frac{dr}r\,d\xi\\
& =
\int_\R  |f(\xi)|^2 \int_0^\infty \big| \wh \vphi(r\xi) - \wh \vphi(0)\big|^2\,\frac{dr}{r^3}\,d\xi.
\end{align*}
By the change of variable $r|\xi|=t$, we get
$$\int_0^\infty \big| \wh \vphi(r\xi) - \wh \vphi(0)\big|^2\,\frac{dr}{r^3}
= |\xi|^2\int_0^\infty \big| \wh \vphi(t) - \wh \vphi(0)\big|^2\,\frac{dt}{t^3} =: \tilde  c(\vphi) |\xi|^2,$$
where $0<\tilde c(\vphi)<\infty$, since $\wh \vphi(t) - \wh \vphi(0) = O(t^2)$ as $t\to0$ (because $\vphi$
is an even function in the Schwartz class).
Hence,
$$\int_\R\int_0^\infty \left|\frac{f*\vphi_r(x) - c(\vphi)f(x)}r \right|^2\,\frac{dr}r\,dx
= \tilde c(\vphi)\int_\R |\xi \wh f(\xi)|^2\,d\xi = c\,\|f'\|_{L^2(\R)}^2.$$
\end{proof}

\begin{lemma}\label{lemfourier2}
Let $f:\R\to\R$ be a Lipschitz function with compact support. Then we have
\begin{equation}\label{eqfour4}
\int_0^\infty\!\!\int_\R \int_{y\in\R:|y-x|\leq r} \left|\frac{ c(\vphi)^{-1}(\vphi_r*f')(x)(y-x)+f(x)-f(y)}r\right|^2\,\frac{dy}r\,dx\,\frac{dr}r = c\,\|f'\|_{L^2(\R)}^2.
\end{equation}
\end{lemma}

\begin{proof}
By replacing $\vphi$ by $c(\vphi)^{-1}\vphi$ if necessary, we may assume that
$\int \vphi\,dx=1$. This is due to the fact that, as we shall see below, the assumption that $\chara_{[-1,1]}\leq \vphi\leq  \chara_{[-1.1,1.1]}$ is not necessary for the validity of this lemma.

Appealing to the change of variable $z=y-x$ and Fubini's theorem, the left hand side of \rf{eqfour4}
(with $c(\vphi)=1$) equals
\begin{align*}
\int_0^\infty\!\!&\int_{z\in\R:|z|\leq r}\int_{x\in\R}  \left|\frac{(\vphi_r*f')(x)z+f(x)-f(x+z)}r\right|^2\,dx\,\frac{dz}r\,\frac{dr}r \\
& \stackrel{\text{Plancharel}}{=}
\int_0^\infty \!\!\int_{z\in\R:|z|\leq r}\int_{\xi\in\R} \left|\frac{2\pi i \xi z\,\wh\vphi(r\xi)\,\wh f(\xi)+\wh f(\xi)-e^{2\pi i \xi z} \wh f(\xi)}r\right|^2\,\frac{dz}r\,d\xi\,\frac{dr}r.
 \end{align*}
Using Fubini's theorem to interchange the inner two integrals, and the changes of variable
$w=\xi z$, $s=|\xi|r$, we infer from the fact that $\vphi$ (and so $\wh{\vphi}$) is even that the last triple integral equals
\begin{multline*}
\int_{\xi\in\R} \int_0^\infty \!\!\int_{w\in\R:|w|\leq s} \left|2\pi i w\,\wh\vphi(s)\,\wh f(\xi)+\wh f(\xi)-e^{2\pi i w} \wh f(\xi)\right|^2\,|\xi|^2\,\frac{ds}{s^4}\,dw\,d\xi\\
= \int_{\xi\in\R} \big|\xi\,\wh f(\xi)|^2\,d\xi \,\int_0^\infty \!\!\int_{w\in\R:|w|\leq s}
\left|2\pi i w\,\wh\vphi(s)+ 1 -e^{2\pi i w} \right|^2\,dw\,\frac{ds}{s^4}.
\end{multline*}
 Hence, to prove the lemma it suffices to show that the last double integral
 $$ I:= \int_0^\infty \!\!\int_{w\in\R:|w|\leq s}
\left|2\pi i w\,\wh\vphi(s)+ 1 -e^{2\pi i w} \right|^2\,dw\,\frac{ds}{s^4} $$
 is absolutely convergent and positive.
 That this is positive is immediate.
 To show that this is  absolutely convergent,
 we split it  as follows:
 \begin{align*}
 I & =
\int_0^\infty  \!\!\int_{|w|\leq \min(s,1)}\cdots\; + \int_1^\infty  \!\!\int_{1\leq |w|\leq s}\cdots  =: I_1 + I_2.
\end{align*}

First we estimate $I_2$:
\begin{align*}
I_2 &\lesssim \int_1^\infty  \!\!\int_{1\leq |w|\leq s}\big(1+ \big|w\,\wh\vphi(s)\big|^2\big)\,dw\,\frac{ds}{s^4} \\ &\lesssim
 \int_1^\infty  \frac{ds}{s^3} +
  \int_1^\infty  \!\!\int_{1\leq |w|\leq s}\big|s\,\wh\vphi(s)\big|^2\,dw\,\frac{ds}{s^4}  \lesssim 1+  \int_1^\infty \big|\wh\vphi(s)\big|^2\,\frac{ds}{s} \lesssim 1
\end{align*}
Concerning $I_1$, we have
\begin{equation}\begin{split}\label{eqfou84}
I_1  \lesssim \int_0^\infty  &\!\!\int_{|w|\leq \min(s,1)}
\left|2\pi i w + 1 -e^{2\pi i w} \right|^2\,dw\,\frac{ds}{s^4}\\
& +
\int_0^\infty  \!\!\int_{|w|\leq \min(s,1)}
\left|2\pi i w\,(\wh\vphi(s) - 1) \right|^2\,dw\,\frac{ds}{s^4}.
\end{split}\end{equation}
The first term on the right hand side satisfies
\begin{align*}
\int_0^\infty  \!\!\int_{|w|\leq \min(s,1)}
\left|2\pi i w + 1 -e^{2\pi i w} \right|^2\,dw\,\frac{ds}{s^4} & \leq
\int_{|w|\leq 1} \int_{s\geq |w|}
\left|2\pi i w + 1 -e^{2\pi i w} \right|^2\,\frac{ds}{s^4}\,dw\\
& \lesssim
\int_{|w|\leq 1}
\left|2\pi i w + 1 -e^{2\pi i w} \right|^2\,\frac{dw}{|w|^3}\lesssim1,
\end{align*}
taking into account that $2\pi i w + 1 -e^{2\pi i w}=O(w^2)$ as $w\to0$.
Finally we turn our attention to the second term on the right hand side of \rf{eqfou84}:
\begin{align*}\int_0^\infty  \!\!\int_{|w|\leq \min(s,1)}\!
\left|2\pi i w\,(\wh\vphi(s) - 1) \right|^2\,dw\,\frac{ds}{s^4}
&\lesssim \int_{|w|\leq 1} |w|^2 \,dw\, \int_0^\infty\! \left|\wh\vphi(s) - 1 \right|^2\,\frac{ds}{s^4}
\\&\lesssim  \int_0^\infty \!\left|\wh\vphi(s) - 1 \right|^2\,\frac{ds}{s^4}.\end{align*}
Since $\vphi\in C^{\infty}$ is even, and $\wh\vphi(0)=1$, we have $\wh\vphi(s) - 1= O(s^2)$ as $s\to0$, and so the last integral is finite. So $I_2<\infty$ and the proof of the lemma
is concluded.
\end{proof}
\vv

\begin{lemma}\label{lem5.4}
Let $f:\R\to\R$ be a Lipschitz function with compact support with $\| f'\|_\infty\leq 1/10$.
For $x=(x_1,x_2)\in\R^2$, denote
$$\rho(x) = \vphi(x_1)\,\vphi(x_2).$$
Then we have
$$\lca_\rho(x,r) = \frac{\vphi_r * f(x_1) - c(\vphi)\,f(x_1)}r\quad \mbox{for all $x\in\R^2$ in the graf of $f$ and all $r>0$}
$$
and
$$\int_\R \A_{\rho}((x_1,f(x_1)))^2\,dx_1 =
\int_\R\int_0^\infty \left|\frac{f*\vphi_r(x_1) - c(\vphi)f(x_1)}r \right|^2\,\frac{dr}r\,dx_1,$$
where $c(\vphi) = \int_\R\vphi\,dx$.
\end{lemma}

\begin{proof}
Observe that
\begin{align*}
c_\rho - \frac1{r^2}\int_{\Omega^+} \rho\bigg(\frac{y-x}{r}\bigg)\,dy & =
\frac1{r^2}\int_{y_1\in\R}\int_{y_2> f(x_1)} \vphi\bigg(\frac{y_1-x_1}{r}\bigg)\vphi\bigg(\frac{y_2-x_2}{r}\bigg)\,dy_2\,dy_1 \\
& \quad - \frac1{r^2}
\int_{y_1\in\R}\int_{y_2> f(y_1)} \vphi\bigg(\frac{y_1-x_1}{r}\bigg)\vphi\bigg(\frac{y_2-x_2}{r}\bigg)\,dy_2dy_1 \\
& = \int_{y_1\in\R} \vphi_r(y_1-x_1)
\int_{f(x_1)}^{f(y_1)} \vphi_r(y_2-x_2)\,dy_2 dy_1
\end{align*}
Observe also that, if $\vphi_{r}(y_1-x_1)\neq 0$, then because $\|f'\|_\infty\leq 1/10$,
$$\vphi_r(y_2-x_2)=\frac1{r}\quad \mbox{ for $y_2\in [f(x_1),f(y_1)]$.}$$
 As a consequence
$$c_\rho - \frac1{r^2}\int_{\Omega^+} \rho\bigg(\frac{y-x}{r}\bigg)\,dy
= \int_{y_1\in\R} \vphi_r(y_1-x_1)
\frac{f(y_1)- f(x_1)}r  dy_1 = \frac{\vphi_r * f(x_1) - c(\vphi)\,f(x_1)}r
.$$
Hence,
$$\A_\vphi(x)^2=
\int_0^\infty \lca_\vphi(x,r)^2 \,\frac{dr}r= \int_0^\infty \left|\frac{\vphi_r * f(x_1) - c(\vphi)\,f(x_1)}r\right|^2 \,\frac{dr}r.$$
Integrating with respect to $x_1$ in $\R$, the lemma follows.
\end{proof}

\begin{lemma}\label{lemdiff1}
Let $f:\R\to\R$ be a Lipschitz function with compact support with $\| f'\|_\infty\leq 1/10$ and let $\Gamma\subset \R^2$ be
its Lipschitz graph.
For $x=(x_1,x_2)\in\R^2$, denote
$$\rho(x) = \vphi(x_1)\,\vphi(x_2)\quad \text{and}\quad \psi(x)=\vphi(|x|)
.$$
Then we have
$$\int_\Gamma |\A_{\rho}(x) - \A_\psi(x)|^2\,d\HH^1(x) \lesssim \| f'\|_\infty^4\,\|f'\|_{L^2(\R)}^2.
$$
\end{lemma}

\begin{proof}
For $r>0$, $x\in\R^2$, we denote
$$\rho_r(x) = \frac1{r^2} \rho\Big(\frac xr\Big),\qquad\psi_r(x) = \frac1{r^2} \psi\Big(\frac xr\Big),
\qquad \vphi_r(x_1) = \frac1{r} \vphi\Big(\frac {x_1}r\Big)
.$$
Then we have
\begin{equation}\label{eq931}
\big|\lca_\rho(x,r) - \lca_\psi(x,r)\big|\leq \big|(\rho_r * \chara_{\Omega^+} - c_\rho) - (\psi_r * \chara_{\Omega^+} - c_\psi)\big|.
\end{equation}

For $x\in\Gamma$, $r>0$, we denote by $L_{x,r}$ the line passing through $x$ with slope equal to $c(\vphi)^{-1}(\vphi_r*f')(x_1)$, and we let $H_{x,r}^+$, $H_{x,r}^-$ be two complementary half planes whose common boundary is $L_{x,r}$, so that $H_{x,r}^+$ is above $L_{x,r}$ and $H_{x,r}^-$ is below $L_{x,r}$.

Observe that, by the radial symmetry of $\psi$,
$$c_\psi = \int_{y\in\R^2_+} \psi(y)\,dy = \psi_r*\chara_{H_{x,r}} (x)\quad \mbox{ for all $x\in\R^2$ and $r>0$.}$$
We claim that the same identity holds replacing $\psi$ by $\rho$. To check this, suppose that $x=0$ for easiness of
notation and let $y_2= b\,y_1$ be equation of the line $L_{0,r}$. Then, by the evenness of $\vphi$ we have
\begin{align*}
\rho_r*\chara_{H_{0,r}} (0) & =\int \vphi_r(y_1)\int_{y_2>by_1}\vphi_r(y_2)\,dy_2\,dy_1\\
& = \frac12 \int \vphi_r(y_1)\int_{y_2>b y_1}\vphi_r(y_2)\,dy_2\,dy_1 + \frac12
\int \vphi_r(y_1)\int_{y_2>-b y_1}\vphi_r(y_2)\,dy_2\,dy_1\\
& = \frac12 \int \vphi_r(y_1)\int_{y_2>b y_1}\vphi_r(y_2)\,dy_2\,dy_1 + \frac12
\int \vphi_r(y_1)\int_{y_2<b y_1}\vphi_r(y_2)\,dy_2\,dy_1\\
& = \frac12 \int \rho_r(y)\,dy =\int_{\R^2_+} \rho_r(y)\,dy = c_\rho,
\end{align*}
which proves the claim.

From the above identities and \rf{eq931}, for $x\in\Gamma$, we obtain
\begin{align}\label{eqplu89}
\big|\lca_\rho(x,r) - \lca_\psi(x,r)\big| & \leq \big|\rho_r * (\chara_{\Omega^+} - \chara_{H_{x,r}^+})(x) - \psi_r * (\chara_{\Omega^+} - \chara_{H_{x,r}^+})(x)\big| \\
& = \big|(\rho_r- \psi_r) * (\chara_{\Omega^+} - \chara_{H_{x,r}^+})(x)\big| \nonumber\\
& \leq \int_{\Omega^+ \Delta H_{x,r}^+} |\rho_r(y-x)- \psi_r(y-x)|\,dy.\nonumber
\end{align}
 But now observe that, if $|x-y|<3r$, then using the fact that $\|f'\|_{\infty}<1/10$, we have
$\vphi_r(y_2-x_2) = \frac1r$ for all $x\in \Gamma$ and $y\in \Omega^+ \Delta H_{x,r}^+$. Thus,
by the definition of $\rho$ and $\psi$,
$$\rho_r(y-x)- \psi_r(y-x) = \vphi_r(y_1-x_1) \vphi_r(y_2-x_2) - \frac1r \vphi_r(|y-x|) =
\frac1r\,\big( \vphi_r(y_1-x_1) - \vphi_r(|y-x|)\big).$$
Still for $x\in \Gamma$ and $y\in \Omega^+ \Delta H_{x,r}^+$,
notice that if $|x-y|\leq r/2$, then
$\vphi_r(y_1-x_1) =\vphi_r(|y-x|) =\frac1r$
and thus $\rho_r(y-x)- \psi_r(y-x)=0$; while if $|x-y|\geq r/2$
$$\big|\rho_r(y-x)- \psi_r(y-x)\big|\leq \frac1r\|(\vphi_r)'\|_\infty \,\big|(y_1-x_1)- |y-x|\big|
\lesssim \frac{|y_2-x_2|^2}{r^4}
.$$
Since $\supp\rho_r \cup\supp\psi_r\subset B(0,3r)$, in any case we get
$$\big|\rho_r(y-x)- \psi_r(y-x)\big|
\lesssim \frac{(\|f'\|_\infty\,r)^2}{r^4} =  \frac{\|f'\|_\infty^2}{r^2}
\quad \mbox{for $x\in \Gamma$ and $y\in \Omega^+ \Delta H_{x,r}^+$}
.$$
Plugging this estimate into \rf{eqplu89}, we obtain
$$\big|\lca_\rho(x,r) - \lca_\psi(x,r)\big|\lesssim \frac{\|f'\|_\infty^2}{r^2}\,\HH^2\big(
(\Omega^+ \Delta H_{x,r}^+)\cap B(x,3r)\big).$$
Next, using the fact that the equation of  the line $L_{x,r}$ is $$y_2 = c(\vphi)^{-1}(\vphi_r*f')(x_1)(y_1-x_1) + f(x_1),$$
we get
\begin{align*}
\HH^2\big(
&(\Omega^+ \Delta H_{x,r}^+)\cap B(x,3r)\big)  \leq \int_{|y_1-x_1|\leq 3r} |c(\vphi)^{-1}(\vphi_r*f')(x_1)(y_1-x_1) + f(x_1) - f(y_1)|\,dy_1\\
& \lesssim r^{1/2}
\left(\int_{|y_1-x_1|\leq 3r} |c(\vphi)^{-1}(\vphi_r*f')(x_1)(y_1-x_1) + f(x_1) - f(y_1)|^2\,dy_1\right)^{1/2}.
\end{align*}
Hence,
\begin{align*}
\big|\lca_\rho(x,r) - &\lca_\psi(x,r)\big|\\
&\lesssim \frac{\|f'\|_\infty^2}{r^{3/2}}\,
\left(\int_{|y_1-x_1|\leq 3r} |c(\vphi)^{-1}(\vphi_r*f')(x_1)(y_1-x_1) + f(x_1) - f(y_1)|^2\,dy_1\right)^{1/2}
.
\end{align*}
Therefore,
\begin{align*}
|\A_{\rho}&(x) - \A_\psi(x)| \\
& = \left|\left(\int_0^\infty \lca_\rho(x,r)^2\,\frac{dr}r\right)^{1/2} -
\left(\int_0^\infty \lca_\psi(x,r)^2\,\frac{dr}r\right)^{1/2}
\right|\\
& \leq \left(\int_0^\infty |\lca_\rho(x,r) - \lca_\psi(x,r)|^2\,\frac{dr}r\right)^{1/2}\\
&\lesssim
\|f'\|_\infty^2
\left(\int_0^\infty \!\!
\int_{|y_1-x_1|\leq 3r} |c(\vphi)^{-1}(\vphi_r*f')(x_1)(y_1-x_1) + f(x_1) - f(y_1)|^2\,dy_1
\,\frac{dr}{r^3}\right)^{1/2}.
\end{align*}
Squaring and integrating on $x$ and applying Lemma \ref{lemfourier2}, we get
\begin{align*}
&\int_\Gamma |\A_{\rho}(x) - \A_\psi(x)|^2\,d\HH^1(x) \approx
\int |\A_{\rho}(x) - \A_\psi(x)|^2\,dx_1\\
&\; \lesssim
\|f'\|_\infty^4 \!\int_\R \int_0^\infty \!\!\!
\int_{|y_1-x_1|\leq 3r} \!\!|c(\vphi)^{-1}(\vphi_r*f')(x_1)(y_1-x_1) + f(x_1) - f(y_1)|^2\,dy_1
\,\frac{dr}{r^3}\,dx_1 \\
&\;\approx \|f'\|_\infty^4\,\|f'\|_{L^2(\R)}^2.
\end{align*}
\end{proof}

\vv
\begin{proof}[\bf Proof of Lemma \ref{lemlips}]
By Lemma \ref{lemfourier} and Lemma \ref{lem5.4}, we have
$$\int_\Gamma \A_{\rho}(x)^2\,d\HH^1(x) \approx
\|f'\|_{L^2(\R)}^2.$$
On the other hand, by Lemma \ref{lemdiff1},
$$\int_\Gamma |\A_{\rho}(x) - \A_\psi(x)|^2\,d\HH^1(x) \lesssim
\|f'\|_\infty^4\,\|f'\|_{L^2(\R)}^2.
$$
Hence,
$$\int_\Gamma \A_\psi(x)^2\,d\HH^1(x)\leq 2 \int_\Gamma \A_{\rho}(x)^2\,d\HH^1(x) + 2 \int_\Gamma |\A_{\rho}(x) - \A_\psi(x)|^2\,d\HH^1(x) \lesssim \|f'\|_{L^2(\R)}^2.
$$
In the converse direction, we have
\begin{align*}
\int_\Gamma \Apsi(x)^2\,d\HH^1(x) & \geq \frac12 \int_\Gamma \A_{\rho}(x)^2\,d\HH^1(x)
- \int_\Gamma \Apsi(x)^2\,d\HH^1(x) \\
&\geq c_1 \|f'\|_{L^2(\R)}^2 - C\,
\|f'\|_\infty^4\,\|f'\|_{L^2(\R)}^2.
\end{align*}
So if $\|f'\|_\infty^4\leq c_1/2C$, the lemma follows.
\end{proof}


\section{Construction of an approximate Lipschitz graph}

This and the remaining sections of the paper  are devoted to the proof of the Main Lemma \ref{l:main2}. To
this end, we need to construct a Lipschitz graph which covers a fairly big proportion of the measure $\mu$. We will achieve this through a construction stemming from  works of David and Semmes in \cite{DS1} and of L\'eger in \cite{Leger}.  Given the form of Main Lemma \ref{l:main2}, it is convenient for us to follow the presentation given in the monograph \cite[Chapter 7]{Tolsa-llibre}.\\

Fix $\ve>0$ and $\theta>0$.  We assume that the assumptions of Main Lemma \ref{l:main2} are satisfied with these choices of parameters $\ve$ and $\theta$.  We shall also introduce $\alpha>0$, $\alpha \ll 1$.  Here $\alpha$ will regulate the slope of a Lipschitz graph that will well approximate the support of $\mu$.  We will eventually determine $\theta$, then we will pick $\alpha$ depending on $\delta$, and finally $\ve$ can be chosen to depend on both $\theta$ and $\alpha$.\\

Set $E=\supp(\mu)$.  Put $B_0 = B(x_0,R)$ to be the ball given in Lemma \ref{l:main2}.  Then $E\subset \overline{B_0}$.   By replacing the ball $B_0$ by a ball with at most double radius, and replacing $c_0$ by $c_0/2$, if necessary, we may assume that a line $L_0$ that minimizes $\beta_{\infty,\Gamma}(B_0)$ passes through $x_0$.  Furthermore, we may assume $x_0=0$ and $L_0$ is the horizontal axis $\R\times\{0\}$.\\

For a ball $B$, $L_B$ denotes a best approximating line for $\beta_{\infty,\Gamma}(B)$.\\

Let $x\in E$ and $0<r\leq 50R$.  We call the ball $B=B(x,r)$ \emph{good}, and write $B\in\G$ if
\begin{enumerate}[label=\textup{(}\alph*\textup{)}]
\item  $\Theta_{\mu}(B)\geq \theta$, and
\item  $\angle(L_B, L_{B_0})\leq \alpha.$
\end{enumerate}
Therefore, by the assumptions of Lemma \ref{l:main2}\begin{equation}\label{goodballsflat}\beta_{\infty,\Gamma}(B)\leq \ve\text{ whenever }B\in \G.\end{equation}

We say that $B$ is \emph{very good}, and write $B\in \VG$, if $B(x,s)\in \G$ for every $r\leq s\leq 50R$.

Since $\theta\ll c_0$, we have that any ball $B$ centred on $E$ that contains $B_0$ with $r(B)\leq 50 R$ is very good.  In particular, although $B_0$ is not very good (we have arranged it to be centered at a point of $L_0$, which we cannot guarantee belongs to $E$), we still have that $\beta_{\infty, \Gamma}(20B_0)\leq \ve$, so in particular
$$\dist(x, L_0)\lesssim \ve R\text{ for all }x\in \Gamma\cap 20 B_0.
$$

For $x\in E$ we then set
$$h(x) = \inf\{r:0<r<50R, \,B(x,r)\in \VG\}.
$$
Observe that $h(x)\leq 2R$, as $B(x, 2R)\supset B_0$.

Notice that, if $x\in E$ and $r\in (h(x), 50R)$,  then, from (\ref{goodballsflat}),
$$\Theta_{\mu}(B(x,r))\geq \theta, \beta_{\infty,\Gamma}(B(x,r))\leq \ve, \text{ and }\angle(L_{B(x,r)},L_0)\leq \alpha.
$$
Put
$$Z = E\cap\{h=0\}.
$$

We now set
$$\LD = \{x\in E\backslash Z: \Theta_{\mu}(B(x,h(x)))\leq \theta\},
$$
and
$$\BA = E\backslash (\LD\cup Z),
$$
so that
$$E = Z\cup \LD\cup \BA.
$$

Since, for $x\in \BA$, $\Theta_{\mu}(B(x, h(x))\geq \theta$, we must have that $L_{B(x,h(x))}$ has a big angle with $L_0$, moreover

\begin{lemma}\cite[Lemma 7.13]{Tolsa-llibre}\label{l:biganglestop}     Provided that $\ve$ is sufficiently small in terms of $\delta$ and $\alpha$, if $x\in \BA$, then
$$\angle(L_{B(x, 2h(x))}, L_0)\geq\frac{\alpha}{2},
$$
for any approximating line $L_{B(x,2h(x))}$.
\end{lemma}


We now introduce a regularized version of the function $h$.  Denote by $\Pi$ the orthogonal projection onto $L_0$ and $\Pi^{\perp}$ the orthogonal projection onto the orthogonal complement of $L_0$.

For $x\in \R^2$, set
$$d(x) = \inf_{B(z,r)\in \VG}\bigl[|x-z|+r\bigl].
$$
(Recall here that in order for $B(z,r)\in \VG$ we must have that $z\in E$.)

Now define, for $p\in L_0$,
$$D(p) = \inf_{x\in \Pi^{-1}(p)}d(x). 
$$
As infimums over $1$-Lipschitz funcctions, we see that $d$ and then $D$ are both $1$-Lipschitz functions.

Observe that $d(x)\leq h(x)$ whenever $x\in E$, so the closed set
$$Z_0 = \{x\in \R^2: d(x)=0\}
$$
contains $Z$.

\begin{lemma}\label{PiperpLip}\cite[Lemma 7.19]{Tolsa-llibre}  For all $x,y\in \R^2$, we have
$$|\Pi^{\perp}(x)-\Pi^{\perp}(y)|\leq 6\alpha |\Pi(x)-\Pi(y)|+4d(x)+4d(y).
$$
\end{lemma}

As a consequence of this lemma, we have that if $x,y\in Z_0$, then
$$|\Pi^{\perp}(x) - \Pi^{\perp}(y)|\leq 6\alpha |\Pi(x)-\Pi(y)|.
$$
In particular, the map $\Pi:Z_0\to L_0$ is injective and the function
$$A:\Pi(Z_0)\to \R, \; A(\Pi(x))=\Pi^{\perp}(x) \text{ for }x\in Z_0,
$$
is Lipschitz with norm at most $6\alpha$.  To extend the definition of $A$ to $L_0$, we appeal to a Whitney decomposition.

\subsection{Whitney decomposition}

Let $\DD_{L_0}$ be the collection of dyadic intervals in $L_0$.

For $I \in \DD_{L_0}$,
\begin{align*}
    D(I) := \inf_{p \in I} D(p). \;\;(\text{Here }D(p) = \inf_{x\in \Pi^{-1}(p)}d(x).)
\end{align*}
Set
\begin{align*}
    \Whit := \{ I \mbox{ maximal in } \DD_{L_0} \, :\, \ell(I)  < 20^{-1}\,  D(I)\}.
\end{align*}

 We index $\Whit$ as $\{R_i\}_{i \in I_\Whit}$.  The basic properties of the cubes in $\Whit$ are summarized in the following lemma. The proof of this result is standard, and can be found as Lemma 7.20 in \cite{Tolsa-llibre}.

\begin{lemma}\label{l:whitney-dec}
    The intervals $R_i$, $i\in I_\Whit$, have disjoint interiors in $L_0$ and satisfy the following properties:
    \begin{enumerate}[label=\textup{(}\alph*\textup{)}]
        \item If $x \in 15R_i$, then $5 \ell(R_i) \leq D(x) \leq 50 \ell(R_i)$.
        \item There exists an absolute constant $C>1$ such that if $15R_i \cap 15R_j \neq \varnothing$, then
        \begin{align}
            C^{-1} \ell(R_i) \leq \ell(R_j) \leq C \ell(R_i).
        \end{align}
        \item For each $i \in I_\Whit$, there are at most $N$ intervals $R_j$ such that $15 R_i \cap 15R_j \neq \varnothing$, where $N$ is some absolute constant.
        \item $L_0\setminus \Pi(Z_0) = \bigcup_{i\in I_\Whit} R_i =  \bigcup_{i\in I_\Whit} 15R_i$.
    \end{enumerate}
\end{lemma}



Now set
\begin{align*}
    I_0 := \{ i \in I_\Whit\, :\, R_i \cap B(0, 10R) \neq \varnothing\}.
\end{align*}


\begin{lemma}\label{lem73}\cite[Lemma 7.21]{Tolsa-llibre}
    The following two statements hold.
    \begin{itemize}
        \item If $i \in I_0$, then $\ell(R_i) \leq R$ and $3R_i \subset L_0 \cap 12 B_0 = (-12R,12R)$.
        \item If $i \notin I_0$, then
        \begin{align*}
            \ell(R_i) \approx \dist(0, R_i) \approx |p| \gtrsim R \mbox{ for all } p \in R_i.
        \end{align*}
    \end{itemize}
\end{lemma}


\begin{lemma}\label{l:whitney-prop1}\cite[Lemma 7.22]{Tolsa-llibre}
    Let $i \in I_0$; there exists a ball $B_i \in \VG$ such that
    \begin{align}
        & \ell(R_i) \lesssim r(B_i) \lesssim \ell(R_i),\text{ and} \\
        & \dist(R_i, \Pi(B_i)) \lesssim \ell(R_i).
    \end{align}
\end{lemma}

For $i \in I_0$, denote by $A_i$ the affine function $L_0 \to L_0^\perp$ whose graph is the line $L_{B_i}$. Insofar as $B_i\in \VG$, $\angle(L_{B_i}, L_{B_0})\leq \alpha$, so $A_i$ is Lipschitz with constant $\tan\alpha\lesssim \alpha$.

On the other hand, for $i \in I_\Whit\setminus I_0$, we put $A_i \equiv 0$.  We are now in a position to be able to define $A$ on $L_0$.
\vv




\subsection{Extending $A$ to $L_0$}  Consider a smooth partition of unity $\{\phi_i\}_{i\in I_{\Whit}}$ subordinate to $\{3R_i\}_{i\in I_{\Whit}}$, i.e. $\phi_i\in C^{\infty}_0(3R_i)$ with $\sum_i \phi_i\equiv 1$ on $L_0=\R$, which moreover satisfies that for every $i\in I_{\Whit}$,

\begin{equation}\nonumber
     \|\phi_i' \|_\infty \lesssim \ell(R_i)^{-1} \ \mbox{ and } \|\phi_i''\|_{\infty} \lesssim \ell(R_i)^{-2}.
\end{equation}
(See \cite[p. 250]{Tolsa-llibre} for an explicit construction.)

Now, if $p \in L_0\backslash \Pi(Z_0)$, we set
\begin{align*}
    A(p) := \sum_{i \in I_\Whit} \phi_i(p) A_i(p)= \sum_{i \in I_0} \phi_i(p) A_i(p).
\end{align*}

We require the following lemma, which combines Lemmas 7.24 and 7.27 from \cite{Tolsa-llibre}.
\begin{lemma}\label{l:A}\
   The function $A: L_0 \to L_0^\perp$ is supported in $[-12R, 12R]$ and is Lipschitz with slope $\lesssim \alpha$. Moreover, if $i\in I_{\Whit}$, then for any $x \in 15 R_i$,
   \begin{align*}
       |A''(x)| \lesssim_\theta \frac{\ve}{\ell(R_i)}.
   \end{align*}
\end{lemma}
We will denote the graph of $A$ by $G_A$, that is
\begin{align}
    G_A := \{ (x, A(x)) \, |\, x \in L_0\}.
\end{align}

\subsection{The Lipschitz graph $G_A$ and $E = \supp\mu$ are close each other}\label{closeprops}

The next four results, concerning the relationship between $G_A$ and $E$, are central to our analysis.

\begin{lemma} \label{l:dist-B0}\cite[Lemma 7.28]{Tolsa-llibre}
    Every $x \in B(0, 10R)$ satisfies
    \begin{align*}
        \dist(x, G_A) \lesssim d(x).
    \end{align*}
\end{lemma}

\begin{lemma}\label{l:distQ-L}\cite[Lemma 7.29]{Tolsa-llibre}
    For $B\in \VG$ and $x\in G_A\cap 3B$, it holds that
    \begin{align}\label{e:distQ-L}
        \dist(x,L_B) \lesssim_{\theta} \ve r(B).
    \end{align}
\end{lemma}

\begin{lemma}\label{l:distG-A}\cite[Lemma 7.30]{Tolsa-llibre}
We have    \begin{align*}
        \dist(x, G_A) \lesssim_{\theta} \ve d(x)\text{ for every }x\in E.
    \end{align*}
\end{lemma}

Lemmas \ref{l:distG-A} and \ref{l:distQ-L} combine to yield the following statement.

\begin{coro}\label{l:LBcloseGA} Suppose that $B\in \VG$.  As long as $\ve$ is small enough in terms of $\theta$,
\begin{align}
\dist(x, G_A)\lesssim_{\theta}\ve r(B) \text{ for all }x\in L_B\cap 2B.
\end{align}
\end{coro}

\begin{proof}  If $B=B(z, r)\in \VG$, then $z\in E$ so by Lemma \ref{l:distG-A}, $\dist(z, G_A)\lesssim _{\theta}\ve r$, so $G_A$ (as well as $L_B$) passes close to $z$. On the other hand, Lemma \ref{l:distQ-L} ensures that if $B\in \VG$, then $G_A\cap 3B\subset U_{C(\theta)\ve r(B)}(L_B)$.  Since both $G_A$ and $L_B$ are connected, we readily deduce the conclusion.
\end{proof}

\begin{lemma}\label{l:GdistL0}\cite[Lemma 7.32]{Tolsa-llibre}
We have
$$\dist(x,L_0)\lesssim_{\theta}\,\ve\,R$$
for all $x\in G_A$.
\end{lemma}


\section{Small measure of $\LD$ and $\BA$}

\subsection{$\LD$ has small measure}


The following lemma shows that $\LD$ has small measure.  The reason for this is that $\LD$ can be covered by balls of small density that are closely aligned to the Lipschitz graph $G_A$.

\begin{lemma} \label{l:LD-small}\cite[Lemma 7.33]{Tolsa-llibre}
    If $\theta$ is sufficiently small, and $\ve>0$ is sufficiently small in terms of $\theta$, then
    \begin{align*}
    \mu(\LD) \leq \frac{1}{1000} \mu(B_0).
    \end{align*}
\end{lemma}

This lemma determines our choice of $\theta$.

\subsection{$\BA$ has small measure}
Our main objective in this section is to prove the following.

\begin{lemma}\label{l:BA-small}
If $\alpha$ is chosen sufficiently small, and $\ve$ is chosen sufficiently small, with respect to  $\alpha$ and $\theta$, then
\begin{align}
    \mu(\BA) \lesssim \ve^{1/2} \mu(B_0).
\end{align}
\end{lemma}

Recall our assumption that the line $L_0$ coincides with the horizontal axis of $\R^2$, and so $L_0^\bot$ is the vertical axis.
We denote by $\lca_{\Gamma,\psi}$ and $\A_{\Gamma,\psi}$ the respective square functions $\apsi$ and $\Apsi$ associated with
the open set $\Omega^+_\Gamma\equiv\Omega^+$, whose boundary is $\Gamma$. The analogous square functions associated
with the domain
$$\Omega_{G_A}^+=\{x\in\R^2:\Pi^\bot(x)>A(\Pi(x))\}$$
are denoted by $\lca_{G_A,\psi}$ and $\A_{{G_A},\psi}$.

\begin{lemma}\label{l:Gtopology}For every $B\in \G$, one of the components of $B\backslash U_{2\ve r(B)}(L_B)$ belongs to $\Omega^+$, while the other belongs to $\Omega^-$.
\end{lemma}

\begin{proof} For $B\in \G$ we have that $\beta_{\infty,\Gamma}(B)\leq\ve$. In particular, this implies that
if $L_B$ minimizes $\beta_{\infty,\Gamma}(B)$, then
$$B\setminus U_{2\ve r(B)}(L)\subset \Omega^+\cup \Omega^-.$$
By connectivity, it is clear that each component of $B\setminus U_{2\ve r(B)}(L)$ is contained either in $\Omega^+$ or in $\Omega^-$. Also, since $z_B\in E$, we have that $\int_{r(B)/2}^{r(B)} \lca_{\psi}(z_B,r)^2\,\frac{dr}r\leq \ve$, which easily implies that one of those components must be contained in $\Omega^+$ and the other in $\Omega^-$.\end{proof}

 Applying this lemma to a ball $B'\in \G$ containing $15B_0$, interchanging the upper half plane by the lower half plane if necessary,
 we find a constant $C>0$ such that
\begin{equation}\label{eqconnec1}
15 B_0\cap \R^2_+\setminus U_{C\ve r(B_0)}(L_0)\subset \Omega^+\quad\text{ and } \quad 15 B_0\cap \R^2_-\setminus U_{C\ve r(B_0)}(L_0)\subset \Omega^-.
\end{equation}

To prove Lemma \ref{l:BA-small} we will show that if $\BA$ has noticeable measure, then $\|A'\|_2^2$ is (relatively) large, and that this in turn contradicts the smallness assumption of the smoothed square function in Main Lemma \ref{l:main2}.
 The proof is split into several lemmas.
The first one is Lemma 7.35 of \cite{Tolsa-llibre}, which is a consequence of Lemmas \ref{l:biganglestop} and \ref{l:distG-A}.

\begin{lemma} \label{l:BA-a}\cite[Lemma 7.35]{Tolsa-llibre}
    Provided that $\ve$ is small enough in terms of $\alpha$,
    \begin{align}
        \mu(\BA) \lesssim \alpha^{-2} \|A'\|_2^2.
    \end{align}
\end{lemma}

Assume that $\alpha$ is small enough to apply Lemma \ref{lemlips}.  Then
\begin{equation}\label{eqga534}
\int_{G_A} \A_{G_A,\psi}(x)^2\,d\HH^1(x)\approx \|A'\|_{2}^2,
\end{equation}
where $\A_{G_A,\psi}$ stands for the square function $\A_\psi$ associated with the graph $G_A$.
From this and Lemma \ref{l:BA-a} we infer that
\begin{equation}\label{eqba10}
\mu(\BA) \lesssim  \alpha^{-2} \int_{G_A} \A_{G_A,\psi}(x)^2\,d\HH^1(x).
\end{equation}
Our next objective is to compare
$\int_{G_A} \A_{G_A,\psi}(x)^2\,d\HH^1(x)$ with $\int \Agapsi(x)^2\,d\mu(x)$.

We denote
$$\ell(x):=\frac1{50}D(x)=\frac1{50}D(\Pi(x)).$$
Lemma \ref{l:whitney-dec} ensures that if $x\in 15I$, $I\in\Whit$, then
$$\frac1{10}\ell(I)\leq\ell(x)\leq \ell(I).$$
We set
\begin{align*}
    \wt \A_{G_A, \psi}(x)^2 := \int_{\ell(x)}^{R} \lca_{G_A, \psi}(x,r)^2\, \dr.
\end{align*}

It will be convenient to denote $L^2(G_A)=L^2(\HH^1|_{G_A})$.

\begin{lemma}\label{l:BA-est2}
    We have
    \begin{align*}
        \|\A_{G_A, \psi}- \wt \A_{G_A, \psi}\|_{L^2(G_A)}^2 \lesssim C(\theta)\ve^2 R + \alpha^4 \|A'\|_{2}^2.
    \end{align*}
\end{lemma}

\begin{proof}

To prove the lemma we need to bound the integrals
\begin{align}\label{defi1i2}
  I_1:= \int_{G_A} \int_0^{\ell(x)} \lca_{G_A, \psi}(x,r)^2 \, \dr \, d \hi(x), \quad I_2:=
   \int_{G_A} \int_{R}^{\infty} \lca_{G_A, \psi}(x,r)^2 \, \dr \, d \hi(x).
\end{align}
To do so, we consider the square function $\lca_{G_A, \rho}$, introduced in Section \ref{s:fourier}. We write
\begin{align*}
    I_1 &\lesssim  \int_{G_A} \int_0^{ \ell(x)} \av{\lca_{G_A, \psi}(x,r) - \lca_{G_A, \rho}(x,r) }^2 \, \dr\, d \hi(x) \\
    &\quad
    + \int_{G_A}  \int_0^{ \ell(x)} \lca_{G_A, \rho} (x,r)^2 \, \dr \, d \hi(x) =: I_{1,1} + I_{1,2}.
    \end{align*}
    The first term $I_{1,1}$ can be estimated as in the proof of Lemma \ref{lemdiff1}, to obtain
    \begin{align*}
    I_{1,1} \lesssim \|A'\|^4_{\infty}\|A'\|_{L^2(\R)}^2\lesssim \alpha^4 \|A'\|_{2}^2.
\end{align*}
Let us look at the term $I_{1,2}$. First, recall from Lemma \ref{lem5.4} that
\begin{align} \label{e:BA-d}
   I_{1,2}
    & \approx \int_{\Pi(G_A)} \int_0^{ D(p)/50} \lca_{G_A, \rho}((p, A(p)), r)^2\, \dr\, dp \nonumber\\
    & = \int_{\Pi(G_A)}  \int_0^{ D(p)/50} \av{ \int_{q \in \R} \vp_r(q-p) \ps{\frac{A(q) -A(p)}{r}} \, dq}^2 \, \dr \, dp.
\end{align}
where $\vphi_r(\cdot) = \frac1{r} \vphi\big(\frac\cdot r\big)$.
We write the last integral as
\begin{equation}\label{eqri34}
\sum_{i\in I_\Whit}
\int_{R_i}  \int_0^{ D(p)/50} \av{ \int_{q \in \R} \vp_r(q-p) \ps{\frac{A(q) -A(p)}{r}} \, dq}^2 \, \dr \, dp.
\end{equation}
Observe that for $p\in R_i\in\Whit$, $0<r\leq D(p)/50\leq \ell(R_i)$, and $q\in \supp\vphi_r(\cdot-p)
\subset
\bar B(p,1.1r)$, we have $q\in 4R_i$. Since $\supp A\subset12B_0$ (Lemma \ref{l:A}), we can restrict the sum in \rf{eqri34} to the intervals
$R_i$ such that $4R_i\cap 12B_0\neq \varnothing$. Appealing to Lemma \ref{lem73}, we infer that these cubes are contained in $CB_0$, for some absolute constant $C>1$.

To estimate each of the summands in \rf{eqri34}, let $p\in R_i$ and $q\in \supp\vphi_r(\cdot-p)$.
Taylor's theorem gives, with $\xi_{q,p}$ on the line segment between $q$ and $p$,
\begin{align*}
    A(q) = A(p) + A'(p)(q-p) + \frac{A''(\xi_{q, p})}{2} (q-p)^2.
\end{align*}
Thus we can write the interior most integral in the right hand side of \eqref{e:BA-d} as
$$
     \frac1r\int \vp_r(q-p) A'(p)(p-q) dq + \frac{1}{2r} \int \vp_r(q-p) A''(\xi_{q,p
    }) |p-q|^2 \, d q.
$$
By symmetry we immediately see that the first integral vanishes. Concerning the second integral,
 for $p\in R_i\in\Whit$, $0<r\leq D(p)/50\leq \ell(R_i)$, and $q\in
B(p,1.1r)$ we have
$\xi_{q,p}\in 15R_i$, and then
from Lemma \ref{l:A} we see that
\begin{align*}
    \av{ \frac{1}{2r} \int \vp_r(p-q) A''(\xi_{p,q}) (q-p)^2 \, dq  } \lesssim \frac{1}{r} \,\sup_{\xi \in 15R_i}|A''(\xi)|\,r^2
    \lesssim_\theta
    \frac{r \, \ve}{\ell(R_i)}.
\end{align*}
Using again that $D(p)/50\leq \ell(R_i)$ we deduce that
\begin{align*}
I_{1,2} & \lesssim_\theta
\sum_{R_i\subset C B_0}
\int_{R_i}  \int_0^{ D(p)/50} \av{
\frac{r \, \ve}{\ell(R_i)}
}^2 \, \dr \, dp
\lesssim
\sum_{R_i\subset C B_0} \ve^2\,\ell(R_i)\lesssim \ve^2\,R.
\end{align*}

\vv
Next we have to estimate the integral $I_2$ in \rf{defi1i2}. Given $x\in G_A$ and $r\geq R$,
let $L_{x,r}$ be a line passing through $x$ and parallel to the line minimizing $\beta_{\infty,G_A}(B(x,1.1r))$, and let $H_{x,r}$ be the half plane whose boundary is $L_{x,r}$ lying above $L_{x,r}$.
From the definition of $\lca_{G_A,\psi}(x,r)$, it follows that
\begin{equation}\label{eqas432}
\big|\lca_{G_A,\psi}(x,r)\big| \lesssim \frac1{r^2}|(\Omega_{G_A}^+\Delta H_{x,r})\cap B(x,1.1r)|\lesssim
\beta_{\infty,G_A}(B(x,2r)).
\end{equation}
Taking into account that $\supp A\subset 12B_0$ and that
$\dist(x,L_0)\lesssim_{\theta} \ve R$ for every $x\in G_A$, by Lemma \ref{l:GdistL0}, it
follows easily that
$$\beta_{\infty,G_A}(B(x,2r))
\lesssim \frac{\ve\,R}{\max\big(r,\dist(x,B_0)\big)}\quad\mbox{ for all $x\in G_A$ and $r\geq R$.}
$$
So we deduce that
\begin{align*}
I_2 &\lesssim_{\theta} \int_{G_A} \int_{R}^{\infty} \left(\frac{\ve\,R}{\max\big(r,\dist(x,B_0)\big)}\right)^2 \, \dr \, d \hi(x)\\
& \lesssim_{\theta}
\int_{G_A\cap 2B_0} \int_{R}^{\infty} \big(\ve\,R\big)^2 \, \frac{dr}{r^3} \, d \hi
+\int_{G_A\setminus 2B_0} \int_{R}^{\infty} \frac{\big(\ve\,R\big)^2}{\dist(x,B_0)^{3/2}\,r^{1/2}} \, \dr \, d \hi(x)\\
&\lesssim_{\theta}
\ve^2\,R^2 \int_{R}^{\infty}  \,\frac{dr}{r^3} +
\ve^2\,R^2 \int_{G_A\setminus 2B_0} \frac1{\dist(x,B_0)^{3/2}}  \, d \hi(x)
\int_{R}^{\infty} \frac{dr}{r^{3/2} }\\
& \lesssim_{\theta} \ve^2\,R,
\end{align*}
Gathering the estimates obtained for $I_{1,1}$, $I_{1,2}$, and $I_2$, the lemma follows.
\end{proof}

\vv

Observe that, from \rf{eqga534} and the previous lemma, we have that
\begin{equation}\begin{split}\nonumber\|A'\|_{2}^2&\lesssim \|\wt \A_{G_A, \psi}\|_{L^2(G_A)}^2+   \|\A_{G_A, \psi}- \wt \A_{G_A, \psi}\|_{L^2(G_A)}^2\\&\lesssim \|\wt \A_{G_A, \psi}\|_{L^2(G_A)}^2 +
 C(\theta)\ve^2 R + \alpha^4 \|A'\|_{2}^2.\end{split}\end{equation}
Hence, for $\alpha$ small enough, this gives
$$\|A'\|_{2}^2\lesssim \|\wt \A_{G_A, \psi}\|_{L^2(G_A)}^2 +
 C(\theta)\ve^2 R,$$
and combining this inequality with \rf{eqba10} we obtain
\begin{align} \label{e:BA-100}
    \mu(\BA) \lesssim    \alpha^{-2} \|\wt \A_{G_A, \psi} \|_{L^2(G_A)}^2 + C(\theta)\ve^2 \alpha^{-2}\,R.
\end{align}

To estimate $\|\wt \A_{G_A, \psi} \|_{L^2(G_A)}^2$, we split
\begin{align} \label{e:alpha-G-Gamma}
\|\wt \A_{G_A, \psi} \|_{L^2(G_A)}^2 &= \int_{G_A}\int_{\ell(x)}^{R} \lca_{G_A, \psi}(x,r)^2 \, \dr  \, d \hi(x) \nonumber\\
 &=
  \int_{G_A}\int_{\ell(x)}^{\min(\ve^{-1}\ell(x),R)} \cdots
  +  \int_{G_A}\int_{\min(\ve^{-1}\ell(x),R)}^{R} \cdots.
 \end{align}
Next we estimate each of these integrals separately.

\begin{lemma}\label{l:lemfac64}
We have
\begin{align} \label{e:alpha-G-Gamma2}
\int_{G_A}\int_{\ell(x)}^{\min(\ve^{-1}\ell(x),R)} \lca_{G_A, \psi}(x,r)^2 \, \dr \, d \hi(x) \lesssim_\theta \ve^2 |\log\ve| R.
\end{align}
\end{lemma}

\begin{proof}
From Lemma \ref{l:distQ-L}, it easily follows that $\beta_{\infty,G_A}(B(x,r))\lesssim_\theta \ve$ for all
$x\in G_A$ and $r>\ell(x)$. Then, arguing as in  \rf{eqas432}, we deduce that
$$\lca_{G_A, \psi}(x,r)\lesssim_\theta \ve.$$
Thus, for every $x\in G_A$,
\begin{equation}\label{pontwiseellx}\int_{\ell(x)}^{\ve^{-1}\ell(x)} \lca_{G_A, \psi}(x,r)^2 \, \dr \lesssim_\theta \ve^2 \int_{\ell(x)}^{\ve^{-1}\ell(x)}  \dr \approx \ve^2 |\log\ve|.\end{equation}
On the other hand, if $|x|>CR$ for $C>1$ big enough, then $\ell(x)>R$ and thus
$$\int_{\ell(x)}^{\min(\ve^{-1}\ell(x),R)} \lca_{G_A, \psi}(x,r)^2 \, \dr =0.$$
Consequently, integrating the pointwise estimate (\ref{pontwiseellx}) over $x\in G_A$ with $|x|\leq CR$ yields the lemma.
\end{proof}

To estimate the second integral on the right hand side of \rf{e:alpha-G-Gamma} we need to introduce some additional notation.
We denote by $\Pi_{G_A}$ the projection $\R^2 \to G_A$ orthogonal to $L_0$.
We let $\DD_{G_A}$ be the family of ``dyadic cubes" on $G_A$ of the form
$$\DD_{G_A}= \big\{\Pi_{G_A}(I):I\in\DD_{L_0}\big\}.$$
We define the length of $I\in \DD_{G_A}$ (and write $\ell(I)$) to be equal to $\ell(R)$, where $R\in \DD_{L_0}$ satisfies $I=\Pi_{G_A}(R)$.  Then $\ell(I)$ is comparable to $\mathcal{H}^1(I)$.

Then we set
$$\pStop = \big\{\Pi_{G_A}(I):I\in\Whit\big\}.$$
%

Denote by $\ppStop = \{\Pi_{G_A}(R_i): i\in I_0\}$, so $I\in \ppStop$ if $I = \Pi_{G_A}(R_i)$ for some $R_i$ that intersects $B(0, 10R)$.

\begin{lemma}\label{l:15B0} If $I\in \ppStop$, $x\in I$, and $0<r\leq R$, then $B(x, 2r)\subset 15B_0$.
\end{lemma}

\begin{proof}  By Lemma \ref{l:GdistL0}, $\dist(x, L_0)\ll R$ for $x\in G_A$, so the claimed statement follows from Lemma \ref{lem73}, which states that $I=\Pi_{G_A}(R)$ for an interval $R$ satisfying $3R\subset [-12R, 12R]$.
\end{proof}


We now claim that there is an absolute constant $C_1$ such that for each $I \in \pStop$ that intersects $15B_0$, there exists a ball $B_I\in \VG$ (centered at $z_I\in E$) such that
\begin{align}\label{cond22*}
    & C_1^{-1} r(B_I) \leq \ell(I)  \leq C_1 r(B_I),\\
 \label{cond23*}   & 
  I\subset C_1B_I,\text{ and }\Pi_{G_A}(z_{I})\in C_1B_I.
\end{align}
Indeed, if $I\in \ppStop$, then $I=\Pi(R_i)$ for $i\in I_0$, so by Lemma \ref{l:whitney-prop1}, there exists a ball $B_I\in \VG$ such that  $r(B_I) \approx \ell(I)$ and $\dist(I, \Pi_{G_A}(B_I)) =\dist(\Pi_{G_A}(R_i), \Pi_{G_A}(\Pi(B_I)))\lesssim \ell(I)$.  But then, by Lemma
\ref{l:distG-A}, $$\dist(B_I, G_A)\leq \dist(z_I,G_A)\lesssim_{\theta} \ve d(z_{B_I})\lesssim_{\theta}\ve\,r(B_I),$$
where $z_{I}$ is the center of $B_I$.

Provided that $\ve$ is small enough in terms of $\theta$, we therefore have
$$\dist(I, B_I) \lesssim \dist(I, \Pi_{G_A}(B_I)) + \dist(B_I,\Pi_{G_A}(B_I))+\diam(\Pi_{G_A}(B_I))\lesssim r(B_I)\approx \ell(I).$$
We now can readily deduce that (\ref{cond23*}) holds for $C_1$ large enough (recall that $G_A$ is a Lipschitz graph with Lipschitz constant $\lesssim \alpha\ll 1$).

On the other hand, if $I\in \pStop\backslash \ppStop$ and $I\subset 15B_0$, then $\ell(I)\approx R$, and we can set for $B_I$ a ball centered on $E$ of radius $2R$, say.



We need the following auxiliary result, which appears as Lemma 7.41 in \cite{Tolsa-llibre} in slightly different notation.

\begin{lemma} \label{lemrepart}\cite[Lemma 7.41]{Tolsa-llibre}
For each $I\in \ppStop $ there exists some function $g_I\in L^\infty(\mu)$, $g_I\geq0$ supported
on $B_I$ such that
\begin{equation} \label{co1}
\int g_I\,d\mu = \HH^1(I),
\end{equation}
and
\begin{equation} \label{co2}
\sum_{I\in \ppStop} g_I \lesssim c(\theta).
\end{equation}
\end{lemma}

We will also need the next geometric lemma.

\begin{lemma}\label{lemgeom2} If $\ve$ is small enough, then there exist constants $C_2$ and $C(\theta)>0$ such that
$$(\Omega^+_{G_A}\Delta
\Omega^+_{\Gamma})\cap 15 B_0 \subset \bigcup_{J\in\pStop:J\cap 15 B_0\neq \varnothing}
 B(x_J,C_2\ell(J)) \cap U_{C(\theta)\ve\ell(J)}(G_A).$$
Moreover, for each $J\in\pStop$ such that $J\cap 15 B_0\neq \varnothing$,
$$\big|(\Omega^+_{G_A}\Delta
\Omega^+_{\Gamma})\cap B(x_J,C_2\ell(J))\cap 15B_0\big|\lesssim_{\theta} \ve\,\ell(J)^2.$$
\end{lemma}

\begin{proof}
Let $x\in\Gamma \cap 15 B_0$ and let $J\in\pStop$ be such that $\Pi_{G_A}(x)\in J$. Let $B_J=B(z_J, r(B_J))\in\VG$ satisfy the properties \rf{cond22*} and \rf{cond23*}.  Then both $\Pi_{G_A}(x)$ and $\Pi_{G_A}(z_J)$ belong to $C_1B_J$, where $C_1$ is the constant appearing in \rf{cond22*} and \rf{cond23*}.  Recall here that $\Pi_{G_A}(x)$ is the projection of $x$ onto $G_A$ \emph{orthogonal to} $L_0$.

 We claim that $x\in 10C_1B_J$.  If $x\notin 10C_1B_j$, then the majority of the length of $x-z_J$ is in the component orthogonal to $L_0$.  Therefore, since $x$ and $z_J$ belong to $\Gamma$, from the fact that $\beta_{\infty,\Gamma}(B')\leq \ve$ whenever $B'=B(z_J,s)$ and $r_J\leq s\leq 50R$, we deduce that such a ball $B'$ with $x\in B'\backslash \frac{1}{2}B'$ satisfies
 $\angle(L_{B'},L_0)\gtrsim 1\gg\alpha$. But this cannot happen, and yields the claim.

 But since $10C_1B_J\in \VG$,
$$\dist(x, L_{10C_1B_J})\lesssim \ve r_J\lesssim \ve\ell(J).$$


We therefore infer that
$$\Gamma\cap 15 B_0 \subset \bigcup_{J\in\pStop:J\cap 15 B_0\neq \varnothing}
 B(x_J,10C_1\ell(J)) \cap U_{C\ve\ell(J)}(L_{10C_1B_J}).$$
On the other hand, from Corollary \ref{l:LBcloseGA}, there is a constant $C(\theta)>0$ such that for any $B\in \VG$,
$$L_B\cap B\subset U_{C(\theta)\ve r(B)}(G_A).
$$
Consequently,
$$\Gamma\cap 15 B_0 \subset \bigcup_{J\in\pStop:J\cap 15 B_0\neq \varnothing}
 B(x_J,10C_1\ell(J)) \cap U_{C(\theta)\ve\ell(J)}(G_A).$$
By connectivity arguments, using \rf{eqconnec1}, we deduce that
\begin{equation}\label{eqconn49}
\Omega^\pm_{\Gamma}\cap 15B_0 \subset \Omega_{G_A}^\pm
\cup \bigcup_{J\in\pStop:J\cap 15B_0\neq \varnothing}
 B(x_J,10C_1\ell(J)) \cap U_{C(\theta)\ve\ell(J)}(G_A),
\end{equation}
which implies the first part of the lemma with $C_2=10C_1$.\\

For the second claim of the lemma, set $\wt{B} = B(z_J, C_2\ell(J))$.  If $C_2\ell(J)>R$, then we bound
$$\big|(\Omega^+_{G_A}\Delta
\Omega^+_{\Gamma})\cap \wt B\cap 15B_0\big|\leq \big|(\Omega^+_{G_A}\Delta
\Omega^+_{\Gamma})\cap 15B_0\big|\stackrel{\rf{eqconnec1}}{\lesssim_{\theta}}\ve R^2.
$$
On the other hand, if $C_2\ell(J)<R$, then $\wt{B}\in \VG$, so Lemma \ref{l:distQ-L} and Corollary \ref{l:LBcloseGA} yield that for any $B'\supset \wt{B}$ with $r(B')\leq R$
$$G_A\cap \wt{B}\subset U_{C(\theta)\ve r(B')}(L_{B'}), \text{ and } L_{B'}\cap B'\subset U_{C(\theta)\ve r(B')}(G_A).
$$
Since $\beta_{\infty,\Gamma}(B')\leq \ve$, we infer that both $G_A\cap B'$ and $\Gamma\cap B'$ are both contained in a strip of width $C(\theta)\ve r(B')$ around $L_{B'}$.  Since $\int_0^{R}\lca_{\psi}(z_J,r)^2\frac{dr}{r}<\ve$, one of the components of $B'\backslash U_{C(\theta)\ve r(B')}(L_{B'})$ must belong to $\Omega_{\Gamma}^+$, with the other belonging to $\Omega_{\Gamma}^-$ (Lemma \ref{l:Gtopology}).  But now we infer from \rf{eqconnec1} and a continuity argument that the component of $B'\backslash U_{C(\theta)\ve r(B')}(L_{B'})$ that lies above $L_{B'}$ belongs to $\Omega_{\Gamma}^+$, while the component that lies below $L_{B'}$ belongs to $\Omega_{\Gamma}^-$.  Therefore $\Omega_{\Gamma}^+\Delta\Omega_{G_A}^+\cap B'\subset  U_{C(\theta)\ve r(B')}(L_{B'})$.  In the case $B'=\wt{B}$, we have that $r(\wt{B})\lesssim_{\theta}\ell(J)$, so $|\Omega_{\Gamma}^+\Delta\Omega_{G_A}^+\cap \wt B|\lesssim_{\theta}\ve\ell(J)^2$.
\end{proof}

Now we are ready to deal with the first integral on the right hand of \rf{e:alpha-G-Gamma}:

\begin{lemma}\label{lemllumy}
We have
\begin{align}\label{eqx2}
\int_{G_A}\int_{\min(\ve^{-1}\ell(x),R)}^{R} \lca_{G_A, \psi}(x,r)^2 \, \dr  &\, d \hi(x)
\nonumber
\\ & \lesssim_{\theta} \ve^2R +
\iint_0^{R} \lca_{\Gamma, \psi}(x,r)^2 \, \dr  \, d \mu(x).
\end{align}
\end{lemma}

Throughout the proof of Lemma \ref{lemllumy}, we will let the implicit constant in the symbol $\lesssim$ depend on $\theta$ without further mention.

\begin{proof}
By Lemma \ref{lem73}, $\ve^{-1}\ell(x)\geq R$ for $x\in I\in\pStop\setminus \ppStop$, so we may write
\begin{align}\label{eqajg32}
\int_{G_A}\int_{\min(\ve^{-1}\ell(x),R)}^{R} \lca_{G_A, \psi}&(x,r)^2 \, \dr  \, d \hi(x)\nonumber
\\
&= \sum_{I\in\ppStop} \int_{I}\int_{\min(\ve^{-1}\ell(x),R)}^{R} \lca_{G_A, \psi}(x,r)^2 \, \dr  \, d \hi(x).
\end{align}

Given $x\in I\in\ppStop$, we consider an arbitrary point $x'\in B_{I}$.
Then we write
\begin{equation}\label{eqx1}
  \av{\lca_{G_A, \psi} (x,r) - \lca_{\Gamma, \psi}(x',r)} \leq
     \av{ \lca_{G_A, \psi} (x,r) - \lca_{G_A, \psi}(x',r)}+
      \av{ \lca_{G_A, \psi}(x',r) - \lca_{\Gamma, \psi}(x',r)}.
\end{equation}
 Regarding the first term on the right hand side, using the fact that $r\geq \ve^{-1}\ell(x)\approx \ve^{-1}\ell(I)\gg\ell(I)$ and
taking into account that $|x-x'|\leq \dist(I,B_{I})\lesssim\ell(I)$ by \rf{cond23*}, we get
 \begin{align*}
 \av{ \lca_{G_A, \psi} (x,r) - \lca_{G_A, \psi}(x',r)}   &\leq r^{-2}\int_{\Omega^+_{G_A}}
 \bigg| \psi\ps{\frac{x-y}{r}} - \psi\ps{\frac{x'-y}{r}}\bigg|\,dy\\
 & \lesssim \|\nabla\psi\|_\infty \int_{B(x,2r)}\frac{|x-x'|}{r^3}\,dy \lesssim \frac{\ell(I)}r,
\end{align*}

Next we deal with the last term in \rf{eqx1}:
\begin{align*}
\av{ \lca_{G_A, \psi}(x',r) - \lca_{\Gamma, \psi}(x',r)}
& = r^{-2}\av{ \int_{\Omega^+_{G_A}} \psi\ps{\frac{x'-y}{r}} \, dy -\int_{\Omega^+_{\Gamma}} \psi\ps{\frac{x'-y}{r}} \, dy}\\
& \lesssim r^{-2}\,  \big|(\Omega^+_{G_A}\Delta
\Omega^+_{\Gamma})\cap B(x,2r)\big|.
\end{align*}
Notice next that, insofar as $I\in \Whit_{G_A}$, there is an absolute constant $C>0$ such that if $x_I$ is the center of $I$, and if $J\in \Whit_{G_A}$ satisfies $J\cap B(x,2r)\neq \varnothing$, then $J\subset B(x_I, Cr)$. Also by Lemma \ref{l:15B0}, $B(x,2r)\subset 15B_0$.  But then, Lemma \ref{lemgeom2} ensures that,
\begin{align*}
\big|(\Omega^+_{G_A}\Delta
\Omega^+_{\Gamma})\cap B(x,2r)\big|& \leq \sum_{J\in\pStop:J\cap B(x,2r)\neq \varnothing}
\big|(\Omega^+_{G_A}\Delta
\Omega^+_{\Gamma})\cap B(x_J,C_2\ell(J))\big| \\
&\lesssim \sum_{J\in \pStop: J\cap B(x,2r)\neq \varnothing}\ve\ell(J)^2\lesssim \sum_{J\in\pStop:J\subset  B(x_I,Cr)} \ve \,\ell(J)^2.
\end{align*}
From the last estimates we derive
$$\av{ \lca_{G_A, \psi} (x,r) - \lca_{\Gamma, \psi}(x',r)}^2
\lesssim \frac{\ell(I)^2}{r^2} + \ve^2 \bigg(\sum_{J\in\pStop:J\subset  B(x_I,Cr)} \frac{\ell(J)^2}{r^2}
\bigg)^2.$$
But since
$$\sum_{J\in\pStop:J\subset  B(x_I,Cr)} \ell(J)^2\lesssim r\sum_{J\in\pStop:J\subset  B(x_I,Cr)} \ell(J)\lesssim_\theta r^2,$$
we deduce that
$$\av{ \lca_{G_A, \psi} (x,r) - \lca_{\Gamma, \psi}(x',r)}^2
\lesssim_\theta \frac{\ell(I)^2}{r^2} + \ve^2 \sum_{J\in\pStop:J\subset  B(x_I,Cr)} \frac{\ell(J)^2}{r^2}.$$
Since this holds for all $x'\in B_I$,
$$ \lca_{G_A, \psi} (x,r)^2 \lesssim_\theta  \inf_{x'\in  B_I}
\lca_{\Gamma, \psi}(x',r)^2
+\frac{\ell(I)^2}{r^2} + \ve^2 \sum_{J\in\pStop:J\subset  B(x_I,Cr)} \frac{\ell(J)^2}{r^2}$$
for all $x\in I\in\ppStop$.

Plugging this inequality into the right hand side of \rf{eqajg32}, we estimate the integral on the left side of \rf{eqx2} as follows:
\begin{align*}
\int_{G_A}\int_{\min(\ve^{-1}\ell(x),R)}^{R}& \lca_{G_A, \psi}(x,r)^2 \, \dr \,d \hi(x)
\\
& \lesssim_\theta \sum_{I\in \ppStop}\int_I\int_{\min(c\ve^{-1}\ell(I),R)}^{R} \inf_{x'\in  B_I}
\lca_{\Gamma, \psi}(x',r)^2\, \dr \,d \hi(x)\\
&\quad+ \sum_{I\in \ppStop}\int_I\int_{\min(c\ve^{-1}\ell(I),R)}^{R} \frac{\ell(I)^2}{r^2} \, \dr \,d \hi(x)\\
&\quad+
\sum_{I\in \ppStop}\int_I\int_{\min(c\ve^{-1}\ell(I),R)}^{R} \ve^2 \sum_{\substack{J\in\pStop:\\J\subset  B(x_I,Cr)}} \frac{\ell(J)^2}{r^2} \, \dr \,d \hi(x)\\
&=: T_1+ T_2+ T_3.
\end{align*}
First we bound $T_2$ in a straightforward manner by evaluating the double integral:
$$T_2\lesssim
\sum_{I\in \ppStop}\ell(I) \int_{c\ve^{-1}\ell(I)}^\infty \frac{\ell(I)^2}{r^3} \, dr \lesssim
\ve^2 \sum_{I\in \ppStop}\ell(I) \lesssim \ve^2\,R.
$$
We now turn out attention to $T_3$. Given $I,J\in \pStop$ denote
$$D(I,J)= \ell(I)+\ell(J)+\dist(I,J).$$
Notice that if $J\subset  B(x_I,Cr)$, $r>c\ve^{-1}\ell(I)$, then $r\gtrsim D(I,J)$.  Using Fubini's theorem, we therefore infer that
\begin{align*}
T_3&\lesssim \sum_{I\in \ppStop}\ell(I) \int_{\min(c\ve^{-1}\ell(I),R)}^{R} \ve^2 \sum_{\substack{J\in\pStop:\\J\subset  B(x_I,Cr)}} \frac{\ell(J)^2}{r^3} \, dr \\
& \lesssim \ve^2 \sum_{J\in \pStop:J\subset CB_0}\ell(J)^2\sum_{I\in \pStop}\ell(I)
\int_{r>c \,D(I,J)}\frac{dr}{r^3}\\
& \lesssim \ve^2 \sum_{J\in \pStop:J\subset CB_0}\ell(J)^2\sum_{I\in \pStop}\frac{\ell(I)}{D(I,J)^2}.
\end{align*}
Now notice that if $D(I,J)=t$, then $I\subset B(x_J, Ct)$.  Consequently, as $D(I,J)\geq \ell(J)$, we control the inner sum on the right hand side as follows:\begin{align*}
\sum_{I\in \pStop}\frac{\ell(I)}{D(I,J)^2} &=
\sum_{k\geq 0} \sum_{\substack{I\in \pStop:\\2^k\ell(J)\leq D(I,J)\leq 2^{k+1}\ell(J)}} \frac{\ell(I)}{D(I,J)^2}\\
&\lesssim \sum_{k\geq 0} \frac1{2^{2k}\ell(J)^2}
\sum_{\substack{I\in \pStop:\\ I\subset B(x_J,C2^k\ell(J))}} \ell(I) \lesssim
\sum_{k\geq 0} \frac1{2^{k}\ell(J)} \approx \frac1{\ell(J)}.
\end{align*}
Combining these two chains of inequalities, we arrive at $$T_3\lesssim \ve^2 \sum_{J\in \pStop:J\subset CB_0}\ell(J)\lesssim \ve^2 R.$$

Finally we will estimate the term $T_1$. To this end, we consider the functions $g_I$ constructed in
Lemma \ref{lemrepart}. It is clear that
$$T_1 =\sum_{I\in \ppStop}\iint_{\min(c\ve^{-1}\ell(I),R)}^{R} \inf_{x'\in  B_I}
\lca_{\Gamma, \psi}(x',r)^2\, \dr \,g_I(x)\,d \mu(x).$$
Observe now that, for each $x\in B_I$,
$$\inf_{x'\in  B_I}
\lca_{\Gamma, \psi}(x',r) \leq \avint_{B_I} \lca_{\Gamma, \psi}(x',r)\,d\mu(x')
\leq \frac{\mu(3B_I)}{\mu(B_I)}\,
\wt M_\mu \lca_{\Gamma, \psi}(\cdot,r)(x),
$$
where $\wt M_\mu$ is the maximal operator defined by $$\wt M_\mu f(x) = \sup_{B\ni x} \frac1{\mu(3B)}\int_{B}|f|\,d\mu.$$
Since $\frac{\mu(3B_I)}{\mu(B_I)}\lesssim \theta^{-1}\lesssim 1$,  using Fubini and Lemma \ref{lemrepart},
we can write
\begin{align*}
T_1 & \lesssim \sum_{I\in \ppStop}\iint_{\min(c\ve^{-1}\ell(I),R)}^{R}
\wt M_\mu \lca_{\Gamma, \psi}(\cdot,r)(x)^2
 \, \dr \,g_I(x)\,d \mu(x)\\
& \lesssim
 \int_0^{R}\!\!\int
\wt M_\mu \lca_{\Gamma, \psi}(\cdot,r)(x)^2 \,\sum_{I\in \ppStop} g_I(x)\,d\mu(x)\, \dr\\
&\lesssim  \int_0^{R}\!\!\int
\wt M_\mu \lca_{\Gamma, \psi}(\cdot,r)(x)^2 \,d\mu(x)\, \dr
\end{align*}
Using that $\wt M_\mu$ is bounded in $L^2(\mu)$ (see Theorem 9.32 in \cite{Tolsa-llibre}, for example),  we derive
$$T_1\lesssim \int_0^{R}\!\!\int
 \lca_{\Gamma, \psi}(x,r)^2 \,d\mu(x)\, \dr.$$
Gathering the estimates obtained for the terms $T_1$, $T_2$, and $T_3$, the lemma follows.
\end{proof}

\begin{proof}[Proof of Lemma \ref{l:BA-small}]
By \rf{e:BA-100} and Lemmas \ref{l:lemfac64}, \ref{lemllumy}, we get
\begin{align*}
  & \mu(\BA) \lesssim_{\theta}  \alpha^{-2}\ve^2 \,R +  \alpha^{-2} \|\wt \A_{G_A, \psi} \|_{L^2(G_A)}^2  \\
   &\lesssim_{\theta}
\alpha^{-2}\ve^2 \,R + \alpha^{-2}\left(
\ve^2 |\log\ve| R+
\ve R +
\iint_0^{R} \lca_{\Gamma, \psi}(x,r)^2 \, \dr  \, d \mu(x)\right)\\
& \lesssim_{\theta}
\alpha^{-2}\ve^2 |\log\ve|\,R + C(\theta,\alpha)
\iint_0^{R} \lca_{\Gamma, \psi}(x,r)^2 \, \dr  \, d \mu(x)\leq \ve^{1/2}\mu(B_0),
\end{align*}
for $\ve= \ve(\alpha,\theta)$ small enough. This yields the desired conclusion.
\end{proof}

\subsection{Proof of the Main Lemma \ref{l:main2}}

By Lemmas \ref{l:LD-small} and \ref{l:BA-small}, if $\theta$ is chosen small enough and then $\ve$ also
small enough (depending on $\alpha$ and $\theta$), then
$$\mu(\BA \cup \LD)\leq \frac12\,\mu(B_0).$$
But then
$$\mu(Z)\geq \frac{1}{2}\mu(B_0),
$$
and $Z\subset Z_0\subset G_A$.  This completes the proof.


\end{document}